\documentclass[11pt,reqno]{amsart}

\usepackage[OT2,OT1]{fontenc}
\newcommand\cyr
{
\renewcommand\rmdefault{wncyr}
\renewcommand\sfdefault{wncyss}
\renewcommand\encodingdefault{OT2}
\normalfont
\selectfont
}
\DeclareTextFontCommand{\textcyr}{\cyr}

\usepackage{chngcntr}
\usepackage{srcltx}
\usepackage{mathrsfs,amsthm, amsmath, amssymb, cases, color, nicefrac}
\usepackage{hyperref}
    \hypersetup{linkcolor=blue, colorlinks=true,citecolor = red}
\usepackage[capitalise]{cleveref}
\usepackage{upref}
\usepackage{a4wide}
\usepackage{amsfonts}
\usepackage[toc,page]{appendix}
\usepackage{bbm}

\usepackage{amsmath, amssymb,amscd}
\usepackage{amsfonts}
\usepackage{mathrsfs}
\usepackage{graphicx}

 \usepackage{upref}
\usepackage{a4wide}
\hypersetup{linkcolor=blue, colorlinks=true,citecolor = red}
\allowdisplaybreaks

\newtheorem{thm}{Theorem}[section]
\newtheorem{prop}[thm]{Proposition}
\newtheorem{lem}[thm]{Lemma}
\newtheorem{cor}[thm]{Corollary}
\newtheorem{remark}[thm]{Remark}
\newtheorem{defin}[thm]{Definition}

\theoremstyle{definition}
\newtheorem{definition}[thm]{Definition}

\theoremstyle{remark}
\numberwithin{equation}{section}

\newcommand{\ost}{\Omega_{S}(t)}
\newcommand{\oso}{\Omega_{S}(0)}
\newcommand{\oft}{\Omega_{F}(t)}
\newcommand{\ofo}{\Omega_{F}(0)}
\newcommand{\rt}{\mathbb{R}^{3}}
\newcommand{\ct}{\mathbb{C}^{3}}
\newcommand{\mr}{\mathcal{R}}

\newenvironment{PDE} {\left\{ \begin{array}{ll} }  {\end{array} \right.}

  \newcommand {\EE}{\mathbb E}

  \newcommand {\RR}{\mathbb R}

  \newcommand {\norm}[1] {\| #1 \|}  

  \newcommand {\bignorm}[1]{\bigl\| #1 \bigr\|}
  \newcommand {\Bignorm}[1]{\Bigl\| #1 \Bigr\|}

  \newcommand {\SUCHTHAT}{:\;}
  
  \renewcommand {\Re}{\operatorname{Re}}
  
  \newcommand{\ds}{\displaystyle}
  \newcommand{\mx}{\mathcal{X}}
  \newcommand{\mz}{\mathcal{Z}}
   \newcommand{\my}{\mathcal{Y}}
    \newcommand{\dom}{\mathcal{D}}
    \newcommand{\cof}{\mathrm{Cof}}

\renewcommand{\leq}{\leqslant}

\renewcommand{\geq}{\geqslant}
\renewcommand{\ge}{\geqslant}

\DeclareMathOperator{\Cof}{Cof}
\renewcommand{\div}{\operatorname{div}}

\begin{document}

\title[]{Mathematical Analysis of the Motion of a Rigid Body in a  Compressible Navier-Stokes-Fourier Fluid. }

 \date{\today}

\author{Bernhard H. Haak}
\address{Institut de Math\'ematiques, Universit\'e de Bordeaux, Bordeaux INP, CNRS
F-33400 Talence, France}
\email{bernhard.haak@math.u-bordeaux.fr}

\author{Debayan Maity}
\address{Institut de Math\'ematiques, Universit\'e de Bordeaux, Bordeaux INP, CNRS
F-33400 Talence, France}
\email{debayan.maity@u-bordeaux.fr}

 \author{Tak\'eo Takahashi}
\address{INRIA Nancy - Grand Est, 615, rue du
  Jardin Botanique. 54600 Villers-l\`es-Nancy. France}
 \email{takeo.takahashi@inria.fr}
\author{Marius Tucsnak}
\address{Institut de Math\'ematiques, Universit\'e de Bordeaux, Bordeaux INP, CNRS
F-33400 Talence, France}
\email{marius.tucsnak@u-bordeaux.fr}

 \begin{abstract}
We study an initial and boundary value problem modelling the motion of a rigid body in a heat conducting gas. The solid is supposed to be a perfect thermal insulator. The gas is described by the compressible Navier-Stokes-Fourier equations, whereas the motion of the solid is governed by Newton's laws. The main results assert the existence of strong solutions, in an $L^p$-$L^q$ setting, both locally in time and globally in time for small data. The proof is essentially using the maximal regularity property of associated linear systems. This property is checked by proving the $\mr$-sectoriality of the corresponding operators, which in turn is obtained by a perturbation method.
 \end{abstract}

\maketitle



\tableofcontents


\part{Introduction and Statement of the Main Results}

\section{Introduction}\label{sec_intro}
The purpose of this work is to provide existence and uniqueness results for a coupled PDEs-ODEs system which models the motion of a rigid body in a viscous heat conducting  gas. The rigid body is assumed to be a perfect insulator. As far as we know, this system has not been studied in the literature in three space dimensions. A related problem in one space dimension, the so-called {\em adiabatic piston problem} has been studied in \cite{FMNT}.

Let us now mention some related works from the literature. The one-dimensional piston problem with homogenous boundary conditions has been studied by Shelukhin \cite{Sh_unique,Sh_temp}. Maity, Takahashi and Tucsnak \cite{MTT17} proved existence and uniqueness of global in time strong solutions with nonhomogeneous boundary conditions in a Hilbert space setting.  Local in time existence and uniqueness of a heat conducting piston in $L^{p}$-$L^{q}$ framework is studied by Maity and Tucsnak \cite{MT17}. Concerning three-dimensional models,  global existence of  weak solutions for compressible fluid and rigid body interaction problems was studied  by Desjardins and Esteban \cite{DE00} and Feireisl \cite{FE03}.  Boulakia and Guerrero \cite{BoulGuer}  proved global  existence and uniqueness of strong solutions for small initial data within the Hilbert space framework. Hieber and Murata \cite{HiebMur} proved local in time existence and uniqueness in a $L^{p}$-$L^{q}$ setting. Let us also mention that an important influence on the methods in this work comes from several recent advances on the $L^{p}$-$L^{q}$ theory of viscous compressible fluids (without structure), see Enomoto and Shibata \cite{ShibaEno13} and Murata and Shibata \cite{ShibataMurata16}.

In this work we are interested in strong solutions and the main novelties we bring in are:
\begin{itemize}
\item The full nonlinear free boundary system coupling the compressible Navier-Stokes-Fourier system with the ODE system for the solid has not, at our knowledge, been studied in the literature.
\item The existence and uniqueness results are proved in a $L^{p}$-$L^{q}$ setting, which, at least as global existence is concerned,  is new even in the case when the fluid is barotropic.
\end{itemize}

The methodologies we employ for the local in time, versus the global in time (for small data), existence results are quite different. More precisely, in the proof of the local existence theorem, we begin by considering a linear ``cascade'' system.  The corresponding operator is proved to have the maximal regularity property in appropriate spaces by combining  various existing maximal $L^{p}$-$L^{q}$ regularity results for parabolic equations. This allows us to develop a quite simple fixed point procedure to obtain the local in time existence and uniqueness of solutions. More precisely, using this associated linear system we estimate the nonlinear terms with a coefficient involving the length of the considered time interval,
see \cref{prop:estimate} below. This is why this method is suitable for local existence results (but not relevant if we are interested in global existence for small data).

The strategy developed in proving global existence and uniqueness for small initial data is more involved.
More precisely, in this case it is essential to  linearize  around a stationary solution and to prove that the corresponding  linear system is exponentially stable. To prove this property we use a ``monolithic'' approach, which means  that the linear system preserves the coupling between fluid and structure. However, in order to obtain the maximal regularity property we repeatedly use a perturbation argument. Roughly speaking, this means we deduce the maximal regularity property for the coupled system from the corresponding properties of the linearized fluid equations with homogeneous boundary conditions.

The plan of the paper is as follows. In the next section, we introduce the governing equations and we state our main results. \cref{sec:notation} is devoted to notation, which introduces, in particular, several function spaces playing an important role in the remaining part of this work.

Part 2 is devoted to the proof of a local in time existence result. More precisely,  in \cref{sec:cov-local} we rewrite the governing equations  in Lagrangian coordinates,  in \cref{sec:linearized-problem-loc}  we prove the maximal $L^{p}$-$L^{q}$ regularity for an associate  ``cascade'' type linear system, whereas in \cref{sec:nonlin-terms-loc}, we derive the required estimates and Lipschitz properties of the nonlinear terms in order to apply a fixed-point procedure.

In Part 3 we prove global existence and uniqueness of solutions for small initial data. This is divided into several sections. In \cref{sec:cov-g}, we linearize the system around a constant steady state and we rewrite the system in the reference configuration.  In \cref{sec_back} we recall some results concerning maximal $L^{p}$ regularity for abstract Cauchy problems and its connections with the $\mr$-sectoriality property. In \cref{sec:lin-FSI} and  in \cref{sec:max-lin-g} we prove the maximal $L^{p}$-$L^{q}$ regularity of a linearized coupled fluid-structure interaction problem on  time interval $[0,\infty).$ We prove Lipschitz properties of the nonlinear terms  in \cref{sec:nl-est-g} and finally in \cref{sec:global-existence} we prove the global existence theorem.

\section{Statement of the Main Results} \label{main-res}

We consider a rigid structure immersed in a viscous heat conducting gas and we denote by $\ost$ the domain occupied by the solid at time $t \geq 0.$ We assume that the fluid and rigid  are contained in a smooth bounded domain
$\Omega \subset \mathbb{R}^{3}$. Moreover, we suppose that $\oso$ has a smooth boundary and that
\begin{align} \label{nocontact}
\mathrm{dist}(\Omega_{S}(0),\partial\Omega) \geqslant \nu >  0.
\end{align}
For any time $t \geq 0$,
$\Omega_{F}(t) = \Omega \setminus \overline{\Omega_{S}(t)}$ denotes
the region occupied by the fluid. The motion of the fluid is given by
\begin{equation} \label{fluid-eq}
\begin{cases}
\partial_{t} \rho + \operatorname{div} (\rho u) = 0 & \mbox{ in } (0,T) \times \oft, \\
\rho( \partial_{t} u + ( u \cdot \nabla) u )  - \operatorname{div} \sigma(u,p) = 0 & \mbox{ in } (0,T) \times \oft, \\
c_{v} \rho \left(\partial_{t} \vartheta + u \cdot \nabla \vartheta \right) +  p \operatorname{div} \ u - \kappa \Delta \vartheta
    =  \alpha ( \operatorname{div} \ u)^{2} + 2\mu Du:Du & \mbox{ in } (0,T) \times \oft,
\end{cases}
\end{equation}
where
\begin{gather}
\sigma(u,p) = 2\mu Du  + (\alpha \operatorname{div} \ u - p) I_3, \notag \\
Du  = \frac{1}{2} (\nabla u + \nabla u^{\top}), \notag\\
\mu > 0 \mbox{ and  } \alpha + \frac{2}{3} \mu > 0, \label{tak1.3}\\
A : B = \sum_{i,j}  a_{ij} b_{ij} \qquad \mbox{is the canonical scalar product of two $n\times n$ matrices},\notag \\
p = R \rho \vartheta, \qquad R \mbox{ is the universal gas constant.} \notag
\end{gather}
Note that we denote by $M^\top$ the transpose of a matrix $M$.

At time $t \geq 0$, let $a(t) \in \mathbb{R}^{3}$,
$Q(t) \in SO_{3}(\mathbb{R})$ and $\omega(t) \in \mathbb{R}^{3}$
denote the position of the center of mass, the orthogonal matrix
giving the orientation of the solid and the angular velocity of the
rigid body.  Therefore we have,
\begin{align*}
\dot Q(t)Q(t)^{-1} y = A(\omega(t)) y =  \omega(t) \times y, \quad \forall y \in \mathbb{R}^{3},
\end{align*}
where the skew-symmetric matrix $A(\omega)$ is given by
\begin{align*}
A(\omega) = \begin{pmatrix}
0 & -\omega_{3} & \omega_{2} \\ \omega_{3} & 0 & -\omega_{1} \\ -\omega_{2} & \omega_{1} & 0
\end{pmatrix}, \qquad  \omega \in \mathbb{R}^{3}
\end{align*}
and where $\dot f$ denotes the time derivative of $f$.

Without loss of generality we can assume that
\begin{align}
a(0) = 0 \quad \mbox{ and } \quad Q(0) = I_{3}.
\end{align}
Thus the domain occupied by the structure $\ost$ is given by 
\begin{align} \label{def:ost}
\ost = a(t) + Q(t) y, \quad t \geqslant 0, y \in \oso.
\end{align}
We denote by $m > 0$ the mass of rigid structure and
$J(t) \in \mathcal{M}_{3 \times 3} (\mathbb{R})$ its tensor of inertia
at time $t$. The equations of the structures are given by
\begin{equation} \label{eq:solid}
\begin{cases}
\displaystyle m\frac{d^{2}}{dt^{2}} a = - \int_{\partial \ost} \sigma(u,p) n \ d\gamma & \text{in} \ (0,T),  \\
\displaystyle  J \frac{d}{dt} \omega = (J\omega) \times \omega - \int_{\partial \ost} (x-a(t)) \times \sigma(u,p) n \ d\gamma & \text{in} \ (0,T),
\end{cases}
\end{equation}
where $n(t,x)$ is the unit normal to $\partial \ost$ at the point $x$
directed toward the interior of the rigid body.  We assume that the fluid velocity
satisfies the no-slip boundary conditions:
\begin{align}
&u(t,x) = 0, \quad x \in \partial\Omega, \notag \\
&u(t,x) = \dot a(t) + \omega(t) \times (x -a(t)) \qquad (x \in \partial\ost).
\end{align}
We also suppose that the structure is thermally insulating:
\begin{equation}\label{tak0.0}
\frac{\partial \vartheta}{\partial n }(t,x) = 0  \qquad\qquad (t\in (0,T), \ x \in \partial\Omega_{F}(t)).
\end{equation}
The above system is completed by the following initial conditions
\begin{align} \label{inicond}
&\rho(0, \cdot) = \rho_{0}, \quad u(0, \cdot) = u_{0}, \quad \vartheta(0, \cdot) = \vartheta_{0} \qquad (\mbox{in } \ofo), \notag \\
& a(0) = 0, \quad \dot a(0) = \ell_{0}, \quad Q(0) = I_{3}, \quad \omega(0) = \omega_{0}.
\end{align}

To state our main results we introduce some notation. Firstly $W^{s,q}(\Omega)$, with $s\geqslant 0$ and $q>1$, denote the usual Sobolev spaces.  Let $k \in \mathbb{N}$.
For every $0 < s < k, $ $1\leqslant p < \infty$, $1\leqslant q < \infty$, we define the  Besov spaces by real interpolation of Sobolev spaces
\begin{align*}
B^{s}_{q,p}(\Omega) = (L^{q}(\Omega), W^{k,q}(\Omega))_{s/k,p}\ .
\end{align*}
We refer to \cite{Adams} and \cite{Triebel} for a detailed presentation of Besov spaces.
We also need some notation specific to our problem.

Let $2< p < \infty$ and $3 < q< \infty$  such that $\displaystyle \frac{1}{p} + \frac{1}{2q} \neq \frac12$. We set
\begin{multline}
\mathcal{I}_{p,q} = \Big\{ (\rho_{0}, u_{0}, \vartheta_{0}, \ell_{0},\omega_{0}) \mid \rho_{0} \in W^{1,q}(\ofo), \
u_{0} \in B^{2(1-1/p)}_{q,p}(\ofo)^{3},
\\
\vartheta_{0} \in B^{2(1-1/p)}_{q,p}(\ofo),
 \ell_{0} \in \mathbb{R}^{3},\  \omega_{0} \in \mathbb{R}^{3},\  \min_{\overline{\ofo}} \rho_{0} > 0
 \\
 u_{0} = 0 \mbox{ on } \partial \Omega, \
\quad u_{0}(y) = \ell_{0} + \omega_{0} \times y \quad y\in \partial\oso  \} \Big\}.
\end{multline}
Such a definition has a sense since from the Sobolev embedding we have
$$
W^{1,q}(\ofo) \subset C(\overline{\ofo}).
$$
Moreover since $\ds \frac{1}{p} + \frac{1}{2q} < 1$,
if $f\in B^{2(1-1/p)}_{q,p}(\ofo)$, then (see, for instance, \cite[p.200]{Triebel}), $f$ admits a trace on $\partial \ofo$
with
$f_{|\partial \ofo}\in B^{2(1-1/p)-1/q}_{q,p}(\partial\ofo).$

The norm of $\mathcal{I}_{p,q}$ is the norm of
$$
W^{1,q}(\ofo)\times B^{2(1-1/p)}_{q,p}(\ofo)^{3} \times  B^{2(1-1/p)}_{q,p}(\ofo)\times \mathbb{R}^{6}.
$$
We introduce the set of initial data
\begin{align} \label{inispace}
\mathcal{I}^{cc}_{p,q} =
\begin{cases}
\mathcal{I}_{p,q} & \mbox{ if } \displaystyle \frac{1}{2} < \frac{1}{p} + \frac{1}{2q} < 1, \\
\left\{(\rho_{0}, u_{0}, \vartheta_{0}, \ell_{0},\omega_{0}) \in \mathcal{I}_{p,q}
             \mid \displaystyle \frac{\partial \vartheta_{0}}{\partial n}  = 0, \mbox{ on } \partial\ofo  \right\} & \mbox{ if }  \displaystyle \frac{1}{p} + \frac{1}{2q} < \frac{1}{2}.
\end{cases}
\end{align}
Again, the normal derivative in the above definition is well-defined due to the trace theorem for Besov spaces (see, for instance \cite[p.200]{Triebel}).

We also need a definition of Sobolev spaces in the time dependent domain $\Omega_{F}(t)$. Let $\Lambda(t,\cdot)$ be a $C^{1}$-diffeomorphism from $\ofo$ onto $\oft$ such that all the derivatives up to second order in space variable and all the derivatives up to first order in time variable exist. For all functions $v(t,\cdot) : \oft \mapsto \mathbb{R},$ we denote $\widehat v(t,y) = v(t, \Lambda(t,y))$ Then for any $1 < p,q < \infty$  we define
\begin{align*}
&L^{p}(0,T;L^{q}(\Omega_{F}(\cdot))) = \left\{ v  \mid \widehat v \in L^{p}(0,T;L^{q}(\ofo))\right\}, \\
&L^{p}(0,T;W^{2,q}(\Omega_{F}(\cdot))) = \left\{ v  \mid \widehat v \in L^{p}(0,T;W^{2,q}(\ofo))\right\}, \\
&W^{1,p}(0,T;L^{q}(\Omega_{F}(\cdot))) = \left\{ v  \mid \widehat v \in W^{1,p}(0,T;L^{q}(\ofo))\right\}, \\
&C([0,T];W^{1,q}(\Omega_{F}(\cdot))) = \left\{ v  \mid \widehat v \in C([0,T];W^{1,q}(\ofo))\right\}, \\
&C([0,T];B^{2(1-1/p)}_{q,p}(\Omega_{F}(\cdot))) = \left\{ v  \mid \widehat v \in C([0,T];B^{2(1-1/p)}_{q,p}(\ofo)\right\}.
\end{align*}

We are now in a position to state  our first main result.
\begin{thm} \label{mainthm}
Let $2 < p < \infty$ and $3 < q < \infty$ satisfying the condition $\displaystyle \frac{1}{p} + \frac{1}{2q} \neq \frac12$.  Assume that \eqref{nocontact} is satisfied and  $(\rho_{0},u_{0}, \vartheta_{0},\ell_{0},\omega_{0})$ belongs to $\mathcal{I}^{cc}_{p,q}$. Let $M> 0$ be such that
\begin{align} \label{ini-ball}
\|(\rho_{0}, u_{0}, \vartheta_{0}, \ell_{0},\omega_{0})\|_{\mathcal{I}_{p,q}} \leqslant M, \quad \frac{1}{M} \leqslant \rho_{0}(x) \leqslant M \mbox{ for } x \in \ofo.
\end{align}
Then, there exists $T > 0$, depending only on $M$ and $\nu$ such that the system \eqref{fluid-eq} - \eqref{inicond} admits a unique strong solution
\begin{gather*}
\rho \in W^{1,p}(0,T; W^{1,q}(\Omega_{F}(\cdot))) \cap C([0,T];W^{1,q}(\Omega_{F}(\cdot))), \\
u \in L^{p}(0,T; W^{2,q}(\Omega_{F}(\cdot))^{3}) \cap W^{1,p}(0,T; L^{q}(\Omega_{F}(\cdot))^{3}) \cap C([0,T]; B^{2(1-1/p)}_{q,p}(\Omega_{F}(\cdot))^{3}), \\
\vartheta \in L^{p}(0,T; W^{2,q}(\Omega_{F}(\cdot))) \cap W^{1,p}(0,T; L^{q}(\Omega_{F}(\cdot))) \cap C([0,T]; B^{2(1-1/p)}_{q,p}(\Omega_{F}(\cdot))), \\
a \in W^{2,p}(0,T;\mathbb{R}^{3}), \quad \omega \in W^{1,p}(0,T;\mathbb{R}^{3}).
\end{gather*}
Moreover, there exists a constant $M_{T} > 0$ such that $\displaystyle \frac{1}{M_{T}}\leqslant \rho(t,x) \leqslant M_{T} $ for all $t \in (0,T), x \in \oft$ and  $\mathrm{dist}(\Omega_{S}(t),\partial\Omega) \geqslant \nu/2$ for all $t \in [0,T].$
\end{thm}

Our second main result asserts global existence and uniqueness under a smallness condition on the initial data.

\begin{thm} \label{mainthm_glob}
Let $2 < p < \infty$ and $3 < q < \infty$ satisfying the condition $\displaystyle \frac{1}{p} + \frac{1}{2q} \neq \frac12$. Assume that \eqref{nocontact} is satisfied. Let $\overline \rho > 0$ and $\overline \vartheta > 0$ be two given constants. Then there exists $\eta_{0} > 0$ such that, for all $\eta \in (0,\eta_{0})$
 there exist two constants  $\delta_{0} > 0$ and $C > 0,$  such that, for all $\delta \in (0,\delta_{0})$ and  for any $(\rho_{0} ,u_{0}, \vartheta_{0},\ell_{0},\omega_{0})$ in $\mathcal{I}^{cc}_{p,q}$ with
 \begin{align} \label{eq:br}
\frac{1}{|\ofo|} \int_{\ofo} \rho_{0} \ {\rm d}x = \overline\rho,
\end{align}
and
\begin{align*}
\|(\rho_{0} - \overline\rho, u_{0}, \vartheta_{0} - \overline \vartheta, \ell_{0},\omega_{0})\|_{\mathcal{I}_{p,q}} \leqslant \delta,
\end{align*}
the system \eqref{fluid-eq} - \eqref{inicond} admits a unique strong solution  $(\rho, u, \vartheta, \ell, \omega)$ in the class of functions satisfying
\begin{multline} \label{est:glob}
\|(\rho  - \overline\rho) \|_{L^{\infty}(0,\infty; W^{1,q}(\Omega_{F}(\cdot)))} + \|e^{\eta (\cdot)}\nabla \rho\|_{W^{1,p}(0,\infty; L^{q}(\Omega_{F}(\cdot)))}+ \norm{ e^{\eta (\cdot)} \partial_{t} \rho}_{L^{p}(0,\infty; L^{q}(\Omega_{F}(\cdot)))}  \\+ \|e^{\eta (\cdot)} \ u\|_{L^{p}(0,\infty; W^{2,q}(\Omega_{F}(\cdot))^{3})}
+  \|e^{\eta (\cdot)} \partial_{t} u\|_{L^{p}(0,\infty; L^{q}(\Omega_{F}(\cdot))^{3})} + \|e^{\eta (\cdot)} u\|_{L^{\infty}(0,\infty; B^{2(1-1/p)}_{q,p}(\Omega_{F}(\cdot))^{3})} \\
+ \|e^{\eta(\cdot)} \partial_{t} \vartheta \|_{L^{p}(0,\infty; L^{q}(\Omega_{F}(\cdot)))}   + \norm{e^{\eta(\cdot)} \nabla \vartheta }_{L^{p}(0,\infty; L^{q}(\Omega_{F}(\cdot)))} + \norm{e^{\eta(\cdot)} \nabla^{2} \vartheta}_{L^{p}(0,\infty; L^{q}(\Omega_{F}(\cdot)))} \\
+ \| (\vartheta - \overline \vartheta)\|_{L^{\infty}(0,\infty; B^{2(1-1/p)}_{q,p}(\Omega_{F}(\cdot)))}
+ \|e^{\eta (\cdot)} \dot a \|_{L^{p}(0,\infty;\mathbb{R}^{3})} + \norm{ e^{\eta (\cdot)} \ddot a }_{L^{p}(0,\infty;\mathbb{R}^{3})} \\
+ \|a\|_{{L^{\infty}(0,\infty;\mathbb{R}^{3})}}+ \|e^{\eta (\cdot)} \ \omega \|_{W^{1,p}(0,\infty;\mathbb{R}^{3})} \leqslant C \delta.
\end{multline}
Moreover, $\rho(t,x) \geqslant \displaystyle \frac{\overline\rho}{2}$ for all $t \in (0,\infty)$, $x \in \oft$ and
$\mathrm{dist}(\Omega_{S}(t),\partial\Omega) \geqslant \nu/2$ for all $t \in [0,\infty)$.
\end{thm}

As shown in \cref{sec:global-existence}, a simple consequence of the above theorem is:

\begin{cor}\label{cor_density}
With the assumptions and notation in Theorem \ref{mainthm_glob} we have
\begin{align*}
\|u(t,\cdot)\|_{B^{2(1-1/p)}_{q,p}(\oft)^{3}} + \|\dot a(t)\|_{\rt} + \|\omega(t)\|_{\rt}\leqslant C \delta e^{-\eta t},
\end{align*}
\begin{align} \label{den}
\|\rho(t,\cdot) - \overline\rho\|_{W^{1,q}(\Omega_{F}(\cdot))} \leqslant C \delta e^{-\eta t},
\end{align}
where the constant $C$ is independent of $t>0$.
\end{cor}

To prove \cref{mainthm} and \cref{mainthm_glob} we follow a strategy which is widely used in the literature on existence and uniqueness of solutions for fluid-solid interaction models, which is:
\begin{itemize}
\item{\bf Step 1.} Since the domain of the fluid equation is one of the unknowns, we first rewrite the system in a fixed spatial domain. This can be achieved either by a ``geometric'' change of variables (see \cite{BoulGuer}), by using Lagrangian coordinates (see \cite{MTT17}) or by combining these two change of coordinates (see \cite{HiebMur}). In the present work, we found more convenient to use Lagrangian variables. Apart from allowing to rewrite the coupled system in a fixed cylindrical domain  this allows us to tackle the term $u \cdot \nabla \rho$ in the density equation.
\item{\bf Step 2.} Next we associate to the original nonlinear problem a linear one, involving source terms.
A crucial step here is to establish the $L^p$-$L^q$ regularity property for this linear problem. This is done by proving that the associated linear operators are $\mr$-sectorial in  an appropriate Banach spaces.
\item{\bf Step 3.}
We estimate the nonlinear terms in the governing equations and we use the Banach fixed point theorem to prove existence and uniqueness results in the reference configuration.
\item{\bf Step 4.} In the final step we come back to the original configuration.
\end{itemize}


\section{Notation} \label{sec:notation}

In this section, we fix some notations that we use throughout this
paper.  For $s \in (0,1)$ and a Banach space $U,$ $F^{s}_{p,q}(0,T,U)$
stands for $U$ valued Lizorkin-Triebel space. For precise definition
of such spaces we refer to \cite{Triebel,sch12}. If
$T < \infty$, this spaces can be characterised as follows (see
\cite{wei05})
\begin{align*}
F^{s}_{p,q}(0,T;U) =  \Big\{ f \in L^{p}(0,T,U) \mid | f|_{F^{s}_{p,q}(0,T;U)} < \infty \Big\},
\end{align*}
where
\begin{align*}
| f|_{F^{s}_{p,q}(0,T;U)} = \left(\int_{0}^{T}  \left( \int_{0}^{T-t} h^{-1 - sq}\|f(t+h) - f(t)\|^{q}_{U} \, {\rm d}h \right)^{p/q} \, {\rm d}t \right)^{1/p}.
\end{align*}
These spaces endowed with the natural norm
\begin{align}\label{num_LIZ}
\|f\|_{F^{s}_{p,q}(0,T;U)} = \|f\|_{L^{p}(0,T;U)} + | f|_{F^{s}_{p,q}(0,T;U)} .
\end{align}

If $T\in (0,\infty]$, we set $Q_{T}^{F} = (0,T) \times \ofo$ and 
\begin{align*}
W^{2,1}_{q,p} (Q_{T}^{F}) = L^{p}(0,T;W^{2,q}(\ofo))  \cap W^{1,p}(0,T;L^{q}(\ofo)),
\end{align*}
with
\begin{align*}
\|u\|_{W^{2,1}_{q,p} (Q_{T}^{F})} = \| u \|_{ L^{p}(0,T;W^{2,q}(\ofo))} + \|u\|_{W^{1,p}(0,T;L^{q}(\ofo))}.
\end{align*}

For $T\in (0,\infty]$, the space ${\mathcal S}_{T,p,q}$ is  defined by
 \begin{align} \label{solspace}
 \mathcal{S}_{T,p,q} &= \Big\{ (\rho,u,\vartheta, \ell, \omega) \mid \rho \in W^{1,p}(0,T;W^{1,q}(\ofo)), \quad  u \in \left( W^{2,1}_{q,p} (Q_{T}^{F})\right)^{3},  \notag \\
& \qquad\quad   \vartheta \in W^{2,1}_{q,p} (Q_{T}^{F}), \quad \ell \in W^{1,p}(0,T;\mathbb{R}^{3}), \quad \omega \in W^{1,p}(0,T;\mathbb{R}^{3})\Big\}
 \end{align}
and
 \begin{align*}
 \|(\rho,u,\vartheta,\ell,\omega)\|_{S_{T,p,q}}  = \|\rho\|_{W^{1,p}(0,T;W^{1,q}(\ofo))} +  \|u\|_{W^{2,1}_{q,p} (Q_{T}^{F})} + \|  \vartheta\|_{W^{2,1}_{q,p}(Q_{T}^{F})}  \notag \\
 \qquad \qquad \qquad + \|\ell\|_{W^{1,p}(0,T)} + \|\omega\|_{W^{1,p}(0,T)}.
 \end{align*}

 For any $T < \infty$ or $T=\infty$ we define $\mathcal{B}_{T,p,q}$ as follows
 \begin{multline}\label{com0.5}
 \mathcal{B}_{T,p,q} = \Big\{ (f_{1}, f_{2}, f_{3}, h, g_{1}, g_{2}) \mid f_{1} \in  L^{p}(0,T,W^{1,q}(\ofo)), f_{2} \in  L^{p}(0,T;L^{q}(\ofo))^{3}, \\ f_{3} \in  L^{p}(0,T;L^{q}(\ofo)),
  h \in F^{(1-1/q)/2}_{p,q}(0,T;L^{q}(\partial \ofo)) \cap L^{p}(0,T;W^{1-1/q,q}(\partial \ofo)), \\
  g_{1} \in L^{p}(0,T),   g_{2} \in L^{p}(0,T) \mbox{  with } h(0,y) = 0 \mbox{ for all } y \in \partial\ofo  \mbox{ if } \displaystyle \frac{1}{p} + \frac{1}{2q} < \frac{1}{2} \Big\},
 \end{multline}
 with
 \begin{align*}
 \|(f_{1},f_{2}, f_{3}, h, g_{1},g_{2})\|_{\mathcal{B}_{T,p,q}} =  \|f_{1}\|_{L^{p}(0,T;W^{1,q})} + \|f_{2}\|_{L^{p}(0,T;L^{q})}   + \|f_{3}\|_{L^{p}(0,T;L^{q})} \\ + \|h\|_{F^{(1-1/q)/2}_{p,q}(0,T;L^{q}(\partial \ofo)) \cap L^{p}(0,T;W^{1-1/q,q}(\partial \ofo))} + \|g_{1}\|_{L^{p}(0,T)} + \|g_{2}\|_{L^{p}(0,T)}.
 \end{align*}

\part{Local in Time Existence and Uniqueness}
\section{Lagrangian Change of Variables} \label{sec:cov-local}
In this section, we describe a change of variables, obtained by a
slight variation of the usual passage to Lagrangian coordinates, which
allows us to rewrite the governing equations in a fixed spatial domain
and to preserve the linear form of the transmission condition for the
velocity field.  More precisely, we consider the characteristics $X$
associated to the fluid velocity $u$, that is the solution of the
Cauchy problem
\begin{equation} \label{ode}
\begin{PDE}
\partial_{t}  X(t,y) & =  u(t,X(t,y))  \quad (t >0), \\
X(0,y)               &  = y\in \overline{\ofo}.
\end{PDE}
\end{equation}
Assume that $X(t,\cdot)$ is a $C^{1}$-diffeomorphism from $\ofo$ onto
$\oft$ for all $t \in (0,T)$ (see \eqref{tak0.1}).  For each
$t\in (0,T)$, we denote by $Y(t,\cdot) = [X(t,\cdot)]^{-1}$ the
inverse of $X(t,\cdot)$.  We consider the following change of
variables
\begin{gather}
\widetilde  \rho(t,y)  = \rho(t,X(t,y)) , \qquad \widetilde  u (t,y)   = Q^{-1}(t)u(t,X(t,y)), \notag \\
\widetilde  \vartheta(t,y)  = \vartheta(t,X(t,y)), \qquad \widetilde   p  = R \widetilde  \rho \widetilde  \vartheta , \label{lpq01}\\
\widetilde  \ell(t)   = Q^{-1}(t) \dot a(t), \qquad \widetilde  \omega(t)  = Q^{-1}(t) \omega(t), \notag
\end{gather}
for $(t,y) \in (0,T) \times \ofo$. 
In particular,
\begin{equation}
\label{lpq02} \rho(t,x)   = \widetilde  \rho(t,Y(t,x)), \quad  u(t,x)  = Q(t)\widetilde  u(t,Y(t,x)),\quad
\vartheta(t,x)   = \widetilde  \vartheta(t,Y(t,x)),
\end{equation}
for $(t,x) \in (0,T) \times \oft$.
This change of variables implies that
$(\widetilde  \rho, \widetilde  u, \widetilde  \vartheta, \widetilde  \ell, \widetilde  \omega, a, Q)$ satisfies
\begin{align}
\label{NL0.0} &
   \begin{PDE}
      \partial_{t}  \widetilde \rho + \rho_{0} \operatorname{div}  \widetilde  u = \mathcal{F}_{1} & \mbox{ in } (0,T) \times \ofo, \\
      \widetilde  \rho(0,\cdot) = \rho_{0} \hspace*{4.1cm} & \mbox{ in } \ofo, \\
   \end{PDE}\\
\label{NL0.1} &
   \begin{PDE}
       \partial_{t} \widetilde u - \dfrac{\mu}{ \rho_{0}} \Delta \widetilde  u - \dfrac{\alpha+\mu}{ \rho_{0}} \nabla (\operatorname{div} \ \widetilde u)
       = \mathcal{F}_{2}  & \mbox{ in } (0,T) \times \ofo,\\
       \widetilde u = 0 &\mbox{ on } (0,T) \times \partial \Omega,\\
       \widetilde u =  \widetilde \ell +  \widetilde\omega \times y & \mbox{ on } (0,T) \times \partial\oso, \\
       \widetilde u(0,\cdot) = u_{0} & \mbox{ in } \ofo,
     \end{PDE}  \\
\label{NL0.2} &
   \begin{PDE}
       m\dfrac{d}{dt} { \widetilde\ell} =  \mathcal{G}_{1}  & \qquad \qquad \qquad \mbox{ in } (0,T), \\
       J(0)\dfrac{d}{dt}\widetilde \omega  =   \mathcal{G}_{2} & \qquad \qquad \qquad \mbox{ in } (0,T), \\
       \widetilde \ell(0) = \ell_{0}, \quad
       \widetilde \omega(0) = \omega_{0}.
   \end{PDE}\\
\label{NL0.3} &
   \begin{PDE}
       \partial_{t}\widetilde \vartheta - \dfrac{\kappa}{\rho_{0}c_{v}}\Delta \widetilde \vartheta  \   = \mathcal{F}_{3} & \mbox{ in } (0,T) \times \ofo, \\
       \dfrac{\partial \widetilde \vartheta}{\partial n} = \mathcal{H}\cdot n &  \mbox{ on } (0,T) \times \partial \ofo,  \\[3mm]
       \widetilde \vartheta(0,\cdot) = \vartheta_{0} \hspace*{4.1cm} & \mbox{ in } \ofo,
     \end{PDE} \\
 \label{def-Q} &
 \begin{PDE}
 \dot a =Q \widetilde \ell & \qquad \qquad \qquad \mbox{ in } (0,T), \\
\dot Q  =  Q A(\widetilde  \omega) & \qquad \qquad \qquad \mbox{ in } (0,T), \\
a(0)=0, \quad Q(0) = I_{3},
 \end{PDE}
\end{align}
where
\begin{align} \label{Jacobi}
X(t,y)  = y + \int_{0}^{t} Q(s) \widetilde  u(s,y) \ {\rm d}s, \quad \mbox{ and } \quad \nabla Y(t,X(t,y)) = [\nabla X]^{-1}(t,y),
\end{align}
for every $y\in \Omega_F(0)$ and $t\geqslant 0$. Using  the notation
\begin{align} \label{Z}
Z(t,y) = \left( Z_{i,j}\right)_{1\leqslant i, j\leqslant 3}=[\nabla X]^{-1}(t,y) \qquad\qquad(t\geqslant 0,\ y\in \Omega_F(0)),
\end{align}
the remaining terms in \eqref{NL0.0}--\eqref{NL0.3} are defined by:
\begin{flalign} \label{F1}
& \mathcal{F}_{1}(\widetilde \rho,\widetilde  u,\widetilde  \vartheta,\widetilde  \ell,\widetilde \omega)
=  \;  -(\widetilde  \rho - \rho_{0}) \operatorname{div} \widetilde  u  -  \widetilde  \rho  (\nabla \widetilde  u) : \left[ (ZQ)^{\top} - I_3\right], &
\end{flalign}
\begin{flalign} \label{F2bisold}
(\mathcal{F}_{2})_i(\widetilde \rho,\widetilde  u,\widetilde  \vartheta,\widetilde  \ell,\widetilde \omega)
= &\;  - \frac{\widetilde  \rho}{\rho_0}(\widetilde \omega\times Q \widetilde  u)_i +\left(1-\frac{\widetilde  \rho}{\rho_0}\right)(\partial_t \widetilde  u)_{i}
       -  \frac{\widetilde  \rho}{\rho_{0}} \left[ (Q - I) \partial_{t} \widetilde  u\right]_{i} \notag\\
  &  + \frac{\mu}{\rho_0}\sum_{l,j,k} \frac{\partial^2 (Q\widetilde  u)_{i} }{\partial y_l\partial y_k}  \left(Z_{k,j}  - \delta_{k,j}\right) Z_{l,j}
    + \frac{\mu}{\rho_{0}}\sum_{l,k} \frac{\partial^2 (Q\widetilde  u)_{i} }{\partial y_l\partial y_k} \left( Z_{l,k} - \delta_{l,k} \right) \notag \\
   &   + \frac{\mu}{\rho_{0}} \left[ (Q - I) \Delta \widetilde  u\right]_{i}
     + \frac{\mu}{\rho_0}\sum_{l,j,k} Z_{l,j} \frac{\partial  (Q \widetilde  u)_i}{\partial y_k} \frac{\partial Z_{k,j}}{\partial y_{l}}\notag \\
   &  +\frac{\mu+\alpha}{\rho_{0}} \sum_{l,j,k} \frac{\partial^{2} (Qu)_{j}}{\partial y_{l} \partial y_{k}} \left( Z_{k,j} - \delta_{k,j}\right)  Z_{l,i}
    + \frac{\mu+\alpha}{\rho_0}  \sum_{l,j} \frac{\partial^{2} (Qu)_{j}}{\partial y_{l} \partial y_{j}} \left( Z_{l,i} - \delta_{l,i}\right) \notag \\
   & + \frac{\alpha + \mu}{\rho_{0}} \frac{\partial}{\partial y_{i}} \left[ \nabla \widetilde  u : (Q^{\top} - I_{3}) \right] +\ds
\frac{\alpha + \mu}{\rho_{0}} \sum_{l,j,k} Z_{l,i} \frac{\partial (Qu)_{j}}{\partial y_{k}} \frac{\partial Z_{k,j}}{\partial y_{l}}
\notag \\
& - R \frac{\widetilde  \vartheta}{\rho_{0}}  \left(Z^{\top} \nabla \widetilde  \rho \right)_{i}  -  R \frac{\widetilde  \rho}{\rho_{0}}  \left(Z^{\top} \nabla \widetilde  \vartheta \right)_{i}, &
\end{flalign}
\begin{flalign} \label{F3bis}
\mathcal{F}_{3}(\widetilde \rho,\widetilde  u,\widetilde  \vartheta,\widetilde  \ell,\widetilde \omega)
= & \;
     \left(\frac{\rho_{0} - \widetilde  \rho}{c_{v}\rho_{0}}\right)\partial_t \widetilde  \vartheta
        - \frac{R\widetilde  \vartheta\widetilde  \rho}{c_v \rho_0} \left[ ZQ\right]^\top : \nabla \widetilde  u
      + \frac{\kappa}{c_v \rho_0} \sum_{l,j,k} Z_{l,j}\frac{\partial \widetilde  \vartheta}{\partial y_k} \frac{\partial Z_{k,j}}{\partial y_l} \notag \\
     & + \frac{\kappa}{c_v \rho_0} \sum_{j,k,l} \frac{\partial^2 \widetilde  \vartheta}{\partial y_k\partial y_l} \left(Z_{k,j}-\delta_{k,j} \right)Z_{l,j}  + \frac{\kappa}{c_v \rho_0} \sum_{k,l} \frac{\partial^2 \widetilde  \vartheta}{\partial y_k\partial y_l} \left(Z_{l,k} -\delta_{l,k} \right)
     \notag \\
     &  + \frac{\alpha}{c_{v}\rho_0} \left(\left[ ZQ\right]^\top : \nabla \widetilde  u\right)^{2}
      + \frac{\mu}{2 c_{v} \rho_{0}} \left|\nabla \widetilde  u Z Q + \left(\nabla \widetilde  u Z Q\right)^{\top}\right|^2, &
\end{flalign}
\begin{flalign}
\label{H}
&\mathcal{H}(\widetilde \rho,\widetilde  u,\widetilde  \vartheta,\widetilde  \ell,\widetilde \omega) = \; \mathbbm{1}_{\partial \Omega} \mathcal{H}_{F} + \mathbbm{1}_{\partial \oso} \mathcal{H}_{S}, \\
\notag
&\mathcal{H}_{F} (\widetilde \rho,\widetilde  u,\widetilde  \vartheta,\widetilde  \ell,\widetilde \omega) =  \; (I_3- Z^{\top}) \nabla \widetilde  \vartheta
\qquad
\mathcal{H}_{S}  (\widetilde \rho,\widetilde  u,\widetilde  \vartheta,\widetilde  \ell,\widetilde \omega) =  (I_3 - (Z Q)^{\top})\nabla \widetilde  \vartheta ,\\
\label{G0}
 &\mathcal{G}_{0} (\widetilde \rho,\widetilde  u,\widetilde  \vartheta,\widetilde  \ell,\widetilde \omega)
=   \; \mu \left[\nabla \widetilde  u  Z Q + \left(\nabla \widetilde  u Z Q \right)^{\top}   \right]
 +\alpha \left(\left[Z  Q\right]^\top : \nabla \widetilde  u\right) I_3
 +  R   \widetilde  \rho \widetilde  \vartheta I_3,\\
\label{G1}
 & \mathcal{G}_{1} (\widetilde \rho,\widetilde  u,\widetilde  \vartheta,\widetilde  \ell,\widetilde \omega) = \; - m (\widetilde  \omega \times \widetilde  \ell) -   \int_{\partial \oso} {\mathcal G}_{0} n \ {\rm d} \gamma , \\
\label{G2}
& \mathcal{G}_{2}(\widetilde \rho,\widetilde  u,\widetilde  \vartheta,\widetilde  \ell,\widetilde \omega) = \;  J(0) \widetilde  \omega \times \widetilde  \omega  -  \int_{\partial \oso} y \times \mathcal{G}_{0} n \ {\rm d}\gamma . &
\end{flalign}

Using the above change of variables, our main result in \cref{mainthm} can be rephrased as:

\begin{thm} \label{mainthm2}
Let $2 < p < \infty$ and $3 < q < \infty$ satisfying the condition $\displaystyle \frac{1}{p} + \frac{1}{2q} \neq \frac12$.  Assume that  $(\rho_{0},u_{0}, \vartheta_{0},\ell_{0},\omega_{0})$ belongs to $\mathcal{I}^{cc}_{p,q}$ such that \eqref{ini-ball} holds.
Then, there exists $\widetilde T > 0$ such that for any $T\in (0,\widetilde T)$, the system \eqref{NL0.0} - \eqref{G2} admits a unique strong solution 
\begin{align} \label{regl}
& \widetilde\rho \in W^{1,p}(0,T; W^{1,q}(\Omega_{F}(0))) 
,\notag \\
& \widetilde u \in L^{p}(0,T; W^{2,q}(\Omega_{F}(0))^{3}) \cap W^{1,p}(0,T; L^{q}(\Omega_{F}(0))^{3}) \cap C([0,T]; B^{2(1-1/p)}_{q,p}(\Omega_{F}(0))^{3}), \notag\\
& \widetilde\vartheta \in L^{p}(0,T; W^{2,q}(\Omega_{F}(0))) \cap W^{1,p}(0,T; L^{q}(\Omega_{F}(0))) \cap C([0,T]; B^{2(1-1/p)}_{q,p}(\Omega_{F}(0))), \notag \\
& \widetilde\ell \in W^{1,p}(0,T;\mathbb{R}^{3}), \quad \widetilde\omega \in W^{1,p}(0,T;\mathbb{R}^{3}),\\
& a \in W^{2,p}(0,T;\mathbb{R}^{3}), \quad Q \in W^{2,p}(0,T;SO(3)), \notag \\
& X \in W^{1,p}(0,T;W^{2,q}(\ofo)) \cap W^{2,p}(0,T;L^{q}(\ofo)), \notag \\
& X(t,\cdot) : \ofo\to \oft \ \text{is a} \ C^1-\text{diffeormorphim for all} \ t\in [0,T]. \notag
\end{align}
Moreover, there exists a constant $M_{T} > 0$, such that $\displaystyle \frac{1}{M_{T}}\leqslant \widetilde\rho(t,y) \leqslant M_{T}$, for all $t \in (0,T), y \in \ofo$.
\end{thm}
The proof of the above theorem relies on a fixed point theorem and a
  linearization. We describe below the main steps of the proof  using the maximal regularity of an associated linear problem and some estimates of the non linear terms involved in the fixed point procedure. For the clarity of the presentation , we postpone the detailed proofs of the maximal regularity and the estimates of the non linear terms technical results in \cref{sec:linearized-problem-loc} and \cref{sec:nonlin-terms-loc} respectively.

\begin{proof}[Proof of \cref{mainthm2}]
   Assume
\[
(\rho_{0},u_{0}, \vartheta_{0},\ell_{0},\omega_{0})\in\mathcal{I}^{cc}_{p,q},
\]
is given (see \eqref{inispace}) and \eqref{ini-ball} holds with $M>0$.

We consider the following linear problem.
\begin{align}
\label{com0.0} &
   \begin{PDE}
      \partial_{t}  \widetilde \rho + \rho_{0} \operatorname{div}  \widetilde  u = f_{1} & \mbox{ in } (0,T) \times \ofo, \\
      \widetilde  \rho(0,\cdot) = \rho_{0} \hspace*{4.1cm} & \mbox{ in } \ofo, \\
   \end{PDE}\\
\label{com0.1} &
   \begin{PDE}
       \partial_{t} \widetilde u - \dfrac{\mu}{ \rho_{0}} \Delta \widetilde  u - \dfrac{\alpha+\mu}{ \rho_{0}} \nabla (\operatorname{div} \ \widetilde u)   = f_{2}  & \mbox{ in } (0,T) \times \ofo,\\
       \widetilde u = 0 &\mbox{ on } (0,T) \times \partial \Omega,\\
       \widetilde u =  \widetilde \ell +  \widetilde\omega \times y & \mbox{ on } (0,T) \times \partial\oso, \\
       \widetilde u(0,\cdot) = u_{0} & \mbox{ in } \ofo,
     \end{PDE}  \\
\label{com0.2} &
   \begin{PDE}
       m\dfrac{d}{dt} { \widetilde\ell} =  g_{1}  & \mbox{ in } (0,T), \\
       J(0)\dfrac{d}{dt}\widetilde \omega  =   g_{2} & \mbox{ in } (0,T), \\
       \widetilde \ell(0) = \ell_{0}, \quad
       \widetilde \omega(0) = \omega_{0}.
   \end{PDE}\\
\label{com0.3} &
   \begin{PDE}
       \partial_{t}\widetilde \vartheta - \dfrac{\kappa}{\rho_{0}c_{v}}\Delta \widetilde \vartheta  \   = f_{3} & \mbox{ in } (0,T) \times \ofo, \\
       \dfrac{\partial \widetilde \vartheta}{\partial n} = h &  \mbox{ on } (0,T) \times \partial \ofo,  \\[3mm]
       \widetilde \vartheta(0,\cdot) = \vartheta_{0} \hspace*{4.1cm} & \mbox{ in } \ofo.
     \end{PDE}
\end{align}

where
\[
(\rho_{0},u_{0}, \vartheta_{0},\ell_{0},\omega_{0})\in\mathcal{I}^{cc}_{p,q},
\quad
(f_{1},f_{2},f_{3}, h,  g_{1},g_{2}) \in \mathcal{B}_{T,p,q}
\]
are given (see \eqref{inispace} and \eqref{com0.5}) and such that the
initial conditions satisfy \eqref{ini-ball} with $M>0$.

In  the following we shall denote for each $T_{*} \in (0,T]$, by
  $ \widetilde {\mathcal{B}}_{T_{*}}$ the closed unit ball in
  $\mathcal{B}_{T_{*},p,q}$.

\medskip

In \cref{sec:linearized-problem-loc} we will construct a solution
to (\ref{com0.0})---(\ref{com0.3}) providing a bounded 'solution
operator'
\[
    F:
    \left\{
    \begin{array}{rcl}
\mathcal{B}_{T,p,q} &\to& \mathcal{S}_{T,p,q} \\
(f_{1},f_{2},f_{3}, h,  g_{1},g_{2}) &\mapsto&  (\widetilde \rho, \widetilde u, \widetilde \vartheta,  \widetilde \ell, \widetilde \omega),
\end{array}
\right.
\]
where we recall that $\mathcal{B}_{T,p,q}$ and $\mathcal{S}_{T,p,q}$ are defined in \eqref{com0.5} and in  \eqref{solspace}.

In Section~\ref{sec:nonlin-terms-loc} we then prove norm estimates for the
nonlinear terms,
$\mathcal{F}_1, \mathcal{F}_{2}, \mathcal{F}_{3}, \mathcal{H}_F,
\mathcal{H}_S, \mathcal{G}_1, \mathcal{G}_2$.
More precisely, assuming that
$$
T\in (0,\widetilde T), \quad\text{with } \widetilde{T}<1\quad \text{small enough,}
$$
we show that
the obtained norm bounds here depend, up to a constant, on $T^\delta$
where $\delta$ depends on $p, q$ only. This allows  us to define the
operator
\[
\mathcal{N}:
 \left\{
    \begin{array}{rcl}
\mathcal{B}_{T,p,q} &\to& \mathcal{B}_{T,p,q} \\
(f_{1},f_{2},f_{3}, h,  g_{1},g_{2}) &\mapsto&  (\mathcal{F}_1,  \mathcal{F}_{2},  \mathcal{F}_{3},  \mathcal{H}, \mathcal{G}_1,  \mathcal{G}_2),
\end{array}
\right.
\]
and to show that, for sufficiently small $\widetilde T$, it becomes a self-map of the closed ball
$$
\{ (f_{1},f_{2},f_{3}, h,  g_{1},g_{2}) \in \mathcal{B}_{T,p,q} \mid \| (f_{1},f_{2},f_{3}, h,  g_{1},g_{2}) \|_{\mathcal{B}_{T,p,q}} \leq 1\}.
$$

Finally, in \cref{prop:lipestimate} a Lipschitz
estimate for $\mathcal{N}$ is proved, again with a Lipschitz constant
depending on $T^\delta$, provided that $T\in (0,\widetilde T)$. This allows us  to enforce a strict contraction
on the above closed ball and hence a fixed point, that provides a
solution to (\ref{NL0.0})---(\ref{G2}) satisfying \eqref{regl}. The bound of $\widetilde \rho$ will be obtained from the estimate \eqref{est:d-id}.
\end{proof}

From \cref{mainthm2} we can now deduce \cref{mainthm}.

\begin{proof}[Proof of \cref{mainthm}.]
  Let us assume that
  $(\rho_{0},u_{0},\vartheta_{0},\ell_{0},\omega_{0}) \in
  \mathcal{I}^{cc}_{p,q}$
  satisfies the condition \eqref{ini-ball}.

  Let $T<\widetilde T$ with $\widetilde T$ as in \cref{mainthm2}. In particular, there exists a unique solution
  $(\widetilde \rho, \widetilde u, \widetilde \vartheta, \widetilde g,
  \widetilde \omega)$
to the system \eqref{NL0.0} - \eqref{G0} satisfying
\begin{align*}
&\widetilde\rho \in W^{1,p}(0,  T; W^{1,q}(\Omega_{F}(0))) \\
&\widetilde u \in L^{p}(0, T; W^{2,q}(\Omega_{F}(0))^{3}) \cap W^{1,p}(0, T; L^{q}(\Omega_{F}(0))^{3}) \\
&\widetilde \vartheta \in L^{p}(0, T; W^{2,q}(\Omega_{F}(0))) \cap W^{1,p}(0, T; L^{q}(\Omega_{F}(0))) \\
&\widetilde \ell \in W^{1,p}(0, T;\mathbb{R}^{3}), \quad \widetilde \omega \in W^{1,p}(0, T;\mathbb{R}^{3}).
\end{align*}
Since  $ T \leqslant \widetilde T$, $X(t,\cdot)$ is $C^{1}-$ diffeomorphism from $\ofo$ into $\oft,$ we set   $Y(t,\cdot) = X^{-1}(t,\cdot)$ and  for $x \in \oft, \ t \geq 0$
\begin{align*}
& \rho(t,x) = \widetilde \rho(t,Y(t,x)), \\
& u(t,x) = Q(t)\widetilde u(t,Y(t,x)),  \\
& \vartheta(t,x) = \widetilde \vartheta(t,Y(t,x)),\\
& \dot a(t) = Q(t) \widetilde \ell(t) ,  \quad  \omega(t) = Q(t)\widetilde \omega(t).
\end{align*}
We can then check that $(\rho,u,\vartheta,a,\omega)$ satisfies the original system \eqref{fluid-eq} - \eqref{inicond} and
\begin{align*}
& \rho \in W^{1,p}(0,T; W^{1,q}(\Omega_{F}(\cdot))) \cap C([0,T];W^{1,q}(\Omega_{F}(\cdot))), \\
& u \in L^{p}(0,T; W^{2,q}(\Omega_{F}(\cdot))^{3}) \cap W^{1,p}(0,T; L^{q}(\Omega_{F}(\cdot))^{3}) \cap C([0,T]; B^{2(1-1/p)}_{q,p}(\Omega_{F}(\cdot))^{3}), \\
& \vartheta \in L^{p}(0,T; W^{2,q}(\Omega_{F}(\cdot))) \cap W^{1,p}(0,T; L^{q}(\Omega_{F}(\cdot))) \cap C([0,T]; B^{2(1-1/p)}_{q,p}(\Omega_{F}(\cdot))), \\
& a \in W^{2,p}(0,T;\mathbb{R}^{3}), \quad \omega \in W^{1,p}(0,T;\mathbb{R}^{3}).
\end{align*}
The uniqueness for the solution of \eqref{fluid-eq} - \eqref{inicond} follows from uniqueness of solution to the system \eqref{NL0.0} - \eqref{G0}. Since $a(t)$ and $\omega(t)$ belongs to $C([0,T];\rt),$ using \eqref{nocontact} we obtain  $\mathrm{dist}(\Omega_{S}(t),\partial\Omega) \geqslant \nu/2$ for all $t \in [0,T]$ if $T$ is small enough. This completes the proof of \cref{mainthm}.
\end{proof}

\section{Maximal $L^{p}-L^{q}$  Regularity for a Linear Problem.}
\label{sec:linearized-problem-loc}
In this section, we fix $\widetilde{T}>0$, and $1 < p, q < \infty$ such that
$\displaystyle \frac{1}{p} + \frac{1}{2q} \neq 1$ and $\displaystyle \frac{1}{p} + \frac{1}{2q} \neq \frac{1}{2}$.  We also take
$T\in (0,\widetilde{T}]$. We consider the linear system (\ref{com0.0})---(\ref{com0.3})
associated with \eqref{NL0.0}--\eqref{NL0.3}, where we replace the terms in the right-hand side
by given source terms. The initial data for the system \eqref{com0.0}---\eqref{com0.3} satisfies the following properties:
\begin{gather}
\rho_{0} \in W^{1,q}(\ofo) \cap C(\overline{\ofo}), \quad \min_{\overline{\ofo}} \rho_{0} \geqslant  M, \notag \\
u_{0} \in B^{2(1-1/p)}_{q,p}(\ofo)^{3}, \quad
\vartheta_{0} \in B^{2(1-1/p)}_{q,p}(\ofo), \quad
 \ell_{0} \in \mathbb{R}^{3},\  \omega_{0} \in \mathbb{R}^{3}, \label{ini-l}\\
  u_{0} = 0 \mbox{ on } \partial \Omega, \
\quad u_{0}(y) = \ell_{0} + \omega_{0} \times y \quad y\in \partial\oso \mbox{ if } \frac{1}{p} + \frac{1}{2q} < 1, \notag \\
\frac{\partial \vartheta_{0}}{\partial n}  = 0, \mbox{ on } \partial\ofo   \mbox{ if }  \displaystyle \frac{1}{p} + \frac{1}{2q} < \frac{1}{2}. \notag
\end{gather}

Observe that the linear system can be solved ``in cascades'':
Equation \eqref{com0.2} can be solved independently and admits a unique solution
\[
  (\widetilde \ell, \widetilde \omega)\in W^{1,p}(0,T)^{3} \times W^{1,p}(0,T)^{3}.
\]
Moreover there exists a constant $C=C(\widetilde{T})$ independent of $T$ such that
\begin{equation} \label{lw}
\norm{ \widetilde \ell }_{W^{1,p}(0,T)^{3}} + \norm{ \widetilde \omega }_{W^{1,p}(0,T)^{3}}
  \leqslant
C \Big(  \norm{ \ell_{0} }_{\mathbb{R}^{3}} + \norm{ \omega_{0} }_{\mathbb{R}^{3}} + \norm{ g_{1} }_{L^{p}(0,T)^{3}} + \norm{ g_{2} }_{L^{p}(0,T)^{3}} \Big).
\end{equation}

We also note that if we show that system \eqref{com0.1} admits a
unique solution $\widetilde{u}\in W^{2,1}_{q,p} (Q_{T}^{F})^{3}$, we
can use the Sobolev embedding to prove that system \eqref{com0.0}
admits a unique solution
$\widetilde \rho\in W^{1,p}(0,T;W^{1,q}(\ofo))$. There exists again a
constant $C=C(\widetilde{T})$ independent of $T$ such that
\begin{equation} \label{com0.4}
\left\|\widetilde \rho \right\|_{W^{1,p}(0,T;W^{1,q}(\ofo))} 
\leqslant  C\left( \norm{ \widetilde{u} }_{W^{2,1}_{q,p} (Q_{T}^{F})^{3}}
+\norm{ f_1 }_{L^{p}(0,T;W^{1,q}(\ofo))}+\norm{ \rho_0 }_{W^{1,q}(\ofo)}
\right).
\end{equation}

Consequently, in order to solve \eqref{com0.0}--\eqref{com0.3}, we
need to solve the two parabolic systems \eqref{com0.1} and
\eqref{com0.3}. This is done below by using \cite[Theorem
2.3]{DenkHieberPruss07}.

\begin{prop} \label{thm:v-s} With the above notation, assume
  $T\in (0,\widetilde{T}]$ and $u_{0}\in B^{2(1-1/p)}_{q,p}(\ofo)^{3}$ with
\begin{align}
u_{0} = 0 \mbox{ on } \partial \Omega,  \quad u_{0} = \ell_{0} + \omega_{0} \times y \mbox{ on } \partial\oso \quad \mbox{if} \quad \frac{1}{p} + \frac{1}{2q} < 1.
\end{align}
Then for any $f_{2} \in L^{p}(0,T;L^{q}(\ofo))^{3}$, system
\eqref{com0.1} admits a unique strong solution
$\widetilde u\in W^{2,1}_{q,p} (Q_{T}^{F})^{3}$. Moreover, there
exists a constant $C > 0$ depending only on $M$, $\widetilde{T}$ and $\ofo$
such that
\begin{multline} \label{est:lin-v}
\norm{ \widetilde u }_{W^{2,1}_{q,p} (Q_{T}^{F})^{3}}
\leqslant C \Big( \norm{ u_{0} }_{B^{2(1-1/p)}_{q,p}(\ofo)^{3}} +
\norm{ \ell_{0} }_{\mathbb{R}^{3}} + \norm{ \omega_{0} }_{\mathbb{R}^{3}} \\
+ \norm{ g_{1} }_{L^{p}(0,T)^{3}} + \norm{ g_{2} }_{L^{p}(0,T)^{3}} + \norm{ f_{2} }_{L^{p}(0,T;L^{q}(\ofo))^{3}}
\Big).
\end{multline}
\end{prop}

\begin{proof}
We take $\eta \in C^{\infty}(\overline{\ofo})$ such that
\[
\eta =  0 \quad \text{on} \ \partial\Omega, \quad  \eta = 1  \quad \text{on} \ \partial\oso.\]
For $(t,y) \in (0,T) \times \ofo$, we set $w(t,y) = \eta(y)( \widetilde \ell(t) + \widetilde \omega(t) \times y)$. Therefore, using \eqref{lw} we see that there exists a positive constant $C$ depending on $\widetilde{T}$  and $\ofo$ such that
\begin{equation} \label{est:w}
  \begin{split}
& \quad \norm{ w(0,\cdot) }_{B^{2(1-1/p)}_{q,p}(\ofo)^{3}} + \norm{ w }_{W^{2,1}_{q,p} (Q_{T}^{F})^{3}} \\
\leqslant & \quad C \Big(  \norm{ \ell_{0} }_{\mathbb{R}^{3}} + \norm{ \omega_{0} }_{\mathbb{R}^{3}} + \norm{ g_{1} }_{L^{p}(0,T)^{3}} + \norm{ g_{2} }_{L^{p}(0,T)^{3}} \Big).
  \end{split}
\end{equation}
We look for the solution of  \eqref{com0.1} of the form  $\widetilde u = v + w$, where $v$ is the solution of
 \begin{equation} \label{eq:v}
   \begin{PDE}
   \ds   \partial_{t} v - \frac{\mu}{ \rho_{0}} \Delta v- \frac{\alpha+\mu}{ \rho_{0}} \nabla (\operatorname{div} \  v )   = \widehat f_{2}, \quad & \mbox{ in } (0,T) \times \ofo,  \\
   \ds   v = 0 & \mbox{ on } (0,T) \times \partial \ofo, \\
    \ds   v(0,\cdot) = v_{0} \quad & \mbox{ in } \ofo,
   \end{PDE}
 \end{equation}
with $v_{0} = u_{0} - w(0,\cdot)$ and
with
\[
\displaystyle \widehat f_{2} = f_{2} - \partial_{t} w + \frac{\mu}{\rho_{0}} \Delta w + \frac{\alpha + \mu}{\rho_{0}} \nabla( \operatorname{div} \ w).\]
We can check that $\widehat f_{2}$ belongs to $L^{p}(0,T;L^{q}(\ofo))^{3}$
and  that there exists a constant depending only on $M$ such that
\begin{align} \label{est:f2}
\norm{ \widehat f_{2} }_{L^{p}(0,T;L^{q}(\ofo))^{3}} \leqslant  C \Big( \norm{  f_{2} }_{L^{p}(0,T;L^{q}(\ofo))^{3}} + \norm{ w }_{W^{2,1}_{q,p} (Q_{T}^{F})^{3}} \Big).
\end{align}
Moreover,  since $v_{0} \in B^{2(1-1/p)}_{q,p}(\ofo)^{3}$ with $v_{0} = 0$ on $\partial \ofo$ if $\displaystyle \frac{1}{p} + \frac{1}{2q} < 1$,
we can apply \cite[Theorem 2.3]{DenkHieberPruss07} and deduce that \eqref{eq:v} admits unique solution $v \in W^{2,1}_{q,p}(Q_{T}^{F})^{3}$.

More precisely, here the ellipticity of the interior symbol can be checked since $\rho_{0}(y) \geq \frac{1}{M} > 0$ for all $y \in \ofo,$ $\mu > 0$ and $\alpha + \frac{2}{3} \mu \geqslant 0.$
 The Lopatinsky-Shapiro condition also can be verified (see \cite[Section 6]{Agra}).

This yields that \eqref{com0.1} admits a unique solution $\widetilde u\in W^{2,1}_{q,p}(Q_{T}^{F})^{3}$.
In order to prove estimate \eqref{est:lin-v}, we apply again \cite[Theorem 2.3]{DenkHieberPruss07} on the system
\[
\begin{PDE}
\ds \partial_{t} \widehat v - \frac{\mu}{ \rho_{0}} \Delta \widehat v- \frac{\alpha+\mu}{ \rho_{0}} \nabla (\operatorname{div} \  \widehat v )   = \widehat f_{T}, \quad & \mbox{ in } (0,\widetilde{T}) \times \ofo,  \\
  \widehat v = 0 & \mbox{ on } (0,\widetilde{T}) \times \partial \ofo \\
 \widehat v(0) = v_{0} \quad & \mbox{ in } \ofo.
\end{PDE}
\]
In particular, for any $\widehat f \in L^{p}(0,\widetilde{T};L^{q}(\ofo)^3)$,
there exists a unique solution
$$\widehat v\in L^{p}(0,\widetilde{T};\linebreak W^{2,q}(\ofo)^3) \cap W^{1,p}(0,\widetilde{T},L^{q}(\ofo)^3)$$
of the above system and by the closed graph theorem,
there exists a constant $C_{\widetilde{T}}>0$ such that
\begin{flalign*}
&\norm{ \widehat v }_{L^{p}(0,\widetilde{T};W^{2,q}(\ofo)^3)} + \norm{ \widehat v }_{W^{1,p}(0,\widetilde{T},L^{q}(\ofo)^3)} \leqslant C_{\widetilde{T}} \Big( \norm{ v_{0} }_{B^{2(1-1/p)}_{q,p}(\ofo)^{3}} + \norm{ \widehat f }_{L^{p}(0,\widetilde{T};L^{q}(\ofo)^3)} \Big).&
\end{flalign*}
Then we take
\begin{align*}
\widehat f= \begin{cases}
\widehat f_{2} &  \mbox { if } 0 < t \leqslant T, \\
0 & \mbox{ if }  T < t \leqslant \widetilde{T},
\end{cases}
\end{align*}
and by the uniqueness of the solution, we note that  $\widehat v = v$ for all $t \in [0,T]$. Thus the above estimate, \eqref{est:w} and \eqref{est:f2} yield
\eqref{est:lin-v}.
\end{proof}

Next we consider system \eqref{com0.3}.
\begin{prop} \label{thm:heat}
With the above notation, assume $T\in (0,\widetilde{T}]$, $\vartheta_{0} \in B^{2(1-1/p)}_{q,p}(\ofo)$ and $h \in F^{(1-1/q)/2}_{p,q}(0,T;L^{q}(\partial \ofo)) \cap L^{p}(0,T;W^{1-1/q,q}(\partial \ofo))$
with the compatibility condition
  \begin{align}
  \frac{\partial \vartheta_{0}}{\partial n}= h(0,\cdot) = 0    \quad \mbox{on} \ \partial\ofo  \quad \mbox{if} \quad \displaystyle \frac{1}{p} + \frac{1}{2q} < \frac{1}{2}.
  \end{align}
 Then for any $f_{3} \in L^{p}(0,T;L^{q}(\ofo))$, system \eqref{com0.3} admits a unique strong solution $\widetilde \vartheta\in W^{2,1}_{q,p} (Q_{T}^{F})$. Moreover, there exists a constant $C > 0$, depending only on $M$ and $\widetilde{T}$, such that
\begin{multline} \label{est:heat}
\norm{ \widetilde \vartheta }_{W^{2,1}_{q,p} (Q_{T}^{F})}  \leqslant C \Big( \norm{ \vartheta_{0} }_{B^{2(1-1/p)}_{q,p}(\ofo)} + \norm{ f_{3} }_{L^{p}(0,T;L^{q}(\ofo))} +\norm{ h }_{F^{(1-1/q)/2}_{p,q}(0,T;L^{q}(\partial \ofo))} \\ + \norm{ h }_{L^{p}(0,T;W^{1-1/q,q}(\partial \ofo))} \Big).
\end{multline}
\end{prop}

\begin{proof}
The existence and the regularity results follow from \cite[Theorem 2.3]{DenkHieberPruss07}.

Since we need a constant $C$ in \eqref{est:heat} independent of $T\in (0,\widetilde{T}]$ and this fact is not explicitly stated in \cite{DenkHieberPruss07}, we provide below a short argument showing that the constant $C$ can indeed be chosen to be uniform for $T\in (0,\widetilde{T}]$.  To this aim, we decompose $\widetilde \vartheta$ in the form $\widetilde \vartheta = \widetilde \vartheta_{1} + \widetilde \vartheta_{2}$, where $\widetilde \vartheta_{1}$ solves
\begin{equation} \label{eq:heat1}
  \begin{PDE}
 \ds \partial_{t}\widetilde \vartheta_{1} - \frac{\kappa}{\rho_{0}c_{v}}\Delta \widetilde \vartheta_{1}  \   = f_{3} \quad & \mbox{ in } (0,T) \times \ofo,   \\
 \ds  \frac{\partial \widetilde \vartheta_{1}}{\partial n} = 0  \quad & \mbox{ on } (0,T) \times \partial \ofo, \\
\ds \widetilde \vartheta_{1}(0) = \vartheta_{0}  \quad & \mbox{ in } \ofo,
  \end{PDE}
\end{equation}
and $\widetilde \vartheta_{2}$ solves
\begin{equation} \label{eq:heat2}
  \begin{PDE}
  \ds \partial_{t}\widetilde \vartheta_{2} - \frac{\kappa}{\rho_{0}c_{v}}\Delta \widetilde \vartheta_{2}  \   = 0 \quad & \mbox{ in } (0,T) \times \ofo,   \\
 \ds \frac{\partial \widetilde \vartheta_{2}}{\partial n} = h \quad   & \mbox{ on } (0,T) \times \partial \ofo, \\
\ds \widetilde \vartheta_{2}(0) = 0  \quad & \mbox{ in } \ofo.
  \end{PDE}
\end{equation}
Proceeding as in the proof of \cref{thm:v-s}, we first obtain
\begin{align}
\norm{ \widetilde \vartheta_{1} }_{W^{2,1}_{q,p} (Q_{T}^{F})}  \leqslant C \Big( \norm{ \vartheta_{0} }_{B^{2(1-1/p)}_{q,p}(\ofo)} + \norm{ f_{3} }_{L^{p}(0,T;L^{q}(\ofo))} \Big),
\end{align}
where the constant $C$ may depend  on $\widetilde{T}$ but  is independent of $T$. Let us set
\begin{align*}
 \widehat h = \begin{cases}
 h &  \mbox { if } 0 < t \leqslant T, \\
 0 & \mbox{ if }   T - \widetilde{T} < t \leqslant 0.
\end{cases}
\end{align*}
We first verify that
$$
\widehat h  \in F^{(1-1/q)/2}_{p,q}(T-\widetilde{T},T;L^{q}(\partial \ofo))
\cap L^{p}(T-\widetilde{T},T;W^{1-1/q,q}(\partial \ofo)).$$
Obviously  $\widehat h $  belongs to $L^{p}(T-\widetilde{T},T;W^{1-1/q,q}(\partial \ofo))$. The fact that $\widehat h$ belongs to $F^{(1-1/q)/2}_{p,q}(T-\widetilde{T},T;L^{q}(\partial \ofo))$ follows from  \cite[Remark 2, Section 3.4.3, p.211]{Triebel}. Moreover
\begin{multline*}
 \norm{ \widehat h }_{F^{(1-1/q)/2}_{p,q}(T-\widetilde{T},T;L^{q}(\partial \ofo))} + \norm{ \widehat h }_{L^{p}(T-\widetilde{T},T;W^{1-1/q,q}(\partial \ofo))} \\
 =  \norm{ h }_{F^{(1-1/q)/2}_{p,q}(0,T;L^{q}(\partial \ofo))}   + \norm{ h }_{L^{p}(0,T;W^{1-1/q,q}(\partial \ofo))}.
\end{multline*}
We consider the system
\begin{equation} \label{eq:heat3}
  \begin{PDE}
 \partial_{t}\widehat \vartheta_{2} - \frac{\kappa}{\rho_{0}c_{v}}\Delta \widehat \vartheta_{2}  \   = 0 \quad & \mbox{ in } (T-\widetilde{T},T) \times \ofo,   \\
  \frac{\partial \widehat \vartheta_{2}}{\partial n} = \widehat h  \quad & \mbox{ on } (T-\widetilde{T},T) \times \partial \ofo, \\
\widehat \vartheta_{2}(T-\widetilde{T}) = 0  \quad & \mbox{ in } \ofo.
  \end{PDE}
\end{equation}
Note that, $\widehat \vartheta_{2} = 0$ for all $t\in [T- \widetilde{T},0]$ and  $\widehat \vartheta_{2} = \widetilde \vartheta_{2} $ for all $t \in [0,T]$. Therefore, we have
\begin{align*}
  & \; \norm{  \widetilde \vartheta_{2} }_{L^{p}(0,T;W^{2,q}(\ofo))} +  \norm{  \widetilde \vartheta_{2}  }_{W^{1,p}(0,T,L^{q}(\ofo))}  \\
= &\; \norm{ \widehat \vartheta_{2} }_{L^{p}(T-\widetilde{T},T;W^{2,q}(\ofo))} + \norm{ \widehat \vartheta_{2} }_{W^{1,p}(T-\widetilde{T},T,L^{q}(\ofo))}  \\
\leqslant & \; C_{\widetilde{T}} \Big( \norm{ \widehat h }_{F^{(1-1/q)/2}_{p,q}(T-\widetilde{T},T;L^{q}(\partial \ofo))} + \norm{ \widehat h }_{L^{p}(T-\widetilde{T},T;W^{1-1/q,q}(\partial \ofo))} \Big)  \\
\leqslant & \; C_{\widetilde{T}} \Big( \norm{ h }_{F^{(1-1/q)/2}_{p,q}(0,T;L^{q}(\partial \ofo))}  + \norm{ h }_{L^{p}(0,T;W^{1-1/q,q}(\partial \ofo))}  \Big).
\end{align*}
This completes the proof of the proposition.
\end{proof}

Combining \cref{thm:v-s} and \cref{thm:heat}, we obtain the following result
\begin{thm} \label{thm1}
    Let $\widetilde{T}$ be an arbitrary fixed given time. Let $1 < p < \infty$ and $1 < q < \infty$ satisfying the conditions $\ds \frac{1}{p} + \frac{1}{2q} \neq 1$ and $\ds \frac{1}{p} + \frac{1}{2q} \neq \frac12$. Let  $(\rho_{0},u_{0}, \vartheta_{0},\ell_{0},\omega_{0})$ satisfy the assumptions \eqref{ini-l}.
    Then for any  $(f_{1},f_{2},f_{3}, h,  g_{1},g_{2}) \in \mathcal{B}_{T,p,q}$ and $T \in (0,\widetilde{T})$ the system \eqref{com0.0}--\eqref{com0.3} admits a unique solution $(\widetilde \rho, \widetilde u,\widetilde \vartheta,\widetilde \ell,\widetilde \omega) \in \mathcal{S}_{T,p,q}$
    and there exists a constant $C > 0$ depending  on $p,q,M,\widetilde{T}$  and independent of $T$ such that
\begin{equation} \label{est:thm1}
\norm{ (\widetilde \rho,\widetilde u,\widetilde \vartheta,\widetilde \ell,\widetilde \omega) }_{S_{T,p,q}}
  \leqslant C  \Big(  \norm{ (\rho_{0}, u_{0}, \vartheta_{0}, \ell_{0},\omega_{0}) }_{\mathcal{I}_{p,q}}  +  \norm{ (f_{1},f_{2}, f_{3}, h, g_{1},g_{2}) }_{\mathcal{B}_{T,p,q}} \Big).
\end{equation}
\end{thm}

\section{Estimating the Nonlinear Terms} \label{sec:nonlin-terms-loc}

In order to prepare the forthcoming fixed point argument, we provide
in this section estimates of
$\mathcal{F}_1,\ \mathcal{F}_2,\ \mathcal{F}_3, \mathcal{G}_1,
\mathcal{G}_2$,
$\mathcal{H}_F$ and $\mathcal{H}_S$ defined in \eqref{F1}-\eqref{G2}
where $(f_{1},f_{2},f_{3}, h, g_{1},g_{2})$ are given and where
$(\widetilde \rho, \widetilde u,\widetilde \vartheta,\widetilde
\ell,\widetilde \omega)$
is the corresponding solution of \eqref{com0.0}---\eqref{com0.3}
given by Theorem~\ref{thm1}.

\medskip

Assume $2 < p < \infty$ and $3 < q < \infty$ satisfy
$\ds \frac{1}{p} + \frac{1}{2q} \neq \frac12$. Let $p'$
denote the conjugate of $p$, i.e.,
$\frac{1}{p} + \frac{1}{p'} = 1$. We will frequently use the following immediate consequences of H\"older's inequality.
\begin{align}
\norm{ f }_{L^{p}(0,T)} \leqslant T^{\nicefrac{1}{p} - \nicefrac{1}{r}} \norm{ f }_{L^{r}(0,T)}, \qquad &\mbox{ for all } f \in L^{r}(0,T), \ r > p,  \label{est:lplr} \\
 \norm{ f }_{L^{\infty}(0,T)} \leqslant T^{\nicefrac{1}{p'}} \norm{ f }_{W^{1,p}(0,T)}, \qquad &\mbox{ for all } f \in W^{1,p}(0,T), f(0) = 0. \label{est:liw1p}
\end{align}
In the following, when no confusion is possible, we will use the notation
\[
\norm{  \cdot  }_{W^{r,p}(0,T;W^{s,q})} = \norm{ \cdot }_{W^{r,p}(0,T;W^{s,q}(\ofo))}.
\]
We next recall three estimates which play an essential role in the
remaining part of this section. For the first two estimates we refer
to the relevant literature, whereas for the third one we
provide a short proof.

\begin{prop} \cite[Lemma 4.2]{ShibataMurata16} \label{prop:LiBesov}
  Let $1 < p,q < \infty$ and $T$ be any positive number. Let $\Omega$
  be a smooth domain in $\mathbb{R}^{n}$. Then for any
  $u \in W^{2,1}_{q,p}((0,T)\times \Omega)$,
\begin{align} \label{eq:u-lib}
\sup_{t \in (0,T)} \norm{ u(t) }_{B^{2(1-\nicefrac{1}{p})}_{q,p}(\Omega)} \leqslant C \left( \norm{ u(0) }_{B^{2(1-\nicefrac{1}{p})}_{q,p}(\Omega)} + \norm{ u }_{W^{2,1}_{q,p}((0,T)\times \Omega)}\right),
\end{align}
where the constant $C$ is independent of $T$.
\end{prop}

To state the second  estimate, we use the Lizorkin-Triebel space
$F^{s}_{p,q}(0,T;X)$ defined in \eqref{num_LIZ}.

\begin{prop} \label{prop:Ntrace}\cite[Proposition 6.4]{DenkHieberPruss07}
Let $1 < p,q < \infty$  and $T$ be any positive number. Let $\Omega$ be a smooth domain in $\mathbb{R}^{n}$. Then for any $u \in W^{2,1}_{q,p}((0,T)\times \Omega)$,  $\nabla u |_{\partial \Omega}$ belongs to $ F^{(1-\nicefrac{1}{q})/2}_{p,q}(0,T;L^{q}(\partial \Omega)) \cap L^{p}(0,T;W^{1-\nicefrac{1}{q},q}(\partial \Omega))$. Moreover,
\begin{align} \label{est:Ntrace}
\bignorm{ \nabla u\cdot n}_{F^{(1-\nicefrac{1}{q})/2}_{p,q}(0,T;L^{q}(\partial \Omega)) \cap L^{p}(0,T_;W^{1-\nicefrac{1}{q},q}(\partial \Omega))} \leqslant C \left( \norm{ u(0) }_{B^{2(1-\nicefrac{1}{p})}_{q,p}(\Omega)} + \norm{ u }_{W^{2,1}_{q,p}((0,T)\times \Omega)}\right),
\end{align}
where the constant $C$ is independent of time $T$.
\end{prop}

The third one of the estimates mentioned above is given in the following result.

\begin{prop} \label{prop:TL-product} Let $U_{1},U_{2}$ and $U_{3}$ be
  three Banach spaces and $\Phi: U_{1} \times U_{2} \to U_{3}$ a
  bounded bilinear map. Let us assume that
  $f \in F^{s}_{p,q}(0,T;U_{1})$ and $g \in W^{1,p}(0,T;U_{2})$ for some
  $s \in (0, 1)$, $p > 2$ and $q >3$. Let us assume that $g(0) = 0$.
  If $s + \frac{1}{p} < 1$, then we have
\begin{align} \label{lem6.15est}
\norm{ \Phi(f,g) }_{F^{s}_{p,q}(0,T;U_{3})} \leqslant C T^{\delta}    \norm{ g }_{W^{1,p}(0,T;U_{2})} \norm{ f }_{F^{s}_{p,q}(0,T;U_{2})},
\end{align}
for some positive constant $\delta$ depending only on $p,q$ and $s$
and the constant $C$ is independent of time $T$.
\end{prop}

\begin{proof}
From the boundedness of $\Phi$ and \eqref{est:liw1p} we infer
\begin{align}
\norm{ \Phi(f,g) }_{L^{p}(0,T;U_{3})} \leqslant C \norm{ f }_{L^{p}(0,T;U_{1})} \norm{ g }_{L^{\infty}(0,T;U_{2})} \leqslant C T^{\nicefrac{1}{p'}} \norm{ f }_{L^{p}(0,T;U_{1})} \norm{ g }_{W^{1,p}(0,T;U_{2})},
\end{align}
since $g(0) = 0$. Using  again the boundedness of $\Phi$,
\begin{align*}
   \; |\Phi(f,g)|_{F^{s}_{p,q}(0,T;U_{3})}^{p}
= & \; \int_{0}^{T}  \left( \int_{0}^{T-t} h^{-1 - sq}\norm{ \Phi(f(t+h), g(t+h)) - \Phi(f(t),g(t)) }^{q}_{U_{3}} \ dh \right)^{\nicefrac{p}{q}} \ dt  \\
\leqslant & \; C_{p,q} \int_{0}^{T} \left(\int_{0}^{T-t} h^{-1 - sq} \norm{ (f(t+h) - f(t) }^{q}_{U_{1}} \norm{ g(t+h) }^{q}_{U_{2}}   \  dh \right)^{\nicefrac{p}{q}} \ dt \\
          &  +  \; C_{p,q} \int_{0}^{T} \left( \int_{0}^{T-t} h^{-1 - sq} \norm{ f(t) }^{q}_{U_{1}} \norm{ g(t+h) - g(t) }^{q}_{U_{2}} \ dh \right)^{\nicefrac{p}{q}} \ dt \\
= & \; I_{1} + I_{2}.
\end{align*}
We estimate $I_{1}$ using  \eqref{est:liw1p}
\[
I_{1} \leqslant C_{p,q} \norm{ g }_{L^{\infty}(0,T;U_{2})}^{p}  | f|_{F^{s}_{p,q}(0,T;U_{1})}^{p}
\leqslant C_{p,q} T^{\nicefrac{p}{p'}} \norm{ g }_{W^{1,p}(0,T;U_{2})}^{p} \norm{ f }^{p}_{F^{s}_{p,q}(0,T;U_{1})}.
\]
Since $g \in W^{1,p}(0,T;U_2)$, by using H\"older's inequality we have
\[
\norm{ g(t+h,\cdot) - g(t,\cdot) }_{U_{2}} \leqslant  h^{\nicefrac{1}{p'}}\norm{ g }_{W^{1,p}(0,T;U_{2})}, \mbox{ for all } h \in  (0, T -t),  t \in (0,T).
\]
Using the above estimate and the fact that $0 < s + \nicefrac{1}{p} < 1$, we get
\begin{align*}
I_{2} &\leqslant C_{p,q} \norm{ g }_{W^{1,p}(0,T;U_{2})}^{p}\int_{0}^{T} \norm{ f(t) }_{U_{1}}^{p} \left(\int_{0}^{T-t} h^{-1 -sq} \ h^{\nicefrac{q}{p'}}\ dh  \right)^{\nicefrac{p}{q}} \ dt \\
& \leqslant C_{p,q} \norm{ g }_{W^{1,p}(0,T;U_{2})}^{p} \int_{0}^{T} \norm{ f(t) }_{U_{1}}^{p} \left( \frac{(T-t)^{q(1 - \nicefrac{1}{p} - s)}}{q(1 - \nicefrac{1}{p} - s)}\right)^{\nicefrac{p}{q}} \ dt \\
& \leqslant C_{p,q,s} T^{p(1 - \nicefrac{1}{p} - s)}   \norm{ g }_{W^{1,p}(0,T;U_{2})}^{p}  \norm{ f }^{p}_{L^{p}(0,T;U_{1})}.
\end{align*}
Combining the above estimates, we obtain \eqref{lem6.15est}.
\end{proof}

Our aim is to estimate the non linear terms in \eqref{F1}-\eqref{G2}:
\begin{prop} \label{prop:estimate}
Let $2 < p < \infty$ and $3 < q < \infty$ satisfying
  the condition
  $ \frac{1}{p} + \frac{1}{2q} \neq \frac12$. Let
  $(\rho_{0},u_{0}, \vartheta_{0},\ell_{0},\omega_{0}) \in
  \mathcal{I}^{cc}_{p,q}$
  such that \eqref{ini-ball} holds. There exist $\widetilde{T}<1$, a constant $\delta > 0$ depending only on $p$ and $q$,
  and a constant $C > 0$ depending only on $p,q,M,\widetilde{T}$
  such that for $T\in (0,\widetilde{T}]$
  and for $(f_{1},f_{2},f_{3}, h,  g_{1},g_{2}) \in \mathcal{B}_{T,p,q}$ satisfying
  $$
 \| (f_{1},f_{2},f_{3}, h,  g_{1},g_{2}) \|_{\mathcal{B}_{T,p,q}} \leq 1,
$$
  the solution of $(\widetilde \rho, \widetilde u, \widetilde \vartheta, \widetilde
  \ell, \widetilde \omega ) \in \mathcal{S}_{T,p,q}$
  of \eqref{com0.0}---\eqref{com0.3} verifies 
    \begin{equation*}
  \left\| (\mathcal{F}_{1},\mathcal{F}_{2},\mathcal{F}_{3}, \mathcal{H},  \mathcal{G}_{1}, \mathcal{G}_{2}) \right\|_{\mathcal{B}_{T,p,q}} \leq CT^{\delta}.
  \end{equation*}
\end{prop}

\begin{proof}
We consider $\widetilde{T}<1$ and we assume that $T\in (0,\widetilde{T}]$.
The constants $C$ appearing in this proof depend only on $M$.

From  \eqref{est:thm1} in \cref{thm1}, we first obtain
\begin{equation}
	\norm{  (\widetilde \rho, \widetilde u, \widetilde \vartheta, \widetilde \ell, \widetilde \omega)  }_{\mathcal{S}_{T,p,q}} \leqslant C.  \label{est:all}
\end{equation}

Combining \eqref{est:liw1p} and \eqref{est:all}, we deduce
\begin{equation} \label{est:d-id}
   \norm{ \widetilde \rho- \rho_{0} }_{L^{\infty}(0,T;W^{1,q}(\ofo))} \leqslant C T^{\nicefrac{1}{p'}}
\end{equation}
and
\begin{equation}\label{tak0.6}
   \norm{ \widetilde \rho }_{L^{\infty}(0,T;W^{1,q})} \leqslant  C.
\end{equation}
In a similar manner, we can obtain
\begin{equation}\label{tak0.7}
   \norm{ \widetilde \omega   }_{L^{\infty}(0,T)}, \quad \norm{ \widetilde \ell }_{L^{\infty}(0,T)} \leqslant  C,
\end{equation}
and combining these estimates with \eqref{est:lplr} yields
\begin{equation}
\norm{ \widetilde  \rho }_{L^{p}(0,T;W^{1,q})}+ \norm{ \widetilde \omega   }_{L^{p}(0,T)} + \norm{ \widetilde \ell }_{L^{p}(0,T)} \leqslant  C T^{\nicefrac{1}{p}}. \label{est:lpdma}
\end{equation}

Since $2 < p < \infty$, one has
$B^{2(1-\nicefrac{1}{p})}_{q,p}(\ofo) \hookrightarrow W^{1,q}(\ofo)$. Therefore,
using \cref{prop:LiBesov} and \eqref{est:all} we get
\begin{equation}
\norm{ \widetilde u }_{L^{\infty}(0,T;W^{1,q})}+  \norm{ \widetilde \vartheta  }_{L^{\infty}(0,T;W^{1,q})} \leqslant  C. \label{est:livt}
\end{equation}

For all $s \in (0,1)$ we have by complex interpolation
\[
 \norm{  \widetilde u (t,\cdot)  }_{W^{1+s,q}(\ofo)}   \leqslant  C\norm{   \widetilde u (t,\cdot) }_{W^{2,q}(\ofo)}^{(1+s)/2} \norm{  \widetilde u(t,\cdot) }_{L^{q}(\ofo)}^{(1-s)/2},
\]
and thus
\[
 \norm{  \widetilde u }_{L^{p}(0,T, W^{1+s,q}(\ofo)}
 \leqslant
 C T^{(1-s)/2p} \norm{ \widetilde u }_{L^{\infty}(0,T;L^{q}(\ofo)}^{\nicefrac{1-s}2}  \norm{ \widetilde u }_{L^{p}(0,T;W^{2,q}(\ofo))}^{\nicefrac{1+s}2}.
\]
Therefore,  using \eqref{est:lplr}, \eqref{est:all} and \eqref{est:livt}, we get
\begin{equation}
\norm{ \widetilde u }_{L^{p}(0,T;W^{1+s,q})} + \norm{ \widetilde \vartheta }_{L^{p}(0,T;W^{1+s,q})} \leqslant  C T^{(1-s)/2p}, \quad s \in (0,1). \label{est:imp}
\end{equation}
Combining the above estimate with the fact that  $L^{\infty}(\ofo) \hookrightarrow W^{s,q}(\ofo))$ for  $s \in (\nicefrac{3}{q},1)$,
we deduce
\begin{equation}\label{est:lisv}
\norm{    \widetilde u }_{L^{p}(0,T,L^{\infty})}
+ \norm{ \nabla\widetilde u }_{L^{p}(0,T,L^{\infty})}
+\norm{ \nabla\widetilde \vartheta  }_{L^{p}(0,T,L^{\infty})}
\leqslant C T^{(1-s)/2p},  \quad s \in (\nicefrac{3}{q},1).
\end{equation}

The solution of \eqref{def-Q} satisfies $Q\in SO(3)$ and thus $|Q(t)|=1$ for all $t$. In particular,
$|\dot Q|\leq C|\widetilde \omega|$ and we deduce from \eqref{est:liw1p} and \eqref{est:all}
\begin{equation} \label{est:Q}
\norm{ Q }_{L^{\infty}(0,T; \RR^{3\times3})} \leqslant C \qquad\text{and}\qquad  \norm{ Q -I_{3}  }_{L^{\infty}(0,T; \RR^{3\times3})} \leqslant C T^{\nicefrac{1}{p'}}.
\end{equation}

Let $X$ be defined as in \eqref{Jacobi}. Then
\[
\sup_{t \in (0,T) }\norm{ \nabla X(t,\cdot) - I_3 }_{W^{1,q}(\ofo)} \leqslant C \int_{0}^{T} \norm{ \nabla \widetilde u    }_{W^{1,q}(\ofo)} \leqslant  C T^{\nicefrac{1}{p'}} \norm{ \nabla \widetilde u }_{L^{p}(0,T;W^{1,q})} .
\]
Now using $W^{1,q}(\ofo)\subset L^\infty(\ofo)$ and \eqref{est:all} we deduce from the above estimate
\begin{equation} \label{est:Jli}
    \sup_{t \in (0,T)}   \norm{ \nabla X(t, \cdot) - I_{3} }_{L^{\infty}(\ofo)} \leqslant C T^{\nicefrac{1}{p'}}
\end{equation}
In particular, there exists $\widetilde T$ such that
\begin{equation}\label{tak0.1}
  \norm{ \nabla X(t, \cdot) - I_{3} }_{L^{\infty}(\ofo)} \leqslant \frac12
\end{equation}
for all $0 < t < T \leqslant  \widetilde T$. This implies that $\nabla X(t, \cdot)$ is invertible and we can thus define $Z = [\nabla X]^{-1}$.
More precisely, combining
\[
\partial_{t} \nabla X (t,y) = Q(t) \nabla \widetilde u(t,y),
\]
and \eqref{est:all} and \eqref{est:Q} we get
\[
\norm{ \partial_{t} \nabla X }_{L^{p}(0,T;W^{1,q})} \leqslant  \norm{ Q }_{L^{\infty}(0,T)} \norm{ \nabla \widetilde u }_{L^{p}(0,T;W^{1,q})} \leqslant C,
\]
where $C$ depends only on $M$. The above estimate and \eqref{est:Jli}
yield
\[
\norm{  \nabla X  }_{W^{1,p}(0,T;W^{1,q}(\ofo))} + \norm{ \nabla X  }_{L^{\infty}(0,T;W^{1,q}(\ofo))}\leqslant C.
\]
Since  $W^{1,p}(0,T;W^{1,q}(\ofo))$ and
$L^{\infty}(0,T;W^{1,q}(\ofo))$ are algebras for $p >2$ and
$q > 3$, this implies
\begin{gather}
\norm{  \det \nabla X  }_{W^{1,p}(0,T;W^{1,q}(\ofo))} + \norm{ \det \nabla X  }_{L^{\infty}(0,T;W^{1,q}(\ofo))}  \leqslant C ,  \\
\norm{  \cof \nabla X  }_{W^{1,p}(0,T;W^{1,q}(\ofo))} + \norm{ \cof \nabla X  }_{L^{\infty}(0,T;W^{1,q}(\ofo))}  \leqslant C.
\end{gather}
From \eqref{tak0.1}, we deduce that $\det \nabla X \geq C>0$ in $(0,T)\times \ofo$ and thus
from
\[
Z = \frac{1}{\det \nabla X} (\cof \nabla X)^\top,
\]
we deduce
\begin{equation}\label{tak0.2}
\norm{  Z  }_{W^{1,p}(0,T;W^{1,q}(\ofo))} + \norm{  Z  }_{L^{\infty}(0,T;W^{1,q}(\ofo))} \leqslant C.
\end{equation}
\begin{equation} \label{est:X}
  \begin{array}{ll}
\norm{  \nabla X  }_{W^{1,p}(0,T;W^{1,q}(\ofo))} + \norm{ \nabla X  }_{L^{\infty}(0,T;W^{1,q}(\ofo))}  & \leqslant C ,  \\
  \end{array}
\end{equation}
The above estimate combined with \eqref{est:Q} and with \eqref{est:all} implies
\begin{equation}\label{tak0.3}
\norm{  QZ  }_{W^{1,p}(0,T;W^{1,q}(\ofo))} + \norm{  QZ  }_{L^{\infty}(0,T;W^{1,q}(\ofo))} \leqslant C.
\end{equation}

We are now in position to estimate the non linear terms in \eqref{F1}-\eqref{G2}:
\paragraph{\underline{Estimate of  $\mathcal{F}_{1}$}}
\begin{equation}\label{tak0.4}
\norm{ \mathcal{F}_{1} }_{L^{p}(0,T;W^{1,q}(\ofo))} \leqslant C T^{\nicefrac{1}{p'}}.
\end{equation}
Since $W^{1,q}(\ofo)$ is an algebra for $q >3$, we can write
\begin{multline}
\norm{ \mathcal{F}_{1} }_{L^{p}(0,T;W^{1,q}(\ofo))}
\leq
C\norm{  \widetilde \rho -  \rho_{0}  }_{L^{\infty}(0,T;W^{1,q})} \norm{ \operatorname{div} \ \widetilde u }_{L^{p}(0,T;W^{1,q})}
\\
+C
\norm{ \widetilde \rho  }_{L^{\infty}(0,T;W^{1,q})} \norm{ QZ-I_{3} }_{L^{\infty}(0,T;W^{1,q})}
               \norm{ \nabla \widetilde u }_{L^{p}(0,T;W^{1,q})}.
\end{multline}
Combining the above estimate with \eqref{tak0.6}, \eqref{est:d-id}, \eqref{est:all}, \eqref{tak0.3} and \eqref{est:liw1p}, we deduce \eqref{tak0.4}.

\paragraph{\underline{Estimate of  $\mathcal{F}_{2}$}}
\begin{equation}\label{tak0.5}
\norm{ \mathcal{F}_{2} }_{L^{p}(0,T;W^{1,q})} \leqslant C T^{\nicefrac{1}{p}}.
\end{equation}

Let us recall the definition \eqref{F2bisold} of $\mathcal{F}_{2}$:
\begin{align*}
(\mathcal{F}_{2})_i&(\widetilde\rho,\widetilde u,\widetilde \vartheta,\widetilde \ell,\widetilde\omega) =
 \;  -\frac{\widetilde \rho}{\rho_0}(\widetilde\omega\times Q \widetilde u)_i +\left(1-\frac{\widetilde \rho}{\rho_0}\right)(\partial_t \widetilde u)_{i}
     -  \frac{\widetilde \rho}{\rho_{0}} \left[ (Q - I) \partial_{t} \widetilde u\right]_{i} \\
& \; + \frac{\mu}{\rho_0}\sum_{l,j,k} \frac{\partial^2 (Q\widetilde u)_{i} }{\partial y_l\partial y_k}  \left(Z_{k,j}  - \delta_{k,j}\right) Z_{l,j}
 + \frac{\mu}{\rho_{0}}\sum_{l,k} \frac{\partial^2 (Q\widetilde u)_{i} }{\partial y_l\partial y_k} \left( Z_{l,k} - \delta_{l,k} \right)\\
& \; + \frac{\mu}{\rho_{0}} \left[ (Q - I) \Delta \widetilde u\right]_{i} +\frac{\mu}{\rho_0}\sum_{l,j,k} Z_{l,j} \frac{\partial  (Q \widetilde u)_i}{\partial y_k} \frac{\partial Z_{k,j}}{\partial y_{l}}\\
& \; + \frac{\mu+\alpha}{\rho_{0}} \sum_{l,j,k} \frac{\partial^{2} (Qu)_{j}}{\partial y_{l} \partial y_{k}} \left( Z_{k,j} - \delta_{k,j}\right)  Z_{l,i}  +
       \frac{\mu+\alpha}{\rho_0}  \sum_{l,j} \frac{\partial^{2} (Qu)_{j}}{\partial y_{l} \partial y_{j}} \left( Z_{l,i} - \delta_{l,i}\right) \\
& \; + \frac{\alpha + \mu}{\rho_{0}} \frac{\partial}{\partial y_{i}} \left[ \nabla \widetilde u : (Q^{\top} - I_{3}) \right] +
       \frac{\alpha + \mu}{\rho_{0}} \sum_{l,j,k} Z_{l,i} \frac{\partial (Qu)_{j}}{\partial y_{k}} \frac{\partial Z_{k,j}}{\partial y_{l}}\\
& \; - R \frac{\widetilde \vartheta}{\rho_{0}}  \left(Z^{\top} \nabla \widetilde \rho \right)_{i}  -  R \frac{\widetilde \rho}{\rho_{0}}  \left(Z^{\top} \nabla \widetilde \vartheta \right)_{i}
\end{align*}
\textbullet \ Estimate of first term of ${\mathcal F}_{2} :$   using \eqref{est:lplr},\eqref{tak0.6}, \eqref{tak0.7} and \eqref{est:livt}, we have
\begin{align*}
             \;   \Bignorm{ \frac{\widetilde \rho}{\rho_0}(\widetilde\omega\times Q \widetilde u)_i  }_{L^{p}(0,T;L^{q})}
\leqslant &  \; C \bignorm{ \widetilde \rho  }_{L^{\infty}(0,T;W^{1,q})} \bignorm{ \widetilde \omega }_{L^{\infty}(0,T; \RR^3 )}  \norm{ \widetilde u  }_{L^{p}(0,T;L^{q})} \\
\leqslant & C T^{\nicefrac{1}{p}}.
\end{align*}
\textbullet \ Estimate of second term of ${\mathcal F}_{2} :$  using  \eqref{est:all}  and \eqref{est:d-id} we have
\begin{align*}
            \Bignorm{ \left(1-\frac{\widetilde \rho}{\rho_0}\right)(\partial_t \widetilde u)_{i}}_{L^{p}(0,T;L^{q})}
\leqslant & \; C \norm{ \widetilde \rho  -  \rho_{0} }_{L^{\infty}(0,T;W^{1,q})}  \norm{ \partial_{t} \widetilde u  }_{L^{p}(0,T;L^{q})} \\
\leqslant & \; C T^{\nicefrac{1}{p'}}.
\end{align*}
\textbullet \ Estimate of third term of ${\mathcal F}_{2} :$ using   \eqref{est:all}, \eqref{tak0.6} and  \eqref{est:Q}
\begin{align*}
\bignorm{ \frac{\widetilde \rho}{\rho_{0}} \left[ (Q - I) \partial_{t} \widetilde u\right]_{i}}_{L^{p}(0,T;L^{q})}
 \leqslant &\;C \norm{ \widetilde \rho }_{L^{\infty}(0,T;W^{1,q})} \norm{ Q - I_{3} }_{L^{\infty}(0,T)} \norm{ \widetilde u  }_{W^{1,p}(0,T;L^{q})}  \\
 \leqslant &\;C T^{\nicefrac{1}{p'}}.
\end{align*}
\textbullet \ Estimate of fourth term of ${\mathcal F}_{2}$ (the estimate of fifth, eighth and ninth therm of ${\mathcal F}_{2}$ are similar)
\begin{align*}
          & \; \Bignorm{  \frac{\mu}{\rho_0}\sum_{l,j,k} \frac{\partial^2 (Q\widetilde u)_{i} }{\partial y_l\partial y_k}  \left(Z_{k,j}  - \delta_{k,j}\right) Z_{l,j}  }_{L^{p}(0,T;L^{q})} \\
\leqslant & \; C\norm{ \widetilde u }_{L^{p}(0,T;W^{2,q})}    \bignorm{  Z - I_3}_{L^{\infty}(0,T;W^{1,q})}  \norm{ Z  }_{L^{\infty}(0,T;W^{1,q})} \\
\leqslant &\; C T^{\nicefrac{1}{p'}},
\end{align*}
by using \eqref{est:all}, \eqref{tak0.2} and \eqref{est:liw1p}.

\textbullet \ Estimate of sixth and tenth term of ${\mathcal F}_{2} :$
\begin{align*}
 &\; \Bignorm{  \frac{\mu}{\rho_{0}} \left[ (Q - I) \Delta \widetilde u\right]_{i} }_{L^{p}(0,T;L^{q})}
+ \Bignorm{  \frac{\alpha + \mu}{\rho_{0}} \frac{\partial}{\partial y_{i}} \left[ \nabla \widetilde u : (Q^{\top} - I_{3}) \right] }_{L^{p}(0,T;L^{q})} \\
\leqslant &\;C \norm{ Q - I_{3} }_{L^{\infty}(0,T, \RR^{3\times3})} \norm{ \widetilde u }_{L^{p}(0,T;W^{2,q})} \\
\leqslant &\;C T^{\nicefrac{1}{p'}}.  \qquad \quad  (\mbox{ using } \eqref{est:all} \mbox{ and }  \eqref{est:Q} )
\end{align*}
\textbullet \ Estimate of seventh term and similarly, the eleventh term of ${\mathcal F}_{2}$:
notice that for any $1\leqslant j,k,l \leqslant 3$ and all $y \in \ofo$,
$\ds \frac{\partial Z_{k,j}}{\partial y_l} (0,y) = 0$. Therefore, using
\eqref{est:liw1p} we have
\[ 
           \Bignorm{ \frac{\partial Z_{k,j}}{\partial y_l} }_{L^{\infty}(0,T;L^{q})}
\leqslant T^{\nicefrac{1}{p'}} \Bignorm{ \frac{\partial Z_{k,j}}{\partial y_l} }_{W^{1,p}(0,T;L^{q})}
\leqslant T^{\nicefrac{1}{p'}} \norm{ Z }_{W^{1,p}(0,T;W^{1,q})} \leqslant C T^{\nicefrac{1}{p'}}.
\] 
Using this estimate, along with  \eqref{est:all}, \eqref{est:Q} and \eqref{tak0.2} we infer
\begin{align*}
          &  \; \Bignorm{  \frac{\mu}{\rho_0}\sum_{l,j,k} Z_{l,j} \frac{\partial  (Q \widetilde u)_i}{\partial y_k} \frac{\partial Z_{k,j}}{\partial y_{l}} }_{L^{p}(0,T;L^{q})} \\
\leqslant & \; C \norm{ Q }_{L^{\infty}(0,T)} \bignorm{  Z}_{L^{p}(0,T;W^{1,q})} \bignorm{  \nabla \widetilde u}_{L^{p}(0,T;W^{1,q})} \sum_{j,k,l} \Bignorm{ \frac{\partial Z_{k,j}}{\partial y_l} }_{L^{\infty}(0,T;L^{q})} \\
\leqslant & \; C T^{\nicefrac{1}{p'}}.
\end{align*}
\textbullet \ Estimate of twelfth  term of ${\mathcal F}_{2} :$
\begin{align*}
          \bignorm{  R \frac{\widetilde \vartheta}{\rho_{0}}  \left(Z^{\top} \nabla \widetilde \rho \right)_{i} }_{L^{p}(0,T;L^{q})}
\leqslant & \; C \norm{ \widetilde \vartheta  }_{L^{p}(0,T;W^{1,q})} \norm{ Z }_{L^{\infty}(0,T;W^{1,q})} \norm{ \nabla \widetilde \rho }_{L^{\infty}(0,T;L^{q})} \\
\leqslant & \; C  T^{\nicefrac{1}{p}}.   \qquad \quad  (\mbox{ using } \eqref{tak0.6}, \eqref{tak0.2}  \mbox{ and } \eqref{est:livt} )
\end{align*}
\textbullet \ Estimate of last  term of ${\mathcal F}_{2} :$
\begin{align*}
\bignorm{  R \frac{\widetilde \rho}{\rho_{0}}  \left(Z^{\top} \nabla \widetilde \vartheta \right)_{i}}_{L^{p}(0,T;L^{q})}
\leqslant & \; C \norm{ \widetilde \rho }_{L^{\infty}(0,T;W^{1,q})} \norm{ Z }_{L^{\infty}(0,T;W^{1,q})} \norm{ \nabla \widetilde \vartheta }_{L^{p}(0,T;L^{q})}\\
\leqslant & \; C  T^{\nicefrac{1}{p}}.   \qquad \quad  (\mbox{ using } \eqref{tak0.6}, \eqref{tak0.2}  \mbox{ and } \eqref{est:livt} )
\end{align*}
We deduce \eqref{tak0.5} by noticing that $1/p\leq 1/p'$.

\paragraph{\underline{Estimate of  $\mathcal{F}_{3}$}}
\begin{equation}\label{tak1.0}
\norm{ \mathcal{F}_{3} }_{L^{p}(0,T;W^{1,q}(\ofo))} \leqslant C T^{(1-s)/2p},  \quad s \in (\nicefrac{3}{q}, 1).
\end{equation}
We recall that $\mathcal{F}_{3}$ is defined by \eqref{F3bis}.
The estimate of first five terms of ${\mathcal F}_{3}$ are similar to
estimates of terms of ${\mathcal F}_{2}$  and we skip their proofs. The estimate of the last two terms are similar and we only consider one of these terms:
using \eqref{est:lisv}, \eqref{tak0.3} yields, for $s \in (\nicefrac{3}{q}, 1)$,
\begin{equation} \label{est:div1}
              \bignorm{  \left[ Z Q\right]^\top : \nabla \widetilde u }_{L^{p}(0,T;L^{\infty})}
\leqslant   \norm{ ZQ }_{L^{\infty}(0,T,W^{1,q})}  \norm{ \nabla\widetilde u }_{L^{p}(0,T,L^{\infty})}
\leqslant  C T^{(1-s)/2p}.
\end{equation}
Using \eqref{tak0.3} and \eqref{est:livt}, we obtain
\begin{equation} \label{est:div2}
              \bignorm{  \left[ Z Q\right]^\top : \nabla \widetilde u }_{L^{\infty}(0,T;L^{q})}
\leqslant   \norm{ ZQ }_{L^{\infty}(0,T,W^{1,q})}  \norm{ \nabla\widetilde u }_{L^{\infty}(0,T,L^{q})}
\leqslant  C.
\end{equation}
where the constant $C$ depends only on $M$. Combining \eqref{est:div1} and \eqref{est:div2} we obtain
\begin{align*}
\Bignorm{ \frac{\alpha}{c_{v}\rho_0} \left(\left[ Z Q\right]^\top : \nabla \widetilde u\right)^{2} }_{L^{p}(0,T;L^{q})}
\leqslant &\; C \bignorm{  \left[ Z Q\right]^\top : \nabla \widetilde u }_{L^{p}(0,T;L^{\infty})} \bignorm{  \left[ Z Q\right]^\top : \nabla \widetilde u }_{L^{\infty}(0,T;L^{q})} \\
\leqslant &\; C T^{(1-s)/2p}, \quad s \in (\nicefrac{3}{q}, 1).
\end{align*}

\paragraph{\underline{Estimate of $\mathcal{G}_{1}$ and $\mathcal{G}_{2}$ }}
\begin{equation}\label{tak1.2}
\norm{ \mathcal{G}_1}_{L^{p}(0,T)} + \norm{ \mathcal{G}_2}_{L^{p}(0,T)}  \leqslant C T^{\delta},
\end{equation}
where $\mathcal{G}_1$ and $\mathcal{G}_2$ are defined by \eqref{G1} and \eqref{G2}

We first show that
\begin{equation}\label{tak1.1}
\Bignorm{ \int_{\partial \oso}{\mathcal{G}}_{0}n}_{L^{p}(0,T)} + \Bignorm{ \int_{\partial \oso} y \times {\mathcal{G}}_{0}n }_{L^{p}(0,T)} \leqslant C T^{\delta},
\end{equation}
where $\mathcal{G}_{0}$ is defined by \eqref{G0}.

Using \eqref{est:imp}  and \eqref{tak0.3}
$$
\left \|\nabla \widetilde  u  Z Q \right\|_{L^{p}(0,T;W^{s,q})}
\leq
C \left \|  Z Q \right\|_{L^{\infty}(0,T;W^{1,q})}
\left\| \widetilde  u \right\|_{L^{p}(0,T;W^{1+s,q})}
\leq C T^{(1-s)/2p}, \quad s \in (\nicefrac{1}{q},1).
$$
Using the trace theorem, we deduce that
$$
\left \|\int_{\partial \oso} \nabla \widetilde  u  Z Q \cdot n \ d \gamma \right\|_{L^{p}(0,T)}
\leq C T^{(1-s)/2p}, \quad s \in (\nicefrac{1}{q},1).
$$
The other terms can be estimated similarly.

On the other hand, from \eqref{tak0.7},
$$
\norm{ \widetilde \omega \times \widetilde \ell }_{L^{p}(0,T)} + \norm{J(0) \widetilde \omega \times \widetilde \omega }_{L^{p}(0,T)}
\leqslant C T^{\nicefrac{1}{p}}
$$
and combining this with \eqref{tak1.1}, we deduce \eqref{tak1.2}.

 \paragraph{\underline{Estimate of  $\mathcal{H}_{F}$  and $\mathcal{H}_{S}$}}
 \begin{align}
&\norm{ \mathcal{H}_{F} \cdot n }_{F^{(1-\nicefrac{1}{q})/2}_{p,q}(0,T;L^{q}(\partial \Omega)) \;\cap\; L^{p}(0,T;W^{1-\nicefrac{1}{q},q}(\partial \Omega))} \leqslant C T^{\delta},   \\
& \norm{ \mathcal{H}_{S} \cdot n }_{F^{(1-\nicefrac{1}{q})/2}_{p,q}(0,T;L^{q}(\partial \oso)) \;\cap\; L^{p}(0,T;W^{1-\nicefrac{1}{q},q}(\partial \oso))} \leqslant C T^{\delta},
\end{align}
where $\mathcal{H}_{F}$  and $\mathcal{H}_{S}$ are defined by \eqref{H}.

Recall that, $\mathcal{H}_{F} = (I_3-Z^{\top}) \nabla \widetilde \vartheta$.
Using \eqref{est:all}, \eqref{tak0.2} and \eqref{est:liw1p}, and recalling that $W^{1,q}(\ofo)$ is an algebra we first obtain
\begin{align*}
\bignorm{  \mathcal{H}_{F} \cdot n}_ {L^{p}(0,T;W^{1-\nicefrac{1}{q}}(\partial \Omega))}
\leqslant & \; C \bignorm{ \mathcal{H}_{F} }_ {L^{p}(0,T;W^{1,q}(\ofo))} \\
\leqslant & \; C \bignorm{  (I_3-Z^{\top}) }_{L^{\infty}(0,T;W^{1,q}(\ofo))} \norm{ \nabla \widetilde \vartheta }_{L^{p}(0,T;W^{1,q}(\ofo))}  \\
\leqslant & \; C T^{\nicefrac{1}{p'}}.
\end{align*}
To estimate the 
Lizorkin-Triebel norm of $\mathcal{H}_{F} \cdot n$, we shall use
\cref{prop:TL-product} with parameter $s{=}(1{-}\nicefrac{1}{q})/2$ : for
$U_{1} = U_{3} = L^{q}(\partial \Omega)$,
$U_{2} = W^{1-\nicefrac{1}{q},q}(\partial \Omega)$ and
$\Phi(f, g) = f \cdot g$.  Since $3 < q < \infty$,
$W^{1-\nicefrac{1}{q},q}(\partial \Omega) \hookrightarrow
L^\infty(\partial \Omega)$
and so the hypothesis of the proposition on $\Phi$  are met. Since
$2 < p < \infty$, we also have $s + \nicefrac{1}{p} < 1$. We write
\[
   \mathcal{H}_{F}|_{\partial \Omega} \cdot n = \sum_{j,k} \Bigl[\left(\delta_{j,k} - Z_{j,k}\right) \frac{\partial \widetilde \vartheta}{\partial y_{k}}\Bigr] (t,y) n_{j}(y), \quad y \in \partial\Omega
\]
By \cref{prop:Ntrace} and \eqref{est:all}, there exists a constant $C$ depending only on $M$ such that
\[
\Bignorm{ \frac{\partial \widetilde \vartheta}{\partial y_{k}}  n_{j}}_{F^{(1-\nicefrac{1}{q})/2}_{p,q}(0,T;L^{q}(\partial \Omega))} \leqslant C
\]
for all $1 \leqslant j,k \leqslant 3$. On the other hand, using \eqref{tak0.2}, one has
\[
\bignorm{  \delta_{j,k} - Z_{j,k}}_{W^{1,p}(0,T;W^{1-\nicefrac{1}{q},q}(\partial \Omega))} \leqslant C \bignorm{  \delta_{j,k} - Z_{j,k}}_{W^{1,p}(0,T;W^{1,q}(\ofo))} \leqslant C,
\]
where the constant $C$ depends only on $M$. Finally
$\left(\displaystyle \delta_{j,k} - Z_{j,k}\right)(0,y) = 0$ for all
$1 \leqslant j,k \leqslant 3$. From \cref{prop:TL-product} we obtain
\[
\bignorm{  \mathcal{H}_{F}|_{\partial \Omega} \cdot n }_{F^{(1-\nicefrac{1}{q})/2}_{p,q}(0,T;L^{q}(\partial \Omega))} \leqslant C T^{\delta}.
\]
The estimate of $\mathcal{H}_{S} \cdot n$ is similar.
\end{proof}

\begin{prop} \label{prop:lipestimate}
Let $2 < p < \infty$ and $3 < q < \infty$ satisfying
  the condition
  $ \frac{1}{p} + \frac{1}{2q} \neq \frac12$. Let
  $(\rho_{0},u_{0}, \vartheta_{0},\ell_{0},\omega_{0}) \in
  \mathcal{I}^{cc}_{p,q}$
  such that \eqref{ini-ball} holds. There exists $\widetilde{T}<1$, a constant $\delta > 0$ depending only on $p$ and $q$,
  and a constant $C > 0$ depending only on $p,q,M,\widetilde{T}$
  such that for $T\in (0,\widetilde{T}]$ we have the following property:
for $(f_{1}^{j},f_{2}^{j},f_{3}^{j},h^{j},g_{1}^{j},g_{2}^{j}) \in \mathcal{B}_{T,p,q}$ satisfying
  $$
 \| (f_{1}^j,f_{2}^j,f_{3}^j, h^j,  g_{1}^j,g_{2}^j) \|_{\mathcal{B}_{T,p,q}} \leq 1,
$$
  for $j=1,2$, let
  $(\widetilde \rho^{j}, \widetilde u^{j}, \widetilde \vartheta^{j}, \widetilde \ell^{j}, \widetilde \omega^{j})$
  be the solution of \eqref{com0.0}---\eqref{com0.3} corresponding to
  the source term $(f_{1}^{j},f_{2}^{j},f_{3}^{j},h^{j},g_{1}^{j},g_{2}^{j})$.
  Let us set
\begin{align*}
& \mathcal{F}_{1}^{j} =  \mathcal{F}_{1}(\widetilde\rho^{j},\widetilde u^{j},\widetilde \vartheta^{j},\widetilde \ell^{j},\widetilde \omega^{j}), \quad \mathcal{F}_{2}^{j} =  \mathcal{F}_{2} (\widetilde\rho^{j},\widetilde u^{j},\widetilde \vartheta^{j},\widetilde \ell^{j},\widetilde \omega^{j}), \quad \mathcal{F}_{3}^{j} =  \mathcal{F}_{3}(\widetilde\rho^{j},\widetilde u^{j},\widetilde \vartheta^{j},\widetilde \ell^{j},\widetilde \omega^{j})  \\
& \mathcal{H}_{F}^{j} =  \mathcal{H}_{F}(\widetilde\rho^{j},\widetilde u^{j},\widetilde \vartheta^{j},\widetilde \ell^{j},\widetilde \omega^{j}), \quad \mathcal{H}_{S}^{j} =  \mathcal{H}_{S}(\widetilde\rho^{j},\widetilde u^{j},\widetilde \vartheta^{j},\widetilde \ell^{j},\widetilde \omega^{j}),  \mathcal{H}^{j} = \; \mathbbm{1}_{\partial \Omega} \mathcal{H}_{F}^{j} + \mathbbm{1}_{\partial \oso} \mathcal{H}_{S}^{j}  \\
& \mathcal{G}_{1}^{j} =  \mathcal{G}_{1}(\widetilde\rho^{j},\widetilde u^{j},\widetilde \vartheta^{j},\widetilde \ell^{j},\widetilde \omega^{j}), \quad \mathcal{G}_{2}^{j} =  \mathcal{G}_{2}(\widetilde\rho^{j},\widetilde u^{j},\widetilde \vartheta^{j},\widetilde \ell^{j},\widetilde \omega^{j}),
\end{align*}
Then
\begin{equation*}
  \Big\| (\mathcal{F}_{1}^{1}- {\mathcal{F}}_{1}^{2},{\mathcal{F}}_{2}^{1} - {\mathcal{F}}_{2}^{2},{\mathcal{F}}_{3}^{1} - {\mathcal{F}}_{3}^{2}, \mathcal{H}^{1} -  \mathcal{H}^{2},  \mathcal{G}^{1}_{1} - \mathcal{G}^{2}_{1}, \mathcal{G}_{2}^{1} - \mathcal{G}_{2}^{2}) \Big\|_{\mathcal{B}_{T,p,q}} \leq CT^{\delta}.
  \end{equation*}
\end{prop}

\begin{proof}
The proof of this proposition is similar to the proof of \cref{prop:estimate} and we only give here some few ideas.
From  \eqref{est:thm1} in \cref{thm1}, we first obtain
\begin{multline}
	\norm{  (\widetilde \rho^1, \widetilde u^1, \widetilde \vartheta^1, \widetilde \ell^1, \widetilde \omega^1)
	-(\widetilde \rho^2, \widetilde u^2, \widetilde \vartheta^2, \widetilde \ell^2, \widetilde \omega^2)
	}_{\mathcal{S}_{T,p,q}}
	\\
	\leqslant C
	\bignorm{ (f_{1}^{1},f_{2}^{1},f_{3}^{1},h^{1},g_{1}^{1},g_{2}^{1}) - (f_{1}^{2},f_{2}^{2},f_{3}^{2},h^{2},g_{1}^{2},g_{2}^{2})}_{\mathcal{B}_{T,p,q}}.  \label{est:all-LIP}
\end{multline}

Combining \eqref{est:liw1p} and \eqref{est:all-LIP}, we deduce
\begin{equation} \label{est:d-id-LIP}
   \norm{ \widetilde \rho^1- \widetilde \rho^2 }_{L^{\infty}(0,T;W^{1,q}(\ofo))} \leqslant C T^{\nicefrac{1}{p'}}
   \bignorm{ (f_{1}^{1},f_{2}^{1},f_{3}^{1},h^{1},g_{1}^{1},g_{2}^{1}) - (f_{1}^{2},f_{2}^{2},f_{3}^{2},h^{2},g_{1}^{2},g_{2}^{2})}_{\mathcal{B}_{T,p,q}}.
\end{equation}
In a similar manner, we can obtain
\begin{multline}\label{tak0.7-LIP}
   \norm{ \widetilde \omega^1-\widetilde \omega^2   }_{L^{\infty}(0,T)}+ \norm{ \widetilde \ell^1-\widetilde \ell^2 }_{L^{\infty}(0,T)}
   \\
   \leqslant  C T^{\nicefrac{1}{p'}}
   \bignorm{ (f_{1}^{1},f_{2}^{1},f_{3}^{1},h^{1},g_{1}^{1},g_{2}^{1}) - (f_{1}^{2},f_{2}^{2},f_{3}^{2},h^{2},g_{1}^{2},g_{2}^{2})}_{\mathcal{B}_{T,p,q}}.
\end{multline}

Since $2 < p < \infty$, one has
$B^{2(1-\nicefrac{1}{p})}_{q,p}(\ofo) \hookrightarrow W^{1,q}(\ofo)$. Therefore,
using \cref{prop:LiBesov} and \eqref{est:all-LIP} we get
\begin{multline}
\norm{ \widetilde u^1-\widetilde u^2 }_{L^{\infty}(0,T;W^{1,q})}+  \norm{ \widetilde \vartheta^1-\widetilde \vartheta^2  }_{L^{\infty}(0,T;W^{1,q})}
\\
\leqslant  C
	\bignorm{ (f_{1}^{1},f_{2}^{1},f_{3}^{1},h^{1},g_{1}^{1},g_{2}^{1}) - (f_{1}^{2},f_{2}^{2},f_{3}^{2},h^{2},g_{1}^{2},g_{2}^{2})}_{\mathcal{B}_{T,p,q}}. \label{est:livt-LIP}
\end{multline}

Proceeding as in the proof of \cref{prop:estimate}, we can then deduce
\begin{multline}\label{est:lisv-LIP}
\norm{    \widetilde u^1-\widetilde u^2  }_{L^{p}(0,T,L^{\infty})}
+ \norm{ \nabla\widetilde u^1-\nabla \widetilde u^2  }_{L^{p}(0,T,L^{\infty})}
+\norm{ \nabla \widetilde \vartheta^1-\nabla \widetilde \vartheta^2  }_{L^{p}(0,T,L^{\infty})}
\\
\leqslant C T^{(1-s)/2p}
\bignorm{ (f_{1}^{1},f_{2}^{1},f_{3}^{1},h^{1},g_{1}^{1},g_{2}^{1}) - (f_{1}^{2},f_{2}^{2},f_{3}^{2},h^{2},g_{1}^{2},g_{2}^{2})}_{\mathcal{B}_{T,p,q}},
\quad s \in (\nicefrac{3}{q},1).
\end{multline}

Let us denote by $Q^1$ and $Q^2$ the solution of \eqref{def-Q} associated with $\widetilde \omega^1$, $\widetilde \omega^2$. Then $Q^1-Q^2$ satisfies
\begin{equation} \label{def-Q-LIP}
\begin{cases}
\dfrac{d}{dt}\left( Q^1-Q^2 \right)  = \left( Q^1-Q^2 \right)  A(\widetilde  \omega^1) + Q^2  A(\widetilde  \omega^1-\widetilde  \omega^2) & \mbox{ in } (0,T), \\
\left( Q^1-Q^2 \right)(0) = 0,
\end{cases}
\end{equation}
and thus from \eqref{tak0.7}, \eqref{tak0.7-LIP} and Gr\"onwall's lemma, we obtain
\begin{equation} \label{est:Q-LIP}
\norm{ Q^1-Q^2 }_{L^{\infty}(0,T; \RR^{3\times3})} \leqslant CT \bignorm{ (f_{1}^{1},f_{2}^{1},f_{3}^{1},h^{1},g_{1}^{1},g_{2}^{1}) - (f_{1}^{2},f_{2}^{2},f_{3}^{2},h^{2},g_{1}^{2},g_{2}^{2})}_{\mathcal{B}_{T,p,q}}.
\end{equation}

Let $X^1$, $X^2$ be defined as in \eqref{Jacobi} with $(Q^1,\widetilde u^1)$ and  $(Q^2,\widetilde u^2)$. Then
\[
\partial_{t} \nabla (X^1-X^2)  = \left(Q^1-Q^2\right) \nabla \widetilde u^1+Q^2 \nabla \left(\widetilde u^1-\widetilde u^2\right),\quad
\nabla (X^1-X^2)(0,\cdot)=0.
\]
and from \eqref{est:all}, \eqref{est:Q}, \eqref{est:all-LIP} and \eqref{est:Q-LIP} we get
\begin{multline*}
\norm{ \partial_{t} \nabla (X^1-X^2) }_{L^{p}(0,T;W^{1,q})}
\\
\leqslant
\norm{ Q^1-Q^2 }_{L^{\infty}(0,T)} \norm{ \nabla \widetilde u^2 }_{L^{p}(0,T;W^{1,q})}
+
\norm{ Q^2 }_{L^{\infty}(0,T)} \norm{ \nabla \left(\widetilde u^1-\widetilde u^2\right) }_{L^{p}(0,T;W^{1,q})}
\\
\leqslant C \bignorm{ (f_{1}^{1},f_{2}^{1},f_{3}^{1},h^{1},g_{1}^{1},g_{2}^{1}) - (f_{1}^{2},f_{2}^{2},f_{3}^{2},h^{2},g_{1}^{2},g_{2}^{2})}_{\mathcal{B}_{T,p,q}}.
\end{multline*}
The above estimate yields
\begin{multline*}
\norm{  \nabla (X^1-X^2)  }_{W^{1,p}(0,T;W^{1,q}(\ofo))}
\\
\leqslant
C \bignorm{ (f_{1}^{1},f_{2}^{1},f_{3}^{1},h^{1},g_{1}^{1},g_{2}^{1}) - (f_{1}^{2},f_{2}^{2},f_{3}^{2},h^{2},g_{1}^{2},g_{2}^{2})}_{\mathcal{B}_{T,p,q}}.
\end{multline*}
the above estimate and \eqref{est:liw1p} yield
\begin{multline*}
 \norm{ \nabla (X^1-X^2)  }_{L^{\infty}(0,T;W^{1,q}(\ofo))}
\\
\leqslant
C T^{1/p'} \bignorm{ (f_{1}^{1},f_{2}^{1},f_{3}^{1},h^{1},g_{1}^{1},g_{2}^{1}) - (f_{1}^{2},f_{2}^{2},f_{3}^{2},h^{2},g_{1}^{2},g_{2}^{2})}_{\mathcal{B}_{T,p,q}}.
\end{multline*}
The rest of the proof runs as the proof of \cref{prop:estimate}.
\end{proof}

\part{Global in Time Existence }

\section{Linearization and Lagrangian Change of Variables} \label{sec:cov-g}
In this section, we slightly modify the change of variables introduced in \cref{sec:cov-local} and we rephrase the global existence and uniqueness result in \cref{mainthm_glob} in terms of the functions issued from this change of variables. The reason of  this modification  is that, here we need to linearize the system around the constant steady state $(\overline\rho, 0, \overline \vartheta, 0, 0).$ More precisely, define
\begin{gather}
\widetilde \rho(t,y)  = \rho(t,X(t,y))  - \overline\rho, \qquad \widetilde u (t,y)   = Q^{-1}(t)u(t,X(t,y)),  \\
\widetilde \vartheta(t,y)  = \vartheta(t,X(t,y)) - \overline\vartheta, \qquad \widetilde  p  = R \widetilde \rho \widetilde \vartheta , \\
\widetilde \ell(t)   = Q^{-1}(t) \dot a(t), \qquad \widetilde \omega(t)  = Q^{-1}(t) \omega(t),
\end{gather}
for $(t,y) \in (0,\infty) \times \ofo,$ where  $X$ has been  defined as in \eqref{ode}.

Then $(\widetilde  \rho, \widetilde  u, \widetilde  \vartheta, \widetilde  \ell, \widetilde  \omega)$ satisfies the following system
\begin{alignat}{2} \label{sys:NL-G}
&\partial_{t} \widetilde\rho + \overline \rho \div\widetilde u = {\mathcal F}_{1}  &\mbox{ in } (0,\infty) \times \ofo, \notag \\
& \partial_{t} \widetilde u - \div\sigma_{l}(\widetilde \rho,\widetilde u,\widetilde\vartheta) = {\mathcal F}_{2} &  \mbox{ in } (0,\infty) \times \ofo, \notag \\
& \partial_{t} \widetilde \vartheta  - \frac{\kappa}{\overline\rho c_{v}} \Delta \widetilde\vartheta + \frac{R \overline \vartheta}{c_{v}} \div\widetilde u = {\mathcal F}_{3} & \mbox{ in } (0,\infty) \times \ofo, \notag \\
&\widetilde  u = 0  & \mbox{ on } (0,\infty) \times \partial\Omega, \notag \\
&\widetilde u = \widetilde\ell + \widetilde\omega \times y & \mbox{ on } (0,\infty) \times \partial \oso \\
& \frac{\partial \widetilde \vartheta}{\partial n} = {\mathcal H} \cdot n & \mbox{ on } (0,\infty) \times \partial \ofo, \notag \\
& \frac{d}{dt} \widetilde \ell = - m^{-1} \int_{\partial\oso} \sigma_{l} (\widetilde\rho,\widetilde u,\widetilde\vartheta) n \ d\gamma + {\mathcal G}_{1} & \quad t \in (0,\infty) \notag \\
&\frac{d}{dt} \omega = - J(0)^{-1} \int_{\partial\oso} y \times \sigma_{l} (\rho,u,\vartheta) n \ d\gamma  + {\mathcal G}_{2} & t \in (0,\infty)  \notag \\
&   \widetilde\rho(0) = \rho_{0} - \overline\rho , \quad u(0) = u_{0},  \quad  \vartheta(0) = \vartheta_{0} - \overline \vartheta  & \mbox{ in } \ofo, \notag \\
&   \ell(0) = \ell_{0},  \quad \omega(0) = \omega_{0},\notag
\end{alignat}
where
\begin{align}
\sigma_{l}( \widetilde \rho,  \widetilde u,  \widetilde \vartheta) = \frac{2\mu}{\overline\rho} D \widetilde u + \left(\frac{\alpha}{\overline\rho} \div \widetilde u  - \frac{R\overline \vartheta}{\overline\rho}  \widetilde \rho  -R  \widetilde \vartheta \right) I_{3}, \quad D( \widetilde u) = \frac12(\nabla  \widetilde u + \nabla  \widetilde u ^{T}),
\end{align}
\begin{align} \label{def-Q-g}
\dot Q  =  Q A(\widetilde  \omega), \quad Q(0) = I_{3}
\end{align}
and
\begin{align} \label{Jacobi-g}
X(t,y)  = y + \int_{0}^{t} Q(s) \widetilde  u(s,y) \ {\rm d}s, \quad \mbox{ and } \quad \nabla Y(t,X(t,y)) = [\nabla X]^{-1}(t,y),
\end{align}
for every $y\in \Omega_F(0)$ and $t\geqslant 0$. Using  the notation
\begin{align} \label{Z-g}
Z(t,y) = \left( Z_{i,j}\right)_{1\leqslant i, j\leqslant 3}=[\nabla X]^{-1}(t,y) \qquad\qquad(t\geqslant 0,\ y\in \Omega_F(0)),
\end{align}
the remaining terms in \eqref{sys:NL-G} are defined by
\begin{equation} \label{F1-g}
\mathcal{F}_{1}(\widetilde \rho,\widetilde  u,\widetilde  \vartheta,\widetilde  \ell,\widetilde \omega)
= - \widetilde \rho \div\widetilde u  - (\widetilde \rho+\overline\rho) (Z^{\top} - I_{3}) : \nabla \widetilde u
\end{equation}
\begin{flalign}\label{F2-g}
(\mathcal{F}_{2})_i(\widetilde \rho,\widetilde  u,\widetilde  \vartheta,\widetilde  \ell,\widetilde \omega)=&  -\frac{\widetilde  \rho + \overline\rho}{\overline\rho}(\widetilde \omega\times Q \widetilde  u)_i +\widetilde \rho (\partial_t \widetilde  u)_{i}  -  \frac{\widetilde  \rho + \overline\rho}{\overline\rho} \left[ (Q - I) \partial_{t} \widetilde  u\right]_{i} \notag \\
& + \frac{\mu}{\overline\rho}\sum_{l,j,k} \frac{\partial^2 (Q\widetilde  u)_{i} }{\partial y_l\partial y_k}  \left(Z_{k,j}  - \delta_{k,j}\right) Z_{l,j}  + \frac{\mu}{\overline\rho}\sum_{l,k} \frac{\partial^2 (Q\widetilde  u)_{i} }{\partial y_l\partial y_k} \left( Z_{l,k} - \delta_{l,k} \right)
\notag \\
& + \frac{\mu}{\overline\rho} \left[ (Q - I) \Delta \widetilde  u\right]_{i} +\frac{\mu}{\overline\rho}\sum_{l,j,k} Z_{l,j} \frac{\partial  (Q \widetilde  u)_i}{\partial y_k} \frac{\partial Z_{k,j}}{\partial y_{l}}
\notag \\
& +\frac{\mu+\alpha}{\overline\rho} \sum_{l,j,k} \frac{\partial^{2} (Qu)_{j}}{\partial y_{l} \partial y_{k}} \left( Z_{k,j} - \delta_{k,j}\right)  Z_{l,i}   +
\frac{\mu+\alpha}{\overline\rho} \sum_{l,j} \frac{\partial^{2} (Qu)_{j}}{\partial y_{l} \partial y_{j}} \left( Z_{l,i} - \delta_{l,i}\right)
\notag \\
& + \frac{\alpha + \mu}{\overline\rho} \frac{\partial}{\partial y_{i}} \left[ \nabla \widetilde  u : (Q^{\top} - I_{3}) \right] +
\frac{\alpha + \mu}{\overline\rho}\sum_{l,j,k} Z_{l,i} \frac{\partial (Qu)_{j}}{\partial y_{k}} \frac{\partial Z_{k,j}}{\partial y_{l}}
\notag \\
& - R \frac{\widetilde  \vartheta}{\overline\rho}  \left(Z^{\top} \nabla \widetilde  \rho \right)_{i}  -  R \frac{\widetilde  \rho}{\overline\rho}  \left(Z^{\top} \nabla \widetilde  \vartheta \right)_{i} - \frac{R\overline \vartheta}{\overline\rho}\left[ \left( Z^{\top} - I\right) \nabla \widetilde \rho\right]_{i} \notag \\
& - R\left[ \left( Z^{\top} - I\right) \nabla \widetilde \vartheta\right]_{i}, &
\end{flalign}

\begin{flalign}\label{F3-G}
\mathcal{F}_{3}(\widetilde \rho,\widetilde  u,\widetilde  \vartheta,\widetilde  \ell,\widetilde \omega) = & - \frac{\widetilde \rho}{\overline\rho}\partial_t \widetilde  \vartheta
-\frac{R}{c_v \overline\rho} \left(\widetilde \vartheta\widetilde \rho +  \overline\rho \widetilde\vartheta + \overline \vartheta \widetilde \rho\right) \left[ ZQ\right]^\top : \nabla \widetilde  u  - \frac{R\overline \vartheta}{c_{v}}  \left(\left[ ZQ\right]^\top - I_{3} \right) : \nabla \widetilde  u \notag \\
& +\frac{\kappa}{c_v \overline\rho} \sum_{l,j,k} Z_{l,j}\frac{\partial \widetilde  \vartheta}{\partial y_k} \frac{\partial Z_{k,j}}{\partial y_l}
+\frac{\kappa}{c_v \overline\rho} \sum_{k,l} \frac{\partial^2 \widetilde  \vartheta}{\partial y_k\partial y_l} \left(Z_{l,k} -\delta_{l,k} \right)
Z_{l,j} \notag \\
& + \frac{\kappa}{c_v \overline\rho} \sum_{j,k,l} \frac{\partial^2 \widetilde  \vartheta}{\partial y_k\partial y_l} \left(Z_{k,j}-\delta_{k,j} \right)Z_{l,j}
\notag \\
& + \frac{\alpha}{c_{v}\overline\rho} \left(\left[ ZQ\right]^\top : \nabla \widetilde  u\right)^{2}
+ \frac{\mu}{2 c_{v} \overline\rho} \left|\nabla \widetilde  u Z Q + \left(\nabla \widetilde  u Z Q\right)^{\top}\right|^2, &
\end{flalign}
\begin{flalign} \label{H-g}
& \mathcal{H}(\widetilde \rho,\widetilde  u,\widetilde  \vartheta,\widetilde  \ell,\widetilde \omega) = \; \mathbbm{1}_{\partial \Omega} \mathcal{H}_{F} + \mathbbm{1}_{\partial \oso} \mathcal{H}_{S} \notag \\
& \mathcal{H}_{F} (\widetilde \rho,\widetilde  u,\widetilde  \vartheta,\widetilde  \ell,\widetilde \omega) = (I_3- Z^{\top}) \nabla \widetilde  \vartheta , \quad \mathcal{H}_{S}  (\widetilde \rho,\widetilde  u,\widetilde  \vartheta,\widetilde  \ell,\widetilde \omega) =  (I_3 - (Z Q)^{\top})\nabla \widetilde  \vartheta, &
\end{flalign}
\begin{flalign} \label{G0-g}
 \mathcal{G}_{0} (\widetilde \rho,\widetilde  u,\widetilde  \vartheta,\widetilde  \ell,\widetilde \omega) \; =  & \; \frac{\mu}{\overline\rho} \left[\nabla \widetilde  u  (Z Q - I_{3}) + [(ZQ)^{\top} - I_{3}]\left(\nabla \widetilde  u  \right)^{\top}   \right] \notag \\
 & +\frac{\alpha}{\overline\rho} \left((\left[Z  Q\right]^\top - I_{3}) : \nabla \widetilde  u\right) I_3
 +  R   \widetilde  \rho \widetilde  \vartheta I_3, &
\end{flalign}
\begin{flalign} \label{G-g}
&\mathcal{G}_{1} (\widetilde \rho,\widetilde  u,\widetilde  \vartheta,\widetilde  \ell,\widetilde \omega) = - \frac{m}{\overline\rho} (\widetilde  \omega \times \widetilde  \ell) -   \int_{\partial \oso} {\mathcal G}_{0} n \ {\rm d} \gamma, \notag \\
& \mathcal{G}_{2}(\widetilde \rho,\widetilde  u,\widetilde  \vartheta,\widetilde  \ell,\widetilde \omega) =  \frac{J(0)}{\overline\rho} \widetilde  \omega \times \widetilde  \omega  -  \int_{\partial \oso} y \times \mathcal{G}_{0} n \ {\rm d}\gamma. &
 \end{flalign}

Using the above change of variables, \cref{mainthm_glob} can be rephrased as follows.
\begin{thm}  \label{thm:g-f-2}
Let $2 < p < \infty$ and $3 < q < \infty$ satisfying the condition $\displaystyle \frac{1}{p} + \frac{1}{2q} \neq \frac12$. Assume that \eqref{nocontact} is satisfied.  Let $\overline \rho > 0$ and $\overline \vartheta > 0$ be two given constants. Then there exists $\eta_{0} > 0$ such that, for all $\eta \in (0,\eta_{0})$
 there exist two constants  $\delta_{0} > 0$ and $C > 0$ such that, for all $\delta \in (0,\delta_{0})$ and 
 for any $(\rho_{0},u_{0}, \vartheta_{0},\ell_{0},\omega_{0})$ belongs to $\mathcal{I}^{cc}_{p,q}$ satisfying
 \begin{align} \label{eq:br-g}
\frac{1}{|\ofo|} \int_{\ofo} \rho_{0} \ {\rm d} x = \overline\rho,
\end{align}
and
\begin{align} \label{eq:ini-g}
\|(\rho_{0} - \overline\rho, u_{0}, \vartheta_{0} - \overline \vartheta, \ell_{0},\omega_{0})\|_{\mathcal{I}_{p,q}} \leqslant \delta,
\end{align}
the system \eqref{sys:NL-G} - \eqref{G0-g} admits a unique solution
$(\widetilde \rho, \widetilde u,\widetilde \vartheta,\widetilde \ell,\widetilde \omega)$ with
\begin{multline} \label{est:g}
\|\widetilde\rho \|_{L^{\infty}(0,\infty; W^{1,q}(\ofo))} + \|e^{\eta (\cdot)}\nabla\widetilde \rho\|_{W^{1,p}(0,\infty; L^{q}(\ofo))}+ \norm{ e^{\eta (\cdot)} \partial_{t} \widetilde\rho}_{L^{p}(0,\infty; L^{q}(\Omega_{F}(0)))}  \\+ \|e^{\eta (\cdot)} \ \widetilde u\|_{L^{p}(0,\infty; W^{2,q}(\Omega_{F}(0))^{3})}
+  \|e^{\eta (\cdot)} \partial_{t}\widetilde u\|_{L^{p}(0,\infty; L^{q}(\Omega_{F}(0))^{3})} + \|e^{\eta (\cdot)} \widetilde u\|_{L^{\infty}(0,\infty; B^{2(1-1/p)}_{q,p}(\Omega_{F}(0))^{3})} \\
+ \|e^{\eta(\cdot)} \partial_{t}\widetilde \vartheta \|_{L^{p}(0,\infty; L^{q}(\Omega_{F}(0)))}   + \norm{e^{\eta(\cdot)} \nabla \widetilde \vartheta }_{L^{p}(0,\infty; L^{q}(\Omega_{F}(0)))} + \norm{e^{\eta(\cdot)} \nabla^{2} \widetilde \vartheta}_{L^{p}(0,\infty; L^{q}(\Omega_{F}(0)))} \\
+ \| \widetilde \vartheta \|_{L^{\infty}(0,\infty; B^{2(1-1/p)}_{q,p}(\Omega_{F}(0)))}
+ \|e^{\eta (\cdot)} \widetilde \ell \|_{W^{1,p}(0,\infty;\mathbb{R}^{3})}  \\
+ \|e^{\eta (\cdot)} \ \widetilde \omega \|_{W^{1,p}(0,\infty;\mathbb{R}^{3})} \leqslant C \delta.
\end{multline}
Moreover, $X \in L^{\infty}(0,\infty;W^{2,q}(\ofo))^{3} \cap W^{1,\infty}(0,\infty;W^{1,q}(\ofo))$ and 
$ X(t,\cdot) : \ofo\to \oft $ is a  $C^1$-diffeormorphim for all  $t\in [0,\infty).$
\end{thm}

\section{Some Background on $\mathcal{R}$ Sectorial Operators}\label{sec_back}

In this section, we recall some definitions and results on maximal
$L^p$--regularity and $\mr$-boundedness.
In what follows, we use Rademacher random variables, that is symmetric random variables with value in $\{-1,1\}$.
We first recall the notion of $\mr$-boundedness.

\begin{definition} [$\mr$-bounded family of operators]
  Let $\mathcal{X}$ and $\mathcal{Y}$ be Banach spaces. A family of operators $\mathcal{T}
  \subset \mathcal{L}(\mathcal{X},\mathcal{Y})$ is called $\mr-$bounded if there exist
  $p\in [1,\infty)$ and a constant $C>0$, such that for any integer $N \ge 1$,
  any $T_1, \ldots T_N \in  \mathcal{T}$,
  any independent Rademacher random variables $r_1, \ldots, r_N$,
  and any $x_1, \ldots, x_N \in \mathcal{X}$,
\[
   \left(\EE \Bignorm{ \sum_{j=1}^{N} r_{j} T_{j} x_{j}}_\mathcal{Y}^p\right)^{1/p} \leq C
   \left(\EE \Bignorm{ \sum_{j=1}^{N} r_{j} x_{j}}_\mathcal{X}^p\right)^{1/p}.
\]
The smallest constant $C$ in the above inequality is called the
$\mr_p$-bound of $\mathcal{T}$ on $\mathcal{L}(\mathcal{X},\mathcal{Y})$ and is denoted by
$\mr_p(\mathcal{T})$. As usual we denote by $\EE$ the expectation. 
\end{definition}

For more information on $\mr$-boundedness we refer to
\cite{ClePruss,DenkHieberPruss,KunstmannWeis:Levico} and references therein.
In particular, it is proved in \cite[p.26]{DenkHieberPruss} that this definition is independent of $p\in [1,\infty)$.

We also recall some useful properties (see Proposition 3.4 in \cite{DenkHieberPruss}):
\begin{equation}  \label{eq:Rbdd-1}
    \mr_p(\mathcal{S} + \mathcal{T}) \leqslant \mr_p(\mathcal{S} ) +  \mr_p( \mathcal{T}),
  \quad
  \mr_p(\mathcal{S} \mathcal{T} ) \leqslant \mr_p(\mathcal{T} )   \mr_p( \mathcal{S}).
\end{equation}


For any $\beta \in (0,\pi)$, we write
\[
   \Sigma_{\beta} = \{ \lambda \in \mathbb{C} \setminus \{0\} \mid   |\arg(\lambda)| < \beta \}.
\]
We now come to the second central definition.
\begin{defin}[sectorial and $\mr$-sectorial operators]
Let $A$ be a densely defined closed linear operator on a Banach space $\mathcal{X}$ with domain $\mathcal{D}(A)$.
We say that $A$ is a sectorial operator of angle $\tau \in(0, \pi)$
if for any $\beta\in (\tau,\pi)$,
$\Sigma_{\pi-\beta} \subset \rho(A)$ and
\[
  R_\beta =  \left\{ \lambda(\lambda - A)^{-1} \SUCHTHAT \lambda   \in \Sigma_{\pi-\beta} \right\}
\]
is bounded  in $\mathcal{L}(\mathcal{X})$. In that case, we write
$$
M_\beta(A) = \sup_{\lambda  \in \Sigma_{\pi-\beta}} \norm{ \lambda(\lambda - A)^{-1} }_{\mathcal{L}(\mathcal{X})}.
$$
Analogously, we say that $A$ is a $\mr$-sectorial operator of angle $\tau$
if $A$ is a sectorial operator of angle $\tau$ and if for any $\beta \in (\tau, \pi)$,
$R_\beta$ is $\mr$-bounded. We denote $\mr_{p, \beta}(A)$ the $\mr_p$-bound of $R_\beta$.
\end{defin}

One can replace in the above definitions $R_\beta$ by
$$
\widetilde{R_\beta}= \left\{ A(\lambda - A)^{-1} \SUCHTHAT \lambda   \in \Sigma_{\pi-\beta} \right\}.
$$
In that case, we denote the uniform bound and the $\mr$-bound by
$\widetilde{M}_\beta(A)$ and $\widetilde{\mr_{p, \beta}}(A)$.

The importance of $\mr$-sectorial operators is explained by the following result:
\begin{thm}[Weis]\label{thm:weis-Lp-maxreg-char}
Let $\mathcal{X}$ be a UMD Banach space and $A$ a densely defined, closed linear operator on $\mathcal{X}$.
Then the following assertions are equivalent
\begin{enumerate}
\item For any $T\in \mathbb{R}_+^*$, $f\in L^p(0,T;\mathcal{X})$
\begin{equation} \label{eq:max0}
u' = A u + f \quad \text{in} \quad (0,T), \quad u(0) = 0
\end{equation}
admits a unique solution $u$ satisfying the above equation almost everywhere and such that $Au\in L^{p}(0,T;\mathcal{X}).$
\item $A$ is $\mr$-sectorial of angle $\tau < \nicefrac{\pi}{2}$.
\end{enumerate}
\end{thm}
This result is due to \cite{Weis01} (see also \cite[p.45]{DenkHieberPruss}). We recall that $\mathcal{X}$ is a UMD Banach space if the Hilbert transform is bounded in $L^p(\mathbb{R};\mathcal{X})$ for $p\in (1,\infty)$. In particular, the closed subspaces of $L^q(\Omega)$ for $q\in (1,\infty)$ are UMD Banach spaces. We refer the reader to \cite[pp.141--147]{Amann} for more information on UMD spaces.

We can also add an initial condition in \eqref{eq:max0} and consider the following system:
\begin{equation} \label{eq:max-reg-g}
u' = A u + f \quad \text{in}\quad  (0,\infty), \quad u(0) = u_{0}.
\end{equation}

 \begin{cor} \label{thm:max-reg-g}
 Let $\mathcal{X}$ be a UMD Banach space, $1 < p < \infty$ and let $A$ be a closed, densely defined operator in $\mathcal{X}$ with domain $\mathcal{D}(A).$
 Let us assume that $A$ is a $\mr$-sectorial operator of angle $\tau < \nicefrac{\pi}{2}$ and that the semigroup generated by $A$ has negative exponential type.
 Then for every
 $u_{0} \in (\mathcal{X}, \mathcal{D}(A))_{1-1/p,p}$ and for every $f \in L^{p}(0,\infty;\mathcal{X}),$
 \cref{eq:max-reg-g} admits a unique solution in $L^{p}(0,\infty;\mathcal{D}(A)) \cap W^{1,p}(0,\infty;\mathcal{X}).$
 \end{cor}

\begin{proof}
The proof follows from the above theorem, \cite[Theorem 2.4]{Dor91} and  \cite[Theorem 1.8.2]{Tri95}.
\end{proof}

In view of the above results, it is natural to consider the perturbation theory of
$\mr$-sectoriality. The following  result was obtained in \cite[Corollary~2]{KunstmannWeis:Pisa2001}.
\begin{prop} \label{pr:perturb}
Let $A$ be a $\mr$-sectorial operator of angle $\tau$
on a Banach space $\mathcal{X}$ and let $\beta \in (\tau, \pi)$.
Let $B$ be  a linear  operator on $\mathcal{X}$ such that $\mathcal{D}(A) \subset \mathcal{D}(B)$ and
\begin{equation} \label{c:small}
    \norm{ B x  }_{\mathcal{X}} \leq a  \norm{  Ax }_{\mathcal{X}} + b \norm{ x }_{\mathcal{X}},
\end{equation}
for some $a,b\geq 0$.
If $a < \left(\widetilde M_\beta(A) \widetilde{\mr_{p, \beta}}(A)\right)^{-1}$ then
$A+B -\lambda$ is $\mr$-sectorial for each
\[
  \lambda >  \lambda_0 = \frac{b M_\beta(A)
  \widetilde{\mr_{p, \beta}}(A)} {1-a \widetilde{M}_\beta(A) \widetilde{\mr_{p, \beta}}(A)}.
\]
\end{prop}


\section{ Linearized Fluid-Structure Interaction System} \label{sec:lin-FSI}

In this section we study the fluid-structure system linearized around
$(\overline \rho, 0, \overline \vartheta, 0, 0)$, $\overline \rho > 0$, $\overline \vartheta > 0$.
More precisely, we consider the following linear system
\begin{alignat}{2} \label{eq:linearS}
&\partial_{t} \rho + \overline \rho \div u = 0  & \mbox{ in } (0,\infty) \times \ofo, \notag \\
& \partial_{t} u - \div\sigma_{l}(\rho,u,\vartheta) = 0 &  \mbox{ in } (0,\infty) \times \ofo, \notag \\
& \partial_{t} \vartheta  - \frac{\kappa}{\overline\rho c_{v}} \Delta \vartheta + \frac{R \overline \vartheta}{c_{v}} \div u = 0 & \mbox{ in } (0,\infty) \times \ofo, \notag \\
& u = \ell + \omega \times y  & \mbox{ on } (0,\infty) \times \partial \oso \\
& u = 0 & \mbox{ on } (0,\infty) \times \partial\Omega, \notag \\
& \frac{\partial \vartheta}{\partial n} = 0  & \mbox{ on } (0,\infty) \times \partial\ofo, \notag \\
& \frac{d}{dt} \ell = - m^{-1} \int_{\partial\oso} \sigma_{l} (\rho,u,\vartheta) n \ d\gamma  &  \quad t \in (0,\infty) \notag \\
&\frac{d}{dt} \omega = - J(0)^{-1} \int_{\partial\oso} y \times \sigma_{l} (\rho,u,\vartheta) n \ d\gamma & \quad t \in (0,\infty)  \notag \\
&   \rho_{0} = \rho_{0} , \quad u(0) = u_{0},  \quad  \vartheta(0) = \vartheta_{0}  & \quad \mbox{ in } \ofo, \notag \\
&   \ell(0) = \ell_{0},  \quad \omega(0) = \omega_{0},\notag
\end{alignat}
where
\begin{equation}\label{tak1.5}
\sigma_{l}(\rho, u, \vartheta) = \frac{2\mu}{\overline\rho} Du + \left(\frac{\alpha}{\overline\rho} \div u  - \frac{R\overline \vartheta}{\overline\rho} \rho  -R \vartheta \right) I_{3}, \quad D(u) = \frac12(\nabla u + \nabla u ^{T}).
\end{equation}

Our aim is to show that the linearized operator is $\mr$-sectorial in a suitable functional space.
In order to do this, we first consider the case of a linearized compressible Navier-Stokes-Fourier system without rigid body and we use \cref{pr:perturb} in order to deal with the equations for the rigid body.

\subsection{Linearized compressible Navier-Stokes-Fourier system} \label{subsec:cns-L}
In this subsection we discuss some properties of the generator of the semigroup describing the linearization around an equilibrium state of the Navier-Stokes-Fourier system.  Most of these properties follow from the corresponding linearized compressible Navier-Stokes system and of the heat equation, and in this case  we just provide the appropriate references.  Our contribution is to show that the coupling terms can be seen as perturbations and thus tackled using either direct estimates or abstract perturbation results for $\mr$-sectorial operators.

 Let us set
\begin{align}
\mathcal{X} = W^{1,q}(\ofo) \times L^{q}(\ofo)^{3} \times L^{q}(\ofo)
\end{align}
and consider the operator $A_{F}:\mathcal{D}(A_{F}) \to \mathcal{X}$ defined by
\begin{multline*}
\mathcal{D}(A_{F}) = \Bigg\{ (\rho,u,\vartheta)  \in W^{1,q}(\ofo) \times W^{2,q}(\ofo)^{3} \times W^{2,q}(\ofo)  \mid
\\
u=0 \mbox{ on } \partial \ofo, \quad \frac{\partial \vartheta}{\partial n} = 0 \mbox{ on }  \partial \ofo
\Bigg\},
\end{multline*}

\begin{align}
A_{F} = \begin{pmatrix}
0 & \displaystyle - \overline\rho \div & 0 \\
- \displaystyle \frac{R\overline \vartheta}{\overline\rho} \nabla &\displaystyle \frac{\mu}{\overline\rho}\Delta + \frac{\alpha + \mu}{\overline\rho} \nabla \div  & -R \nabla  \\
 0 &  \displaystyle -\frac{R\overline \vartheta}{c_{v}} \div & \displaystyle \frac{\kappa}{\overline\rho c_{v}} \Delta
\end{pmatrix}.
\end{align}
  Let us first study some properties of the operator $A_{F}$.
 \begin{thm} \label{thm:rsec-af}
Assume $1 <q< \infty$. Then there exists $\gamma_{0} > 0$ such that  $A_{F} - \gamma_{0}$
is an ${\mathcal R}$-sectorial operator in $\mathcal{X}$ of angle $< \pi/2.$
\end{thm}

\begin{proof}
We first define
\begin{equation}\label{Au}
\mathcal{D}(A_{u}) = (W^{2,q}(\ofo) \cap W^{1,q}_{0}(\ofo))^{3},
\quad
A_{u} = \frac{\mu}{\overline\rho}\Delta + \frac{\alpha + \mu}{\overline\rho} \nabla \div,
\end{equation}
and
\begin{equation}\label{Atheta}
\dom(A_{\vartheta}) = \left\{ \vartheta \in W^{2,q}(\ofo) \mid \frac{\partial \vartheta}{\partial n} = 0\mbox{ on } \partial \ofo\right\},
\quad A_{\vartheta} = \frac{\kappa}{\overline \rho c_{v}} \Delta.
\end{equation}

From \cite[Theorem 8.2]{DenkHieberPruss}, there exists $\gamma_\vartheta\in \mathbb{R}$ such that $A_{\vartheta}-\gamma_\vartheta$ is $\mr$-sectorial of angle $< \nicefrac{\pi}{2}$.
Using \cite[Theorem 2.5]{ShibaEno13}, we also obtain the existence of
$\gamma_u\in \mathbb{R}$ such that $A_{u}-\gamma_u$ is $\mr$-sectorial of angle $< \nicefrac{\pi}{2}$.

In particular there exist $\gamma$ and $\beta<\pi/2$ such that,
\begin{equation}\label{scl01}
\mr_{p} \left\{ \lambda(\lambda - A_{u})^{-1} \SUCHTHAT \lambda   \in \gamma+\Sigma_{\pi-\beta} \right\}<\infty,
\quad
\mr_{p} \left\{ A_{u}(\lambda - A_{u})^{-1} \SUCHTHAT \lambda   \in \gamma+\Sigma_{\pi-\beta} \right\}<\infty,
\end{equation}
and
\begin{equation}\label{scl02}
\mr_{p} \left\{ \lambda(\lambda - A_{\vartheta})^{-1} \SUCHTHAT \lambda   \in \gamma+\Sigma_{\pi-\beta} \right\}<\infty
\end{equation}
respectively in $\mathcal{L}(L^{q}(\ofo)^{3})$ and $\mathcal{L}(L^{q}(\ofo))$.
We deduce from the properties \eqref{eq:Rbdd-1} that
\begin{equation}\label{scl03}
\mr_{p} \left\{ \div(\lambda - A_{u})^{-1} \SUCHTHAT \lambda   \in \gamma+\Sigma_{\pi-\beta} \right\}<\infty
\end{equation}
in $\mathcal{L}(L^{q}(\ofo)^{3},W^{1,q}(\ofo))$.

We rewrite the operator $A_{F}$  in the form $A_{F} = A_{F}^0 + B_{F}$, with
$$
A_{F}^0 = \begin{pmatrix}
0 & \displaystyle - \overline\rho \div & 0
\\ 0 & A_u  & 0  \\
 0 &  0 & A_{\vartheta}
\end{pmatrix}
\mbox{ and }
B_{F} = \begin{pmatrix}
0 & 0 & 0 \\  - \displaystyle \frac{R\overline \vartheta}{\overline\rho} \nabla & 0   & -R \nabla  \\
 0 &  \displaystyle -\frac{R\overline \vartheta}{c_{v}} \div & 0
\end{pmatrix}.
$$
Some standard calculation shows that for $\lambda   \in \gamma+\Sigma_{\pi-\beta}$,
$$
\lambda (\lambda I  - A_{F}^0)^{-1} =  \begin{pmatrix}
I &  - \overline\rho \div (\lambda I - A_{u})^{-1} & 0 \\ 0 & \lambda (\lambda I  -  A_{u})^{-1} & 0 \\
0 & 0 & \lambda(\lambda I - A_{\vartheta})^{-1}
\end{pmatrix}.
$$

Using again the properties \eqref{eq:Rbdd-1}, we deduce from the above formula and from \eqref{scl01}--\eqref{scl03} that
$$
\mr_{p} \left\{ \lambda(\lambda - A_{F}^0)^{-1} \SUCHTHAT \lambda   \in \gamma+\Sigma_{\pi-\beta} \right\}<\infty
$$
in $\mathcal{L}(\mathcal{X})$.

Now for all $(\rho, u, \vartheta) \in \dom(A_{F})$ we have
\begin{equation} \label{e:bf1}
\left\|B_{F} \begin{pmatrix}
\rho \\ u \\ \vartheta \end{pmatrix} \right\|_{\mx}
\leqslant C \left(\|\rho\|_{W^{1,q}(\ofo)}+\|u\|_{W^{1,q}(\ofo)} + \|\vartheta\|_{W^{1,q}(\ofo)}\right).
\end{equation}
Using the compactness of the embedding
$W^{2,q}(\ofo) \hookrightarrow W^{1,q}(\ofo)$
and a classical result (see \cite[Chapter 3, Lemma 2.1]{Tem79}), we deduce that for any
$\delta > 0$, there exists $C(\delta) > 0$ such that for all $f \in W^{2,q}(\ofo)$
\begin{align*}
\|f\|_{W^{1,q}(\ofo)} \leqslant \delta \|f\|_{W^{2,q}(\ofo)} + C(\delta) \|f\|_{L^{q}(\ofo)}.
\end{align*}
Applying the above estimate to \eqref{e:bf1} we obtain, for every $\delta > 0$, there exists $C(\delta) > 0$ such that
\begin{align*}
\left\|B_{F} \begin{pmatrix}
\rho \\ u \\ \vartheta \end{pmatrix} \right\|_{\mx} \leqslant \delta  \left\|A_{F}^0 \begin{pmatrix}
\rho \\ u \\ \vartheta \end{pmatrix}\right\|_{\mx} + C(\delta) \left\|\begin{pmatrix}
\rho \\ u \\ \vartheta \end{pmatrix}\right\|_{\mx}.
\end{align*}
Therefore applying \cref{pr:perturb} we complete the proof of the theorem.
\end{proof}

Now we want to show that the operator $A_{F}$ is invertible in a suitable subspace of $\mx$. For this purpose we consider the following problem
\begin{equation}\label{eq:inv-af}
\begin{cases}
\displaystyle \overline \rho \div u = f_{1} & \mbox{ in }  \ofo,  \\
\displaystyle   -  \frac{\mu}{\overline\rho}\Delta u - \frac{\alpha + \mu}{\overline\rho} \nabla (\div\  u) + \frac{R\overline \vartheta}{\overline\rho} \nabla  \rho + R \nabla\vartheta = f_{2} & \mbox{ in } \ofo,  \\
\displaystyle    - \frac{\kappa}{\overline\rho c_{v}} \Delta \vartheta + \frac{R \overline \vartheta}{c_{v}} \div u = f_{3} & \mbox{ in } \ofo, \\
\displaystyle   u = 0 \mbox{ on } \partial\ofo, \quad \frac{\partial \vartheta}{\partial n} = 0 \mbox{ on } \partial\ofo.
\end{cases}
\end{equation}
By integrating the first and the third equations of \eqref{eq:inv-af} and by using the boundary conditions, we see that we need to impose the following compatibility conditions on $f_1$ and $f_3$:
$$
 \int_{\ofo} f_{1} \ dx = \int_{\ofo}  f_{3} \ dx= 0.
$$

We thus define
\begin{align}
L^{q}_{m}(\ofo) = \left\{ f \in L^{q}(\ofo) \mid \int_{\ofo} f \ dx = 0 \right\},
\end{align}
and
\begin{align} \label{def:xm}
\mx_{m} = \left[W^{1,q}(\ofo) \cap L^{q}_{m}(\ofo)\right] \times L^{q}(\ofo)^{3} \times L^{q}_{m}(\ofo).
\end{align}
Since $\mx_{m}$ is invariant under $(e^{tA_{F}})_{t \geqslant 0}$ the operator $A_{F}$ may be restricted to $\mx_{m}.$ The {\em{part of $A_{F}$ in $\mx_{m}$}} is the restriction of $A_{F}$ to the domain $\mathcal{D}(A_{F}) \cap \mx_{m}$ (\cite[Definition 2.4.1]{TW09}).

\begin{thm} \label{thm:inv-af}
The part of $A_{F}$ in $\mx_{m}$ is invertible in $\mx_{m}$: for every $(f_{1}, f_{2}, f_{3}) \in \mx_{m}$, the system \eqref{eq:inv-af} admits a unique solution $(\rho, u, \vartheta) \in \mathcal{D}(A_{F}) \cap \mx_{m}$ satisfying
\begin{multline}
\|\rho\|_{W^{1,q}(\ofo)} + \|u\|_{W^{2,q}(\ofo)} + \|\vartheta\|_{W^{2,q}(\ofo)} \\
\leqslant C ( \|f_{1}\|_{W^{1,q}(\ofo)} + \|f_{2}\|_{L^{q}(\ofo)^{3}} + \|f_{3}\|_{L^{q}(\ofo)} ).
\end{multline}
\end{thm}
\begin{proof}

Replacing $\div u = f_{1}/\overline \rho$ in the third equation of \eqref{eq:inv-af} yields
\begin{align*}
- \frac{\kappa}{\overline \rho c_{v}} \Delta \vartheta= f_{3} - \frac{R\overline \vartheta}{\overline\rho c_{v}} f_{1} \mbox{ in } \ofo, \quad\frac{\partial \vartheta}{\partial n} = 0 \mbox{ on } \partial\ofo.
\end{align*}
Since $f_1, f_3\in L^q_m(\ofo)$, by the standard elliptic theory (see for instance \cite[chapter 3]{Tr87}),
the above system admits a unique solution $\vartheta \in W^{2,q}(\ofo) \cap L^{q}_{m}(\ofo)$ and we have the estimate
\begin{align} \label{inv-af-1}
\|\vartheta\|_{W^{2,q}(\ofo)} \leqslant C ( \|f_{1}\|_{W^{1,q}(\ofo)}  + \|f_{3}\|_{L^{q}(\ofo)} ).
\end{align}
Then, we are reduced to solve the Stokes type system
$$
\begin{cases}
\displaystyle \overline \rho \div u = f_{1},  \mbox{ in }  \ofo,  \\
\displaystyle  -  \frac{\mu}{\overline\rho}\Delta u  + \frac{R\overline \vartheta}{\overline\rho} \nabla  \rho   = \overline f_{2},  \mbox{ in } \ofo,  \\
\displaystyle  u = 0 \mbox{ on } \partial\ofo,
\end{cases}
$$
with $\overline f_{2} = f_{2} - R \nabla\vartheta  + \frac{\alpha + \mu}{\overline\rho^{2}} \nabla f_{1}$. In particular, $\overline f_{2} \in L^{q}(\ofo)^{3}$ with
\begin{align} \label{inv-af-2}
\|\overline f_{2} \|_{L^{q}(\ofo)^{3}} \leqslant C ( \|f_{1}\|_{W^{1,q}(\ofo)} + \|f_{2}\|_{L^{q}(\ofo)^{3}} + \|\vartheta\|_{W^{2,q}(\ofo)} ).
\end{align}
From \cite[Theorem 2.9(1)]{ShibaEno13}, the above system admits a unique solution
$$
(\rho, u) \in \left[W^{1,q}(\ofo)  \cap L^{q}_{m}(\ofo)\right] \times W^{2,q}(\ofo)^{3}
$$
satisfying the estimate
\begin{align} \label{inv-af-3}
\|\rho\|_{W^{1,q}(\ofo)} + \|u\|_{W^{2,q}(\ofo)}  \leqslant C \left( \|f_{1}\|_{W^{1,q}(\ofo)} + \|\overline f_{2}\|_{L^{q}(\ofo)^{3}}\right).
\end{align}
The proof follows from the estimates \eqref{inv-af-1} - \eqref{inv-af-3}.
\end{proof}

\subsection{Rewriting ~(\ref{eq:linearS}) in an operator form}

Let us consider the following problem

\begin{equation}\label{eq:lift-1}
\begin{cases}
 - \mu \Delta u_{s} + R\overline \vartheta \nabla \rho_{s}  = 0  \mbox{ in } \ofo,  \quad  \div u_{s} = 0 \mbox{ in } \ofo,  \\
 u_{s} = \ell + \omega \times y \mbox{ on } \partial \oso,\quad   u_{s} = 0 \mbox{ on } \partial\Omega,\\
\displaystyle \int_{\ofo} \rho_s \ dy =0.
\end{cases}
\end{equation}

\begin{lem} \label{lem:lift}
Let $(\ell, \omega) \in \ct \times \ct$ and let $\{e_{i}\}$ denote the canonical basis in $\mathbb{C}^{3}$. Then the solution $(\rho_{s},u_{s})$ of \eqref{eq:lift-1} can be expressed as follows
 \begin{align}
  \rho_{s} = \sum_{i=1}^{3} \ell_{i} P_{i} + \sum_{i=4}^{6} \omega_{i-3} P_{i}, \quad u_{s} = \sum_{i=1}^{3} \ell_{i} U_{i} + \sum_{i=4}^{6} \omega_{i-3} U_{i},
\end{align}
where  $(U_{i},P_{i})$, $i = 1, 2, \cdots, 6$ solves the following systems
\begin{equation} \label{eq:UP}
\begin{cases}
-\mu \Delta U_{i} + R \overline \vartheta \nabla P_{i} = 0 \mbox{ in } \ofo, \\
\div U_{i} = 0\quad \mbox{ in } \ofo,  \\
 U_{i} = 0, \quad \mbox{ on } \partial\Omega, \\
 U_{i} = e_{i} \mbox{ on } \partial\ofo, \quad (i=1,2,3),\\
 U_{i} = e_{i-3} \times y \mbox{ on } \partial\ofo, \quad (i=4,5,6),\\
 \displaystyle \int_{\ofo} P_i \ dy =0.
\end{cases}
\end{equation}
Moreover,
\begin{align*}
\begin{pmatrix}
\displaystyle
\int_{\partial\oso} \sigma_{l}(\rho_{s}, u_{s},0 )n \  d\gamma \\ \displaystyle \int_{\partial\oso} y \times \sigma_{l}(\rho_{s}, u_{s},0)n \  d\gamma
\end{pmatrix} = \mathbb{A} \begin{pmatrix}
\ell \\ \omega
\end{pmatrix},
\end{align*}
where
\begin{align} \label{matA}
\displaystyle \mathbb {A}_{i,j} = \frac{2\mu}{\overline\rho} \int_{\ofo} DU_{i} : DU_{j} \ dx.
\end{align}
\end{lem}

\begin{proof}
See  \cite[Chapter 5]{HB65}.
\end{proof}

Let us set
\[ \mathcal{Z} = W^{1,q}(\ofo)\times W^{2,q}(\ofo)^{3} \times W^{2,q}(\ofo) \mbox{ and }  \my = \mx \times \ct \times \ct. \]
We introduce the Dirichlet operator $D_{s} \in \mathcal{L}(\ct \times \ct; \mz )$ defined by
\begin{align}
D_{s} \begin{pmatrix}
\ell \\ \omega
\end{pmatrix} = \begin{pmatrix}
\rho_{s} \\ u_{s} \\ 0
\end{pmatrix},
\end{align}
 where $(\rho_{s},u_{s})$ is the solution of the system \eqref{eq:lift-1}.  In view of \cref{lem:lift}, the operator $D_{s}$ can also be defined as
\begin{align}
D_{s} \begin{pmatrix}
\ell \\ \omega
\end{pmatrix} = \begin{pmatrix}
P_{1} & P_{2} & \cdots & P_{6} \\ U_{1} & U_{2} & \cdots & U_{6} \\ 0 & 0 & \cdots & 0
\end{pmatrix}
\begin{pmatrix}
\ell \\ \omega
\end{pmatrix},
\end{align}
where  $(U_{i},P_{i})$, $i = 1, 2, \cdots, 6$ are the solutions of the systems \eqref{eq:UP}. Let us recall the operator $(A_{F}, \mathcal{D}(A_{F};\mx))$ introduced in \cref{subsec:cns-L}. By \cref{thm:rsec-af} we know that, the operator $A_{F}$ generates a $C^{0}$ semigroup on $\mx$. It is well-known that, the operator $A_{F}$ has an extension, also denoted by $A_{F}$, such that $A_{F} \in \mathcal{L}(\mx, \dom(A_{F}^{*})')$, where $A_{F}^{*}$ denotes the adjoint operator of $A_{F}$ and $\dom(A_{F}^{*})'$ denotes the dual of $\dom(A_{F}^{*})$ (see \cite[Chapter 2, Section 5]{EngNag}).
Let us now briefly describe,  how to rewrite system \eqref{eq:linearS} as an evolution equation. All the details can be found in \cite{MT17}.

We  introduce the operator $\mathcal{A}_{FS}: \mathcal{D}(\mathcal{A}_{FS}) \to \my$ defined by
\begin{align*}
\mathcal{D}(\mathcal{A}_{FS}) = \left\{ (\rho, u, \vartheta, \ell, \omega) \in \my \mid  A_{F} \begin{pmatrix}
\rho \\ u \\ \vartheta \end{pmatrix}   - A_{F} D_{s} \begin{pmatrix}
\ell \\ \omega
\end{pmatrix} \in \mathcal{X} \right\},
\end{align*}
\begin{align} \label{calA}
{\mathcal A}_{FS}  =
\begin{pmatrix}  A_{F}  &    -  A_{F}D_{s} \\  C & 0  \end{pmatrix},
\end{align}
where  $C \in \mathcal{L}( \mz, \ct \times \ct)$ is defined by
\begin{align} \label{def-C}
C \begin{pmatrix}
\rho \\ u \\ \vartheta
\end{pmatrix} =
\begin{pmatrix}
\displaystyle - m^{-1} \int_{\partial\oso} \sigma_{l} (\rho,u,\vartheta) n \ d\gamma   \\
\displaystyle  - J(0)^{-1} \int_{\partial \oso} y \times \sigma_{l} (\rho,u,\vartheta) n \ d\gamma
\end{pmatrix}.
\end{align}

\begin{prop}
Let $1 < p < \infty$ and $1 < q < \infty$. Let $\ell \in W^{1,p}(0,\infty;\ct)$, $\omega \in W^{1,p}(0,\infty;\ct)$, $\rho \in W^{1,p}(0,\infty;W^{1,q}(\ofo))$,  $u \in W^{2,1}_{q,p}(Q_{F}^{\infty})^{3}$  and $\vartheta  \in W^{2,1}_{q,p}(Q_{F}^{\infty})$. Then $(\rho,v, \vartheta,\ell,\omega)$ is a solution of  the system \eqref{eq:linearS} if and only if
\begin{align} \label{eq:evo-fs}
\frac{d}{dt} \begin{pmatrix}
\rho \\ u \\ \vartheta \\ \ell \\ \omega \end{pmatrix} = {\mathcal A}_{FS} \begin{pmatrix}
\rho \\ u \\ \vartheta \\ \ell \\ \omega \end{pmatrix} \quad \text{in} \ \dom(A_{F}^{*})' \times \ct \times \ct,  \quad
\begin{pmatrix}
\rho(0) \\ u(0) \\ \vartheta(0) \\ \ell(0) \\ \omega(0)  \end{pmatrix} = \begin{pmatrix} \rho_{0} \\ u_{0} \\ \vartheta_{0} \\ \ell_{0} \\ \omega_{0}
 \end{pmatrix}.
\end{align}
\end{prop}
We skip the proof since it is  standard. We end this subsection by recalling an equivalence of norms in $\mathcal{D}(\mathcal{A}_{FS})$
(see \cite[Lemma 1.24]{MT17}).
\begin{lem} \label{lem:equiv-norm}
The map
\begin{align*}
(\rho, u, \vartheta,\ell, \omega) \mapsto  \|(\rho, u, \vartheta)\|_{\mz} + \|(\ell, \omega)\|_{\ct \times \ct}
\end{align*}
or equivalently the map
\begin{align*}
(\rho, u , \vartheta, \ell, \omega) \mapsto  \|\rho\|_{W^{1,q}(\ofo)} + \|u\|_{W^{2,q}(\ofo)^{3}} + \|\vartheta\|_{W^{2,q}(\ofo)} + \|\ell\|_{\ct} + \|\omega\|_{\ct}
\end{align*}
is a norm on $\mathcal{D}(\mathcal{A}_{FS})$ equivalent to the graph norm.
\end{lem}

\subsection{${\mathcal R}$-sectoriality of the operator $\mathcal{A}_{FS}$}
In this subsection we prove the following theorem

\begin{thm}  \label{thm:R2}
Let $1 < q < \infty$. Then there exists $\gamma_{3} > 0$ such that  $\mathcal{A}_{FS} - \gamma_{3}$ is an ${\mathcal R}$-sectorial operator in
$\mathcal{Y}$ of angle $< \pi/2.$
\end{thm}

\begin{proof}
We write $\mathcal{A}_{FS}$ in the form $\mathcal{A}_{FS} = \mathcal{A}_{FS, 1} + B_{FS}$ where
\begin{align*}
\mathcal{A}_{FS, 1}  = \begin{pmatrix}  A_{F}  &    -  A_{F}D_{s} \\  0 & 0  \end{pmatrix}, \quad
B_{FS} = \begin{pmatrix}  0  &    0 \\  C & 0  \end{pmatrix}.
\end{align*}
Observe that
\begin{align*}
\lambda (\lambda I  - \mathcal{A}_{FS, 1})^{-1} =  \begin{pmatrix}
\lambda (\lambda I -  A_{F})^{-1} & -(\lambda I -  A_{F})^{-1} A_{F} D_{s} \\ 0 & I
\end{pmatrix}.
\end{align*}
Therefore by \cref{thm:rsec-af} and \eqref{eq:Rbdd-1}, there exists $\gamma$ such that
${\mathcal A}_{FS,1}-\gamma$ is $\mr$-sectorial.

Now, by using trace results on Sobolev spaces, we deduce that for any $(\rho, u, \vartheta) \in \mathcal{Z}$,
\begin{align*}
\left\| C \begin{pmatrix}
\rho \\ u \\ \vartheta
\end{pmatrix}\right\|_{\ct \times \ct} \leqslant C (\|\rho\|_{W^{s,q}(\ofo)} + \|u\|_{W^{1+s,q}(\ofo)} + \|\vartheta\|_{W^{s,q}(\ofo)}), \quad s \in (1/q,1).
\end{align*}
Since the embedding $W^{1,q}(\ofo) \hookrightarrow W^{s,q}(\ofo)$ is compact for $s \in (1/q,1),$
we obtain that for any $\delta > 0$ there exists $C(\delta) > 0$ such that
\begin{align*}
\left\| C \begin{pmatrix}
\rho \\ u \\ \vartheta
\end{pmatrix}\right\|_{\ct \times \ct} \leqslant \delta \|(\rho, u, \vartheta)\|_{\mz} + C(\delta) \|(\rho, u, \vartheta)\|_{\mx}.
\end{align*}
By \cref{lem:equiv-norm}, this implies that for any $\delta > 0$ there exists $C(\delta) > 0$ such that
\begin{align*}
\left\| C \begin{pmatrix}
\rho \\ u \\ \vartheta
\end{pmatrix}\right\|_{\ct \times \ct} \leqslant \delta \left\| \mathcal{A}_{FS,1}\begin{pmatrix}
\rho \\ u \\ \vartheta \\ \ell \\ \omega
\end{pmatrix}\right\|_{\my} + C(\delta) \left\|\begin{pmatrix}
\rho \\ u \\ \vartheta \\ \ell \\ \omega
\end{pmatrix}\right\|_{\my}.
\end{align*}
Therefore the proof follows from  \cref{pr:perturb}.
\end{proof}


\subsection{Exponential stability of the  semigroup $e^{t \mathcal{A_{FS}}}$} \label{sec:stab-afs}

The aim of this subsection is to show the operator $\mathcal{A}_{FS}$ generates an analytic semigroup of negative type  in a suitable subspace of $\my$. Let us set
\[ \my_{m} = \mx_{m} \times \ct \times \ct,\]
where $\mx_{m}$ is defined as in \eqref{def:xm}. One can check that
the space  $\my_{m}$ is invariant under $(e^{t{\mathcal A}_{FS}})_{t \geqslant 0}$.
We prove the following theorem
\begin{thm} \label{thm:stab-afs}
Let $1 < q < \infty$. Then the part of $\mathcal{A}_{FS}$ in $\my_{m}$  generates an exponentially stable semigroup $(e^{t{\mathcal A}_{FS}})_{t\geqslant 0}$ on $\my_{m}$.  In other words, there exist constants $C > 0$ and $\eta_{0} > 0$ such that
\begin{align}
\|e^{t{\mathcal A}_{FS}} (\rho_{0}, u_{0}, \vartheta_{0}, \ell_{0}, \omega_{0})^{\top}\|_{\my_{m}} \leqslant C e^{-\eta_{0} t } \| (\rho_{0}, u_{0}, \vartheta_{0}, \ell_{0}, \omega_{0})^{\top}\|_{\my_{m}},
\end{align}
for all $ (\rho_{0}, u_{0}, \vartheta_{0}, \ell_{0}, \omega_{0})^{\top} \in \my_{m}$.
\end{thm}

We consider the following resolvent problem
\begin{align}  \label{eq:resolvent-s}
&\lambda \rho + \overline \rho \div u = f_{1},  \mbox{ in }  \ofo, \notag \\
& \lambda u - \div\sigma_{l}(\rho,u,\vartheta) = f_{2},  \mbox{ in } \ofo, \notag \\
& \lambda \vartheta  - \frac{\kappa}{\overline\rho c_{v}} \Delta \vartheta + \frac{R \overline \vartheta}{c_{v}} \div u = f_{3}, \mbox{ in } \ofo, \\
& u = \ell + \omega \times y \mbox{ on } \partial \oso \quad  u = 0 \mbox{ on } \partial\Omega, \quad \frac{\partial \vartheta}{\partial n} = 0 \mbox{ on } \partial\ofo, \notag \\
& \lambda \ell = - m^{-1} \int_{\partial\oso} \sigma_{l} (\rho,u,\vartheta) n \ d\gamma  + g_{1},  \notag \\
&\lambda \omega = - J(0)^{-1} \int_{\partial\oso} y \times \sigma_{l} (\rho,u,\vartheta) n \ d\gamma  + g_{2}. \quad  \notag
\end{align}

We want to show that the set $\left\{\lambda \in \mathbb{C} \mid \Re \lambda \geqslant 0 \right\}$, i.e. the entire right half plane is contained in  the resolvent set of part of $\mathcal{A}_{FS}$ in $\my_{m}.$ This will be achieved in two steps. In the first step we show that $0$ belongs to the resolvent of  part of $\mathcal{A}_{FS}$ in $\my_{m}.$ In the second step,  we show that the set $\{ \lambda \in \mathbb{C} \setminus \{0\} \mid 0 \leqslant \Re \lambda\}$  is contained in the resolvent of  part of $\mathcal{A}_{FS}$ in $\my_{m}.$

\begin{remark}
If $\lambda = 0$, integrating the first and third equations of \eqref{eq:resolvent-s} and using the boundary  conditions of $u$ and $\vartheta$ we obtain
\begin{align*}
\int_{\ofo} f_{1} \ dy= \int_{\ofo} f_{3} \ dy= 0.
\end{align*}
Therefore in order to show $\mathcal{A}_{FS}$ generates an exponentially stable semigroup it is necessary to  consider $\my_{m}$ instead of $\my$.
\end{remark}

\begin{thm} \label{thm:resolvent-1}
Let $1 < q < \infty$ and $\lambda = 0$. Then for every $(f_{1}, f_{2}, f_{3}, g_{1}, g_{2}) \in \my_{m}$ the system \eqref{eq:resolvent-s} admits a unique solution $(\rho, u, \vartheta, \ell, \omega) \in \mathcal{D}({\mathcal A}_{FS}) \cap \my_{m}$ satisfying the estimate
\begin{multline}
\|\rho\|_{W^{1,q}(\ofo)} + \|u\|_{W^{2,q}(\ofo)} + \|\vartheta \|_{W^{2,q}(\ofo)} + \|\ell\|_{\ct} + \|\omega\|_{\ct} \leqslant C \|(f_{1}, f_{2}, f_{3}, g_{1}, g_{2})\|_{\my_{m}}.
\end{multline}
\end{thm}

\begin{proof}
When  $\lambda = 0$ it is easy to see that \eqref{eq:resolvent-s} is equivalent to
\begin{align*}
\mathcal{A}_{FS} (\rho, u, \vartheta, \ell, \omega)^{\top} =  (f_{1}, f_{2}, f_{3}, g_{1}, g_{2})^{\top}.
\end{align*}
Thus to prove the theorem, we first show that the operator $\mathcal{A}_{FS}$ is invertible.  One can easily check that, if the operators $A_{F}$ and $CD_{s}$ are invertible then the operator $\mathcal{A}_{FS}$ is invertible and its inverse is given by the formula
\begin{align*}
\mathcal{A}_{FS}^{-1} =  \begin{pmatrix}
A_{F}^{-1}  - D_{s} (CD_{s})^{-1} C A_{F}^{-1} &  D_{s} (CD_{s})^{-1} \\
 - (CD_{s})^{-1} C A_{F}^{-1} & (CD_{s})^{-1}
\end{pmatrix}.
\end{align*}
We know that $A_{F}$ is invertible on $\mx_{m}$ (\cref{thm:inv-af}). Thus to complete the proof we need to verify that $CD_{s}$ is an invertible matrix.
From \cref{lem:lift}, we can see that
\begin{align}
CD_{s} = - \mathbb{M}^{-1} \mathbb{A} \quad  \mbox{where}  \quad \mathbb{M} = \begin{pmatrix}
mI & 0 \\ 0 & J(0)
\end{pmatrix}.
\end{align}
Since the matrix $\mathbb{A}$ is self-adjoint and positive, we deduce the result.
\end{proof}

\begin{thm} \label{thm:resolvent-2}
Assume $1 < q < \infty$ and $\lambda \in \mathbb{C} \setminus \{0\},$ with $\Re \lambda\geq 0$.
Then for any $(f_{1}, f_{2}, f_{3}, g_{1}, g_{2}) \in \my_{m}$, the system \eqref{eq:resolvent-s} admits a unique solution  $(\rho, u , \vartheta, \ell, \omega) \in \mathcal{D}({\mathcal A}_{FS}) \cap \my_{m}$ satisfying the estimate
\begin{multline} \label{est:resolvent-2}
\|\rho\|_{W^{1,q}(\ofo)} + \|u\|_{W^{2,q}(\ofo)} + \|\vartheta\|_{W^{2,q}(\ofo)} + \|\ell\|_{\ct}
+ \|\omega\|_{\ct}
\\
\leqslant C \|(f_{1}, f_{2}, f_{3}, g_{1}, g_{2}) \|_{\my_{m}}.
\end{multline}
\end{thm}

\begin{proof}
Let us fix $\lambda \in \mathbb{C} \setminus \{0\},$ with $\Re \lambda\geq 0$.
By setting $\rho = \displaystyle \frac{1}{\lambda} (f_{1} - \overline \rho \div u )$,  the system \eqref{eq:resolvent-s} can be rewritten as
\begin{equation} \label{eq:R-unique}
\begin{cases}
\displaystyle \lambda u - \div\widehat \sigma_\lambda(u,\vartheta) = \widehat f_{2},  \mbox{ in }  \ofo,  \\
\displaystyle \lambda \vartheta  - \frac{\kappa}{\overline\rho c_{v}} \Delta \vartheta + \frac{R \overline \vartheta}{c_{v}} \div u = f_{3}, \mbox{ in } \ofo, \\
\displaystyle
u = \ell + \omega \times y \mbox{ on } \partial \oso \quad  u = 0 \mbox{ on } \partial\Omega, \quad \frac{\partial \vartheta}{\partial n} = 0 \mbox{ on } \partial\ofo,  \\
\displaystyle
\lambda \ell = - m^{-1} \int_{\partial \oso} \widehat \sigma_\lambda (u, \vartheta) \  n \ d\gamma  + \widehat g_{1}   \\
\displaystyle
\lambda \omega  = - J(0)^{-1} \int_{\partial \oso} y \times \widehat \sigma_\lambda (u, \vartheta) n \ d\gamma + \widehat g_{2}
\end{cases}
\end{equation}
where
$$
\widehat \sigma_\lambda (u, \vartheta) =  \frac{2\mu}{\overline\rho} D(u) + \left(\frac{1}{\overline \rho}\left(\alpha + \frac{R\overline \vartheta \overline\rho}{\lambda}\right) \div u  - R \vartheta\right)I_{3},
$$
$$
\widehat f_{2} = f_{2} - \frac{R \overline \vartheta}{\lambda \overline \rho}\nabla f_{1}, \quad
\widehat g_{1} = \left( g_{1} +m^{-1}  \frac{R \overline \vartheta}{\lambda \overline \rho} \int_{\partial \oso} f_{1}  n \ d\gamma\right),
$$
$$
\widehat g_{2}  =  \left( g_{2} +J(0)^{-1}  \frac{R \overline \vartheta}{\lambda \overline \rho} \int_{\partial \oso} y \times  f_{1}  n \ d\gamma\right).
$$
If $(f_{1}, f_{2}, f_{3}, g_{1}, g_{2}) \in \my_{m}$, the above formulas imply that
$(\widehat f_{2}, \widehat g_{1}, \widehat g_{2}) \in L^{q}(\ofo)^{3} \times \ct \times \ct.$

We introduce the following notation:
\begin{gather*}
\widehat \mz = W^{2,q}(\ofo)^{3} \times \left(W^{2,q}(\ofo)\cap L^{q}_{m}(\ofo)\right),
\quad \widehat \mx  = L^{q}(\ofo)^{3} \times L^{q}_{m}(\ofo) \\
\widehat \my = \widehat \mx \times \ct \times \ct.
\end{gather*}

\begin{itemize}
\item  $\ds \widehat D_{s} \in \mathcal{L}(\ct \times \ct; \widehat \mz)$ and $\ds \widehat E_{s} \in \mathcal{L}(\ct \times \ct; \widehat \mz)$ defined by
\begin{align*}
\widehat D_{s} \begin{pmatrix}
\ell \\ \omega
\end{pmatrix} = \begin{pmatrix}
u_{s} \\ 0
\end{pmatrix}, \quad \widehat E_{s} \begin{pmatrix}
\ell \\ \omega
\end{pmatrix} = \begin{pmatrix}
\dfrac{R\bar\vartheta}{\bar\rho} \nabla \rho_{s} \\ 0
\end{pmatrix},
\end{align*}
where $(\rho_{s},u_{s})$ is the solution of the system \eqref{eq:lift-1}.
\item $\widehat A_{\lambda}$ defined by
$$
\mathcal{D}(\widehat A_{\lambda}) = \left\{(u,\vartheta) \in \widehat \mz \mid u = \frac{\partial \vartheta}{\partial n} = 0 \mbox{ on } \partial \ofo \right\}
$$
\begin{align*}
\widehat A_{\lambda} \begin{pmatrix}
u \\ \vartheta
\end{pmatrix} = \begin{pmatrix}
\ds \div \widehat\sigma_\lambda(u,\vartheta) \\
\ds  \frac{\kappa}{\overline\rho c_{v}} \Delta \vartheta  - \frac{R \overline \vartheta}{c_{v}} \div u
\end{pmatrix},
\quad
\begin{pmatrix}
u \\ \vartheta
\end{pmatrix}\in \mathcal{D}(\widehat A_{\lambda}).
\end{align*}
\item $\widehat C_{\lambda} \in \mathcal{L}( \mz, \ct \times \ct)$ defined by
\begin{align*}
\widehat C_{\lambda} \begin{pmatrix}
 u \\ \vartheta
\end{pmatrix} =
\begin{pmatrix}
\displaystyle - m^{-1} \int_{\partial \oso} \widehat \sigma_\lambda (u,\vartheta) n \ d\gamma   \\
\displaystyle  - J(0)^{-1} \int_{\partial \oso} y \times \widehat\sigma_\lambda (u,\vartheta) n \ d\gamma
\end{pmatrix}.
\end{align*}
\item ${\mathcal A}_{\lambda}$ defined by
\begin{align*}
\mathcal{D}(\mathcal{A}_{\lambda}) =  \left\{ ( u, \vartheta, \ell, \omega) \in \widehat \mz \times \ct \times \ct \mid  \widehat A_{\lambda} \begin{pmatrix}
 u \\ \vartheta \end{pmatrix}   - \widehat A_{\lambda} \widehat D_{s} \begin{pmatrix}
\ell \\ \omega
\end{pmatrix} \in \widehat\mx \right\},
\end{align*}
\begin{align*}
{\mathcal A}_{\lambda}  =
\begin{pmatrix}  \widehat A_{\lambda}  &    -  \widehat A_{\lambda}\widehat D_{s} + \widehat E_{s} \\  \widehat{C}_{\lambda} & 0  \end{pmatrix}.
\end{align*}
\end{itemize}
With the above notation, we can write \eqref{eq:R-unique} as
\begin{equation}\label{jjjj}
(\lambda I - \mathcal{A}_{\lambda}) \begin{pmatrix}
u \\ \vartheta \\ \ell \\ \omega
\end{pmatrix}  =   \begin{pmatrix}
\widehat f_{2} \\ f_{3} \\ \widehat  g_{1} \\ \widehat g_{2}
\end{pmatrix}.
\end{equation}

Proceeding as in \cref{thm:R2}, one can show that there exists $\widetilde\lambda>0$ such that
$\widetilde \lambda I - \mathcal{A}_{\lambda}$ is invertible. Consequently, we write \eqref{jjjj} as
\begin{equation}\label{aaaa}
\left[ I + (\lambda - \widetilde \lambda) (\widetilde \lambda I - \mathcal{A}_{\lambda})^{-1} \right] \begin{pmatrix}
u \\ \vartheta \\ \ell \\ \omega
\end{pmatrix}  =
(\widetilde \lambda I - \mathcal{A}_{\lambda})^{-1}
\begin{pmatrix}
\widehat f_{2} \\ f_{3} \\ \widehat  g_{1} \\ \widehat g_{2}
\end{pmatrix}
\end{equation}
and since $(\widetilde \lambda I - \mathcal{A}_{\lambda})^{-1}$
is a compact operator, in view of the Fredholm alternative theorem, the existence and the uniqueness of a solution of system \eqref{aaaa} are equivalent.

Assume $(u, \vartheta, \ell, \omega) \in \mathcal{D}(\mathcal{A}_{\lambda})$ satisfies
\begin{equation}\label{bbbb}
(\lambda I - \mathcal{A}_{\lambda}) (u,\ \vartheta,\ \ell,\ \omega)^{\top}=0.
\end{equation}
We first show that $(u, \vartheta) \in W^{2,2}(\ofo)^{3} \times W^{2,2}(\ofo)$. If $q\geq 2$, this is a consequence of H\"older's estimates. Assume $1<q<2$. In that case,
we can write \eqref{bbbb} as
\begin{equation}\label{cccc}
(\widetilde\lambda I - \mathcal{A}_{\lambda}) (u,\ \vartheta,\ \ell,\ \omega)^{\top}=(\widetilde\lambda-\lambda)(u,\ \vartheta,\ \ell,\ \omega)^{\top}
\end{equation}
and since $W^{2,q}(\ofo)\subset L^2(\ofo)$, we deduce from the invertibility of $\widetilde\lambda I - \mathcal{A}_{\lambda}$ that
$(u, \vartheta) \in W^{2,2}(\ofo)^{3} \times W^{2,2}(\ofo)$. We then rewrite \eqref{bbbb} as
\begin{equation} \label{eq:R-unique5}
\begin{cases}
\displaystyle \lambda u - \div\widehat \sigma_\lambda(u,\vartheta) = 0,  \mbox{ in }  \ofo,  \\
\displaystyle \lambda \vartheta  - \frac{\kappa}{\overline\rho c_{v}} \Delta \vartheta + \frac{R \overline \vartheta}{c_{v}} \div u = 0, \mbox{ in } \ofo, \\
\displaystyle u = \ell + \omega \times y \mbox{ on } \partial \oso \quad  u = 0 \mbox{ on } \partial\Omega, \quad \frac{\partial \vartheta}{\partial n} = 0 \mbox{ on } \partial\ofo,  \\
\displaystyle  \lambda \ell = - m^{-1} \int_{\partial \oso} \widehat \sigma_\lambda (u, \vartheta) \  n \ d\gamma    \\
\displaystyle  \lambda \omega  = - J(0)^{-1} \int_{\partial \oso} y \times \widehat \sigma_\lambda (u, \vartheta) n \ d\gamma.
\end{cases}
\end{equation}

Multiplying the first equation of \eqref{eq:R-unique5} by $\overline {u}$, the second equation by $\overline{\vartheta}$,
the forth equation by $\overline{\ell}$ and the fifth equation by $\overline{\omega}$, we deduce after integration by parts,
\begin{multline}
\Re \lambda \int_{\ofo} |u|^{2} \ dy
+ \frac{2\mu}{\overline \rho}\int_{\ofo}  |D(u)|^2 \ dy
+  \left(\frac{\alpha}{\overline \rho} + \frac{R\overline \vartheta \ \Re \lambda  }{|\lambda|^{2}} \right) \int_{\ofo} |\div u|^{2} \ dy
\\
+\frac{\Re \lambda  \  c_{v}}{\overline \vartheta} \int_{\ofo} |\vartheta|^{2} \ dy
+ \frac{\kappa}{\overline \rho \overline \vartheta} \int_{\ofo} |\nabla \vartheta|^{2} \ dy
+ \Re  \lambda m |\ell|^{2} + \Re  (\lambda J(0) \omega \cdot \overline{\omega}) = 0.
\end{multline}
Using
\begin{equation*}
|\div u |^{2} \leqslant  3 | D(u) |^2,
\end{equation*}
$\Re \lambda \geq 0$ and \eqref{tak1.3}, we obtain
$$
\int_{\ofo}  |D(u)|^2 \ dy+\int_{\ofo} |\nabla \vartheta|^{2} \ dy \leq 0.
$$
The above estimate and the fact that  $(u, \vartheta, \ell, \omega) \in \mathcal{D}(\mathcal{A}_{\lambda})$ imply that  $u = \vartheta = \ell = \omega = 0$.
\end{proof}

\begin{proof}[Proof of \cref{thm:stab-afs}]
By virtue of \cref{thm:resolvent-1} and \cref{thm:resolvent-2},
we deduce
$$
\{\lambda \in \mathbb{C} \ ; \ \Re \lambda \geq 0\} \subset \rho(\mathcal{A}_{FS}).
$$
Moreover, \cref{thm:R2} yields the existence of $C_1>0$ such that
for any $\lambda \in \gamma_3+\Sigma_{\pi-\beta_3}$ with $\beta_3<\pi/2$,
$$
\| (\lambda -\mathcal{A}_{FS})^{-1}\|_{\mathcal{L}(\mathcal{Y}_m)} \leq C_1.
$$
Since
$$
\{\lambda \in \mathbb{C} \ ; \ \Re \lambda \geq 0\} \setminus \left[\gamma_3+\Sigma_{\pi-\beta_3}\right]
$$
is a compact set, we deduce the existence of $C>0$ such that for any
$\lambda \in \mathbb{C}$ with $\Re \lambda \geq 0$
$$
\| (\lambda -\mathcal{A}_{FS})^{-1}\|_{\mathcal{L}(\mathcal{Y}_m)} \leq C.
$$
This yields that
$$
\{\lambda \in \mathbb{C} \ ; \ \Re \lambda \geq -\eta\} \subset \rho(\mathcal{A}_{FS}),
$$
for some $\eta>0$. Applying standard results on analytic semigroups (see, for instance, Proposition 2.9 in \cite[p.120]{BDDM}),
we deduce the exponential stability of the semigroup generated by the part of $\mathcal{A}_{FS}$ in $\my_{m}$.
\end{proof}

\section{Maximal $L^{p}$-$L^{q}$ Regularity for the Linearized Fluid-Structure System} \label{sec:max-lin-g}

  In this section, we study the maximal $L^{p}$-$L^{q}$ regularity of the system \eqref{eq:linearS} with
  source terms and boundary terms. More precisely, we consider the following system
 \begin{alignat}{2} \label{sys:NL-G-S}
&\partial_{t} \widetilde\rho + \overline \rho \div\widetilde u = f_{1} &  \mbox{ in } (0,\infty) \times \ofo, \notag \\
& \partial_{t} \widetilde u - \div \sigma_{l}(\widetilde \rho,\widetilde u,\widetilde\vartheta) = f_{2} & \mbox{ in } (0,\infty) \times \ofo, \notag \\
& \partial_{t} \widetilde \vartheta  - \frac{\kappa}{\overline\rho c_{v}} \Delta \widetilde\vartheta + \frac{R \overline \vartheta}{c_{v}} \div \widetilde u = f_{3} & \mbox{ in } (0,\infty) \times \ofo, \notag \\
&\widetilde  u = 0 & \mbox{ on } (0,\infty) \times \partial\Omega, \notag \\
& \widetilde u = \widetilde\ell + \widetilde\omega \times y & \mbox{ on } (0,\infty) \times \partial \oso \\
& \frac{\partial \widetilde \vartheta}{\partial n} = h  & \mbox{ on } (0,\infty) \times \partial \ofo,  \notag \\
& \frac{d}{dt} \widetilde \ell = - m^{-1} \int_{\partial\oso} \sigma_{l} (\widetilde\rho,\widetilde u,\widetilde\vartheta) n \ d\gamma + g_{1} & \quad t \in (0,\infty), \notag \\
&\frac{d}{dt} \omega = - J(0)^{-1} \int_{\partial\oso} y \times \sigma_{l} (\rho,u,\vartheta) n \ d\gamma  + g_{2} & \quad t \in (0,\infty)  \notag \\
&   \widetilde\rho_{m}(0) = \rho_{0} - \overline\rho , \quad u(0) = u_{0},  \quad  \varphi(0) = \vartheta_{0} - \overline \vartheta &  \quad \mbox{ in } \ofo, \notag \\
&   \ell(0) = \ell_{0},  \quad \omega(0) = \omega_{0}. \notag
\end{alignat}

We want to combine \cref{thm:max-reg-g}  and \cref{thm:stab-afs}.
However the latter is stated in $\my_{m}$, that is with the constraints that some quantities have to be with zero mean-value. As a consequence, we introduce the following standard decomposition: for any $f\in L^1(\ofo)$,
\begin{equation}\label{tak1.8}
f = f_{m} + f_{avg}, \quad\text{with}\quad  \int_{\ofo} f_{m} \ dy = 0, \quad f_{avg}= |\ofo|^{-1} \int_{\ofo} f(y) \ dy.
\end{equation}
We use the same decomposition and the same notation for a function in $L^1(\partial \ofo)$.
The next result is the main result of this section. We recall that $\mathcal{B}_{\infty,p,q}$ and $\mathcal{S}_{\infty,p,q}$ are defined in \eqref{solspace}
and \eqref{com0.5}.

Assume $1 < p < \infty$ and $1 < q < \infty$ satisfying the conditions
\begin{equation}\label{tak1.4}
\frac{1}{p} + \frac{1}{2q} \neq 1,\quad
\frac{1}{p} + \frac{1}{2q} \neq \frac12.
\end{equation}
We set
\begin{multline}
\mathcal{J}_{p,q}
=
\Big\{ (\rho_{0}, u_{0}, \vartheta_{0}, \ell_{0},\omega_{0}) \mid \rho_{0} \in W^{1,q}(\ofo)\cap L^q_m(\ofo), \
u_{0} \in B^{2(1-1/p)}_{q,p}(\ofo)^{3},
\\
\vartheta_{0} \in B^{2(1-1/p)}_{q,p}(\ofo),\
 \ell_{0} \in \mathbb{R}^{3},\  \omega_{0} \in \mathbb{R}^{3}\Big\},
\end{multline}
and we introduce the space of initial data
$$
\mathcal{J}^{cc}_{p,q} =\mathcal{J}_{p,q} \quad \text{if} \quad \displaystyle \frac{1}{p} + \frac{1}{2q} \geq 1,
$$
\begin{multline*}
\mathcal{J}^{cc}_{p,q} =
\Big\{(\rho_{0}, u_{0}, \vartheta_{0}, \ell_{0},\omega_{0}) \in \mathcal{J}_{p,q}
             \mid u_{0} = 0 \mbox{ on } \partial \Omega, \
 u_{0}(y) = \ell_{0} + \omega_{0} \times y \quad y\in \partial\oso  \Big\}
\\ \mbox{if} \quad \displaystyle \frac{1}{2} < \frac{1}{p} + \frac{1}{2q} < 1,
\end{multline*}
\begin{multline*}
\mathcal{J}^{cc}_{p,q} =
\Bigg\{(\rho_{0}, u_{0}, \vartheta_{0}, \ell_{0},\omega_{0}) \in \mathcal{J}_{p,q}
             \mid u_{0} = 0 \mbox{ on } \partial \Omega, \
 u_{0}(y) = \ell_{0} + \omega_{0} \times y \quad y\in \partial\oso,
\\
\displaystyle \frac{\partial \vartheta_{0}}{\partial n}  = 0 \mbox{ on } \partial\ofo \Bigg\}
\quad \mbox{if} \quad  \displaystyle \frac{1}{p} + \frac{1}{2q} < \frac{1}{2}.
\end{multline*}
The above definition is well-defined due to the trace theorem for Besov spaces (see, for instance \cite[p.200]{Triebel}).

  \begin{thm} \label{thm:nh-g}
  Let $1 < p < \infty$ and $1 < q < \infty$ satisfying \eqref{tak1.4}.
  Let $\overline \rho > 0$ and $\overline \vartheta > 0$ and  $\eta \in (0, \eta_{0})$, where
  $\eta_{0}$ is the constant introduced in  \cref{thm:stab-afs}.
  Then for any
  $$
  (\rho_{0} - \overline\rho,u_{0}, \vartheta_{0} - \overline \vartheta,\ell_{0},\omega_{0})\in \mathcal{J}^{cc}_{p,q}
  $$
and for any  $(f_{1}, f_{2}, f_{3},  h,   g_{1}, g_{2}) \in e^{-\eta(\cdot)}\mathcal{B}_{\infty,p,q}$
with $(f_{1,avg}, f_{3,avg}, h_{avg} )\in L^{1}(0,\infty)^{3}$, the system \eqref{sys:NL-G-S} admits a unique solution
$(\widetilde \rho, \widetilde u,\widetilde \vartheta,\widetilde \ell,\widetilde \omega)$ with
\begin{equation}\label{tak1.6}
(\widetilde \rho_{m}, \widetilde u, \widetilde \vartheta_{m},\widetilde \ell,\widetilde \omega) \in e^{-\eta(\cdot)} \mathcal{S}_{\infty,p,q},
 \quad
(\widetilde \rho_{avg}, \widetilde \vartheta_{avg}) \in L^{\infty}(0,\infty)^{2}.
\end{equation}
Moreover, there exists a positive constant $C_{L} $ depending only on $p,q$ and $\eta$ such that
\begin{multline}\label{tak1.7}
\left\| e^{\eta (\cdot)} (\widetilde\rho_{m}, \widetilde u, \widetilde\vartheta_{m},  \widetilde\ell, \widetilde\omega)\right\|_{{\mathcal S}_{\infty,p,q}}
+ \left \| (\widetilde \rho_{avg},\widetilde \vartheta_{avg})\right \|_{L^{\infty}(0,\infty)^2}
+ \left\|e^{\eta (\cdot)}(\partial_{t}\widetilde \rho_{avg}, \partial_{t}\widetilde \vartheta_{avg})\right\|_{L^{p}(0,\infty)^2}
\\
\leqslant C_{L} \Big( \left\| (\rho_{0} - \overline\rho,u_{0}, \vartheta_{0} - \overline \vartheta,\ell_{0},\omega_{0})\right\|_{\mathcal{J}_{p,q}}
+ \left\|e^{\eta (\cdot)}  (f_{1}, f_{2},  f_{3}, h, g_{1}, g_{2}) \right\|_{\mathcal{B}_{\infty,p,q}} \\
+ \left\|(f_{1,avg}, f_{3,avg}, h_{avg} ) \right\|_{L^{1}(0,\infty)^{3}} \Big).
\end{multline}
  \end{thm}

  \begin{proof}
 Let us first consider the case $\eta=0. $  We  consider the following heat equation
  \begin{alignat*}{2}
  &\partial_{t} \varphi^{1} + \mu' \varphi^{1} - \frac{\kappa}{\overline\rho c_{v}} \Delta \varphi^{1} = f_{3} - f_{3,avg} - \frac{\kappa|\partial\ofo|}{\overline\rho c_{v} |\ofo|} h_{avg} & \quad \mbox{ in } (0,\infty) \times  \ofo \\
  & \frac{\partial \varphi^{1}}{\partial n} = h & \mbox{ on } (0,\infty) \times \partial \ofo,  \\
 & \varphi^{1}(0) = \vartheta_{0} - \overline \vartheta & \mbox{ on } \ofo,
  \end{alignat*}
  with $\mu' > 0$. Using $(f_{1}, f_{2}, f_{3},  h,   g_{1}, g_{2}) \in \mathcal{B}_{\infty,p,q}$ and applying
Proposition 6.4 in \cite{DenkHieberPruss07} (taking $\mu'>0$ large enough), we deduce that the above system admits a unique solution
$\varphi^{1}\in W^{2,1}_{q,p} (Q_{\infty}^{F})$. Moreover, we have the
estimate
  \begin{multline}
  \|\varphi^{1}\|_{W^{2,1}_{q,p} (Q_{\infty}^{F})} \leqslant C \Big(  \|f_{3}\|_{L^{p}(0,\infty;L^{q}(\ofo))}  + \|h\|_{F^{(1-1/q)/2}_{p,q}(0,\infty;L^{q}(\partial \ofo)) \cap L^{p}(0,\infty;W^{1-1/q,q}(\partial \ofo))} \\
  + \|\vartheta_{0} - \overline \vartheta\|_{ B^{2(1-1/p)}_{q,p}(\ofo)}\Big).
  \end{multline}
  Standard calculation on the above system yields
  \begin{align*}
  \partial_{t} \varphi^{1}_{avg} + \mu' \varphi^{1}_{avg} = 0, \quad \varphi^{1}_{avg}(0) = \vartheta_{0,avg} - \overline \vartheta.
  \end{align*}
  Thus $\varphi^{1}_{avg}(t) = (\vartheta_{0,avg} - \overline \vartheta)e^{-\mu' t}$ and $\varphi^{1}_{avg} \in L^{r}(0,\infty)$ for any $1\leqslant r \leqslant \infty$.
 Next, we define
$$
  \varphi^{2}(t) = \int_0^t \left(f_{3,avg}(s) + \frac{\kappa|\partial\ofo|}{\overline\rho c_{v} |\ofo|} h_{avg}(s) + \mu' \varphi^{1}_{avg}(s)\right) \ ds.
 $$
 Since  $f_{3,avg}, h_{avg}, \varphi^{1}_{avg}\in L^{1}(0,\infty)$,
 we obtain $\varphi^{2} \in L^{\infty}(0,\infty)$.

 Integrating the first equation of \eqref{sys:NL-G-S} in $\ofo$ and using the boundary condition of $u$, we deduce that $\widetilde \rho_{avg}$ is solution of the following system
  \begin{align}
   \partial_{t}\widetilde \rho_{avg} = f_{1,avg} \quad t \in (0,\infty),  \quad \widetilde \rho_{avg}(0) = 0.
  \end{align}
  As $f_{1,avg}$ belongs to $L^{1}(0,\infty)$ we have $\widetilde \rho_{avg} \in L^{\infty}(0,\infty)$.
  We set
  $$
  \widetilde \varphi = \widetilde \vartheta -\varphi^1-\varphi^2.
  $$
  Then system \eqref{sys:NL-G-S} is transformed into the following system for
  $(\widetilde\rho_{m}, \widetilde u, \widetilde \varphi, \widetilde \ell, \widetilde \omega)$:
\begin{align}  \label{eq:11-7}
&\partial_{t} \widetilde\rho_{m} + \overline \rho \div\widetilde u = \widetilde f_{1}  \quad\mbox{ in } (0,\infty) \times \ofo, \notag \\
& \partial_{t} \widetilde u - \div\sigma_{l}(\widetilde \rho_{m},\widetilde u,\widetilde\vartheta) = \widetilde f_{2}  \quad \mbox{ in } (0,\infty) \times \ofo, \notag \\
& \partial_{t} \widetilde \varphi  - \frac{\kappa}{\overline\rho c_{v}} \Delta \widetilde\varphi + \frac{R \overline \vartheta}{c_{v}} \div\widetilde u = \widetilde f_{3} \quad \mbox{ in } (0,\infty) \times \ofo, \notag \\
&\widetilde  u = 0 \mbox{ on } (0,\infty) \times \partial\Omega, \quad \widetilde u = \widetilde\ell + \widetilde\omega \times y \mbox{ on } (0,\infty) \times \partial \oso \\
& \frac{\partial \widetilde \varphi}{\partial n} = 0  \mbox{ on } (0,\infty) \times \partial \ofo,  \notag \\
& \frac{d}{dt} \widetilde \ell = - m^{-1} \int_{\partial\oso} \sigma_{l} (\widetilde\rho_{m},\widetilde u,\widetilde\varphi) n \ d\gamma + \widetilde g_{1}, \quad t \in (0,\infty) \notag \\
&\frac{d}{dt} \omega = - J(0)^{-1} \int_{\partial\oso} y \times \sigma_{l} (\widetilde\rho_{m},\widetilde u,\widetilde \varphi) n \ d\gamma  + \widetilde g_{2}, \quad t \in (0,\infty)  \notag \\
&   \widetilde\rho(0) = \rho_{0} - \overline\rho , \quad u(0) = u_{0},  \quad  \widetilde \varphi(0) = 0 , \quad \mbox{ in } \ofo, \notag \\
&   \ell(0) = \ell_{0},  \quad \omega(0) = \omega_{0}, \notag
\end{align}
where
\begin{gather*}
\widetilde f_{1} = f_{1} - f_{1,avg}, \quad \widetilde f_{2} = f_{2} - R\nabla \varphi^{1}, \quad \widetilde f_{3} =  \mu'(\varphi^{1} -\varphi^{1}_{avg}) \\
\widetilde g_{1} = g_{1} + m^{-1}R \int_{\partial \oso} \varphi^{1} n \ d\gamma, \quad \widetilde g_{2} = g_{2} + J(0)^{-1} R \int_{\partial \oso} y \times \varphi^{1} n \ d\gamma,
\end{gather*}
and
where we recall that $\sigma_{l}$ is defined by \eqref{tak1.5}.
Using that $(f_{1}, f_{2}, f_{3},  h,   g_{1}, g_{2}) \in \mathcal{B}_{\infty,p,q}$ and that $\varphi^{1}\in W^{2,1}_{q,p} (Q_{\infty}^{F})$,
one can show that
$(\widetilde f_{1}, \widetilde f_{2}, \widetilde f_{3}, \widetilde g_{1}, \widetilde g_{2})$ belongs to $L^{p}(0,\infty;\my_{m})$.

From \cref{thm:R2} and \cref{thm:stab-afs}, we know that $\mathcal{A}_{FS}$ generates an analytic exponential stable semigroup on $\my_{m}$
and is a $\mr$-sectorial operator on $\my_{m}$.
Moreover, by hypothesis of \cref{thm:nh-g}, we have
 $(\rho_{0} - \overline\rho,u_{0}, 0,\ell_{0},\omega_{0}) \in \left(\my_{m}, \mathcal{D}({\mathcal A}_{FS}) \cap  \my_{m} \right)_{1-1/p,p}.$
 Therefore by \cref{thm:max-reg-g}, the system \eqref{eq:11-7} admits a unique solution
 $$
 (\widetilde \rho_{m}, \widetilde u,\widetilde \varphi,\widetilde \ell,\widetilde \omega) \in L^{p}(0,\infty;\mathcal{D}({\mathcal A}_{FS}) \cap  \my_{m} )
 \cap W^{1,p}(0,\infty; \my_{m} )
 \subset \mathcal{S}_{\infty,p,q}.
 $$
 We recover \eqref{tak1.6} and \eqref{tak1.7} by remarking that
$$
\widetilde \vartheta_{m}= \widetilde \varphi + \varphi^{1}_{m},
\quad
\widetilde \vartheta_{avg}= \varphi^{2} + \varphi^{1}_{avg}.
$$
The case $\eta > 0$ can be reduced to the previous case by multiplying all the functions by $e^{\eta t}$ and using the fact that ${\mathcal A}_{FS} + \eta$ is a $\mr$-sectorial operator with negative type.
  \end{proof}

\section{Estimates of the Nonlinear Terms} \label{sec:nl-est-g}
In this section, we are going to estimate the nonlinear terms $\mathcal{F}_{1},\mathcal{F}_{2},\mathcal{F}_{3},\mathcal{H}, \mathcal{G}_{1}$ and $\mathcal{G}_{2}$ defined in \eqref{F1-g} - \eqref{G-g}.

\medskip

Throughout this section we assume $2 < p < \infty$ and $3 < q < \infty$ satisfy
$\ds \frac{1}{p} + \frac{1}{2q} \neq \frac12$. Let $p'$
denote the conjugate of $p$, i.e.,
$\ds \frac{1}{p} + \frac{1}{p'} = 1$. In the following, when no confusion is possible, we use the notation
\[
\norm{  \cdot  }_{W^{r,p}(0,T;W^{s,q})} = \norm{ \cdot }_{W^{r,p}(0,T;W^{s,q}(\ofo))}.
\]
Let us fix  $\eta \in (0, \eta_{0})$, where $\eta_{0}$ is the constant introduced in  \cref{thm:stab-afs}. and  we introduce the following ball
$$
\widetilde{\mathcal S}_{\varepsilon} = \Big\{(\widetilde \rho, \widetilde u,\widetilde \vartheta,\widetilde \ell,\widetilde \omega) \mid
\left\| (\widetilde \rho, \widetilde u,\widetilde \vartheta,\widetilde \ell,\widetilde \omega)\right\|_{\mathcal{S}} \leqslant \varepsilon
 \Big\},
$$
where
\begin{multline} \label{ball-g}
\left\| (\widetilde \rho, \widetilde u,\widetilde \vartheta,\widetilde \ell,\widetilde \omega)\right\|_{\mathcal{S}}= \left\|(e^{\eta (\cdot)}\widetilde\rho_{m}, e^{\eta (\cdot)} \widetilde u,e^{\eta (\cdot)} \widetilde\vartheta_{m}, e^{\eta (\cdot)} \widetilde\ell, e^{\eta (\cdot)}\widetilde\omega)\right\|_{{\mathcal S}_{\infty,p,q}} \\ +  \|\widetilde \rho_{avg}, \widetilde \vartheta_{avg}\|_{L^{\infty}(0,\infty)^{2}} + \|e^{\eta (\cdot)}\partial_{t}\widetilde \rho_{avg}, e^{\eta (\cdot)}\partial_{t}\widetilde \vartheta_{avg}\|_{L^{p}(0,\infty)^{2}},
\end{multline}
and where we use the notation \eqref{tak1.8}.

Let us first show that $X$ be defined as in \eqref{Jacobi-g} is a $C^{1}$-diffeomorphism.
\begin{lem} \label{lem:diff-g}
Let $X$ be defined as in \eqref{Jacobi-g}. Then there exists a constant $C_{X} > 0,$ depending only on $p,q, \eta$ and $\ofo$ such that, for every    $(\widetilde \rho, \widetilde u,\widetilde \vartheta,\widetilde \ell,\widetilde \omega) \in \widetilde{\mathcal S}_{\varepsilon}$, we have
\begin{align}
\|\nabla X - I_{3}\|_{L^{\infty}((0,\infty) \times \ofo)}+ \norm{\nabla X - I_{3}}_{L^{\infty}(0,\infty;W^{1,q})} \leqslant C_{X} \varepsilon,\\
|X(\cdot, y_1)-X(\cdot, y_2)| \geq (1-C_X\varepsilon) |y_1-y_2| \quad (y_1,y_2\in \ofo). \label{tak1.9}
\end{align}
In particular for every $\varepsilon \in (0,\frac{1}{2C_{X}})$ and for every $(\widetilde \rho, \widetilde u,\widetilde \vartheta,\widetilde \ell,\widetilde \omega) \in \widetilde{\mathcal S}_{\varepsilon},$ we have
\begin{align} \label{inv-X}
\|\nabla X - I_{3}\|_{L^{\infty}((0,\infty) \times \ofo)} < \frac{1}{2}.
\end{align}
\end{lem}

\begin{proof}
From the definition of $X$ we obtain
\begin{align*}
\|\nabla X - I_{3}\|_{L^{\infty}((0,\infty) \times \ofo)} & \leqslant C \|\nabla X - I_{3}\|_{L^{\infty}(0,\infty;W^{1,q}(\ofo))} \\
& \leqslant C  \int_{0}^{\infty}e^{-\eta t} e^{\eta t}\|\nabla \widetilde u(t,\cdot)\|_{W^{1,q}(\ofo)} \ dt \\
&  \leqslant C  \left(\int_{0}^{\infty} e^{-p'\eta t } \ dt \right)^{1/p'} \|e^{\eta (\cdot)} \widetilde u\|_{W^{2,1}_{q,p} (Q_{\infty}^{F})} \\
&\leqslant C  \left( \frac{1}{p'\eta}\right)^{1/p'} \varepsilon,
\end{align*}
where $C$ depends only on $\ofo$.  The proof of \eqref{tak1.9} is similar.
This completes the proof of the lemma.
\end{proof}

From now on we assume that
\begin{align} \label{e0}
\varepsilon_{0} = \min \left\{1, \frac{1}{2C_{X}} \right\},
\end{align}
where $C_{X}$ is the constant in \cref{lem:diff-g}.

In the following lemma we estimate some other norms of $\nabla X$ and $[\nabla X]^{-1}$ that we need to estimate the nonlinear terms.
\begin{lem} \label{lem:X-O-G}
Let $X$ be defined as in \eqref{Jacobi-g} and $Z$ defined by \eqref{Z-g}.  Then there exists a constant $C > 0$ depending only on  $p,q,\eta$ and $\ofo$ such that, for every $\varepsilon \in (0,\varepsilon_{0})$ and for every  $(\widetilde \rho, \widetilde u,\widetilde \vartheta,\widetilde \ell,\widetilde \omega) \in \widetilde{\mathcal S}_{\varepsilon}$, we have
\begin{multline}
  \norm{\det(\nabla X - I_{3})}_{L^{\infty}(0,\infty;W^{1,q})}
  + \norm{\Cof(\nabla X - I_{3})}_{L^{\infty}(0,\infty;W^{1,q})}
  \\
+ \norm{Z - I_{3}}_{L^{\infty}(0,\infty;W^{1,q})}
+ \norm{\partial_{t} \nabla X}_{L^{p}(0,\infty;W^{1,q})}
+ \norm{\partial_{t}(Z - I_{3})}_{L^{p}(0,\infty;W^{1,q})} \\
+ \norm{Z - I_{3}}_{C^{1/p'}(0,\infty;W^{1,q})}
\leqslant C\varepsilon.
\end{multline}
\end{lem}

\begin{proof}
The estimates of $\det(\nabla X - I_{3})$ and $\Cof(\nabla X - I_{3})$ follow from \cref{lem:diff-g}
and from
the fact that the space $L^{\infty}(0,\infty;W^{1,q}(\ofo))$ is an algebra for $q>3$.

From \eqref{inv-X}, we deduce that $\det \nabla X \geqslant C > 0$ in $(0,\infty) \times \ofo$ and thus from
\[ \displaystyle Z  = \frac{\mathrm{Cof}(\nabla X )}{\mathrm{det}(\nabla X )}\]
we obtain
\begin{align*}
\|Z\|_{L^{\infty}(0,\infty;W^{1,q}(\ofo))} \leqslant C.
\end{align*}
Therefore,
\begin{align*}
\|Z - I_{3}\|_{L^{\infty}(0,\infty;W^{1,q})} \leqslant \|Z\|_{L^{\infty}(0,\infty;W^{1,q}(\ofo))} \|\nabla X - I_{3}\|_{L^{\infty}(0,\infty;W^{1,q})} \leqslant C \varepsilon.
\end{align*}
Next notice  that,
\[ \partial_{t} \nabla X(t,y) = Q(t) \nabla\widetilde u(t,y).\]
Therefore
\begin{align*}
\|\partial_{t} \nabla X\|_{L^{p}(0,\infty;W^{1,q})} \leqslant
\|e^{\eta (\cdot)} \nabla \widetilde u\|_{L^{p}(0,\infty;W^{1,q})} \leqslant C\varepsilon.
\end{align*}
We also have
\[\partial_{t} (Z - I_{3}) = \partial_{t} Z = -\nabla X^{-1} \left( \partial_{t} \nabla X\right) \nabla X^{-1}  = - Z \left( \partial_{t} \nabla X\right) Z.\]
Thus
\begin{align*}
\|\partial_{t} (Z - I_{3})\|_{L^{p}(0,\infty;W^{1,q})} \leqslant \|Z\|^{2}_{L^{\infty}(0,\infty;W^{1,q})}  \|\partial_{t} \nabla X\|_{L^{p}(0,\infty;W^{1,q})} \leqslant C \varepsilon.
\end{align*}
Finally, by H\"older's  inequality we have
\begin{align*}
\left\|(Z -I_{3})(t_{2},\cdot) - (Z -I_{3})(t_{1},\cdot)\right\|_{W^{1,q}} &\leqslant  \int_{t_{1}}^{t_{2}} \|\partial_{t} (Z - I_{3})(s,\cdot)\|_{W^{1,q}}  \\
&\leqslant |t_{1} - t_{2}|^{1/p'} \|\partial_{t} (Z - I_{3})\|_{L^{p}(0,\infty;W^{1,q})}  \leqslant C\varepsilon |t_{1} - t_{2}|^{1/p'}.
\end{align*}
This completes the proof of the lemma.
\end{proof}

Now we are in a position to estimate the nonlinear terms in \eqref{F1-g} - \eqref{G-g}. More precisely, we prove the following
\begin{prop} \label{prop:NL-g}
Let $\varepsilon_{0}$ be the constant defined as in \eqref{e0}. Then there exists a constant $C_{N}$ depending only on $p,q,\eta,\ofo$ such that   for every $\varepsilon \in (0,\varepsilon_{0})$ and for every  $(\widetilde \rho, \widetilde u,\widetilde \vartheta,\widetilde \ell,\widetilde \omega) \in \widetilde{\mathcal S}_{\varepsilon},$ we have
\begin{align*}
&\left\|(e^{\eta (\cdot)} \mathcal{F}_{1}, e^{\eta (\cdot)} \mathcal{F}_{2}, e^{\eta (\cdot)} \mathcal{F}_{3}, e^{\eta (\cdot)} \mathcal{H} ,  e^{\eta (\cdot)} \mathcal{G}_{1}, e^{\eta (\cdot)} \mathcal{G}_{2}) \right\|_{\mathcal{B}_{\infty,p,q}} \\
& \qquad \qquad + \Big\|(\mathcal{F}_{1,avg}, \mathcal{F}_{3,avg}, \mathcal{H}_{avg} ) \Big\|_{L^{1}(0,\infty)^{3}} \leqslant C_{N} \varepsilon^{2}.
\end{align*}
\end{prop}
\begin{proof}
The constants $C$ appearing in this proof depend only on $p,q,\eta, \ofo$ and are independent of $\varepsilon.$

For every  $(\widetilde \rho, \widetilde u,\widetilde \vartheta,\widetilde \ell,\widetilde \omega) \in \widetilde{\mathcal S}_{\varepsilon},$ we have
\begin{align} \label{d1}
\|\widetilde \rho\|_{L^{\infty}(0,\infty;W^{1,q})} &\leqslant \| e^{\eta (\cdot)}\widetilde \rho_{m}\|_{L^{\infty}(0,\infty;W^{1,q})} + \|\widetilde \rho_{avg}\|_{L^{\infty}(0,\infty)} \notag \\
& \leqslant C \| e^{\eta (\cdot)}\widetilde \rho_{m}\|_{W^{1,p}(0,\infty;W^{1,q})} + \|\widetilde \rho_{avg}\|_{L^{\infty}(0,\infty)} \leqslant C \varepsilon.
\end{align}
Since $2 < p < \infty,$ using the the following continuous embedding
\begin{align*}
W^{2,1}_{q,p}(Q_{\infty}^{F}) \hookrightarrow L^{\infty}(0,\infty;B^{2(1-1/p)}_{q,p}(\ofo)) \hookrightarrow  L^{\infty}(0,\infty;W^{1,q}(\ofo)),
\end{align*}
we can similarly deduce that
\begin{equation} \label{d2}
\|\widetilde u\|_{L^{\infty}(0,\infty;W^{1,q})^{3}} + \|\widetilde \vartheta\|_{L^{\infty}(0,\infty;W^{1,q})} \leqslant C \varepsilon.
\end{equation}
Using the fact that $L^{\infty}(\ofo) \hookrightarrow W^{1,q}(\ofo)$ for $q > 3$ it is easy to see that
\begin{align} \label{d3}
\|\widetilde u\|_{L^{p}(0,\infty;L^{\infty}(\ofo))} + \|\nabla \widetilde u\|_{L^{p}(0,\infty;L^{\infty}(\ofo))} + \|\nabla \widetilde \vartheta\|_{L^{p}(0,\infty;L^{\infty}(\ofo))} \leqslant C \varepsilon.
\end{align}

Let $Q$ be defined as in \eqref{def-Q-g}. Then
\begin{align} \label{d4}
\|Q - I_{3}\|_{L^{\infty}(0,\infty;\mathbb{R}^{3\times 3})} &\leqslant \|Q\|_{L^{\infty}(0,\infty;\mathbb{R}^{3\times 3})} \|e^{\eta (\cdot)} \widetilde \omega\|_{L^{\infty}(0,\infty;\rt)}  \int_{0}^{\infty} e^{- \eta s} ds \notag \\
& \leqslant \frac{C}{\eta} \|e^{\eta (\cdot)} \widetilde \omega\|_{W^{1,p}(0,\infty;\rt)}  \leqslant C \varepsilon.
\end{align}
Now we estimate the nonlinear terms in \eqref{F1-g} - \eqref{G-g}.
\paragraph{\underline{Estimates of  $\mathcal{F}_{1}$ and $\mathcal{F}_{1,avg}$}}
\begin{align} \label{d-f1}
\|e^{\eta(\cdot)} \mathcal{F}_{1}\|_{L^{p}(0,\infty;W^{1,q})} + \|\mathcal{F}_{1,avg}\|_{L^{1}(0,\infty)} \leqslant C\varepsilon^{2}.
\end{align}
Let us recall
\begin{equation*}
e^{\eta t}\mathcal{F}_{1}(\widetilde \rho,\widetilde  u,\widetilde  \vartheta,\widetilde  \ell,\widetilde \omega)
=   - e^{\eta t}(\widetilde \rho+\overline\rho) (Z^{\top} - I_{3}) : \nabla \widetilde u  - e^{\eta t}\widetilde \rho \div\widetilde u.
\end{equation*}
Using the \cref{lem:X-O-G} and estimates \eqref{d1}-\eqref{d3}, we obtain
\begin{align*}
&\left\| - e^{\eta t}(\widetilde \rho+\overline\rho) (Z^{\top} - I_{3}) : \nabla \widetilde u  - e^{\eta t}\widetilde \rho \div\widetilde u \right\|_{L^{p}(0,\infty;W^{1,q})} \\
& \leqslant \left\|(\widetilde \rho+\overline\rho) \right\|_{L^{\infty}(0,\infty;W^{1,q})} \left\|Z^{\top} - I_{3} \right\|_{L^{\infty}(0,\infty;W^{1,q})}  \left\| e^{\eta (\cdot)} \nabla \widetilde u\right\|_{L^{p}(0,\infty;W^{1,q})} \\
& \qquad \qquad \qquad + \|\widetilde \rho\|_{L^{\infty}(0,\infty;W^{1,q})} \left\| e^{\eta (\cdot)} \div\widetilde u\right\|_{L^{p}(0,\infty;W^{1,q})} \\
& \leqslant C \varepsilon^{2}.
\end{align*}
We have
\begin{align*}
\mathcal{F}_{1,avg} =  -\frac{1}{|\ofo|} \int_{\ofo} (\widetilde \rho+\overline\rho) (Z^{\top} - I_{3}) : \nabla \widetilde u \ dy  - \frac{1}{|\ofo|}\int_{\ofo} \widetilde \rho \div\widetilde u \ dy.
\end{align*}
We estimate the first term of $\mathcal{F}_{1,avg} $ as follows
\begin{align*}
&\left\| \int_{\ofo} (\widetilde \rho+\overline\rho) (Z^{\top} - I_{3}) : \nabla \widetilde u \ dy\right\|_{L^{1}(0,\infty)} \\
& \leqslant \|\widetilde \rho+\overline\rho\|_{L^{\infty}((0,\infty) \times \ofo)} \|(Z^{\top} - I_{3})\|_{L^{\infty}((0,\infty) \times \ofo)} \int_{0}^{\infty} \int_{\ofo} |\nabla \widetilde u | \ dy dt  \\
& \leqslant C \|\widetilde \rho+\overline\rho\|_{L^{\infty}(0,\infty;W^{1,q})} \left\|Z^{\top} - I_{3} \right\|_{L^{\infty}(0,\infty;W^{1,q})} \int_{0}^{\infty} e^{-\eta t} e^{\eta t} \|\nabla \widetilde u(t,\cdot)\|_{L^{q}(\ofo)} \ dt \\
&\leqslant C \varepsilon \left(\int_{0}^{\infty} e^{-p'\eta t} \ dt \right)^{1/p'} \left\| e^{\eta(\cdot)} \nabla \widetilde u\right\|_{L^{p}(0,\infty;L^{q}(\ofo))} \leqslant C\varepsilon^{2}.
\end{align*}
The other estimate  can be obtained similarly.
\paragraph{\underline{Estimates of  $\mathcal{F}_{2},$ $\mathcal{F}_{3}$ and $\mathcal{F}_{3,avg}$}}
\begin{align} \label{d-f2}
\|e^{\eta(\cdot)} \mathcal{F}_{2}\|_{L^{p}(0,\infty;L^{q})}+ \|e^{\eta(\cdot)} \mathcal{F}_{3}\|_{L^{p}(0,\infty;L^{q})} + \|\mathcal{F}_{3,avg}\|_{L^{1}(0,\infty)} \leqslant C\varepsilon^{2}.
\end{align}
The proof is similar to the proof of \eqref{d-f1}. Note that the terms of $\mathcal F_{2}$ and $\mathcal{F}_{3}$ are at least quadratic functions of $\widetilde \rho$, $\widetilde u$, $\widetilde \vartheta$, $Z^{\top} - I_{3}$ and $Q-I_{3}$. Therefore using \cref{lem:X-O-G} and estimates \eqref{d1} - \eqref{d4}, we obtain \eqref{d-f2}.

\paragraph{\underline{Estimates of  $\mathcal{H}_{F}$  and $\mathcal{H}_{S}$}}
\begin{multline} \label{d-h}
\|e^{\eta(\cdot)} \mathcal{H}_{F}\cdot n \|_{F^{(1-1/q)/2}_{p,q}(0,\infty;L^{q}(\partial \Omega)) \cap L^{p}(0,\infty;W^{1-1/q,q}(\partial \Omega))} + \|\mathcal{H}_{F,avg}\|_{L^{1}(0,\infty)} \\
+\|e^{\eta(\cdot)} \mathcal{H}_{S} \cdot n \|_{F^{(1-1/q)/2}_{p,q}(0,\infty;L^{q}(\partial \oso)) \cap L^{p}(0,\infty;W^{1-1/q,q}(\partial \oso)}  + \|\mathcal{H}_{S,avg}\|_{L^{1}(0,\infty)}
\leqslant C\varepsilon^{2}.
\end{multline}
Recall that
\[ e^{\eta t}\mathcal{H}_{F} = e^{\eta t}(I_3- Z^{\top}) \nabla \widetilde  \vartheta.\]
Using \cref{lem:X-O-G} and estimates \eqref{d2}-\eqref{d3}, we first obtain
\begin{align*}
&\left\|e^{\eta(\cdot)}(I_3- Z^{\top}) \nabla \widetilde  \vartheta\right\|_{L^{p}(0,\infty;W^{1-1/q,q}(\partial \Omega))} \\
&\leqslant C \left\|(I_3- Z^{\top}) e^{\eta(\cdot)} \nabla \widetilde  \vartheta\right\|_{L^{p}(0,\infty;W^{1,q}(\ofo))} \\
& \leqslant C \left\|(I_3- Z^{\top}) \right\|_{L^{\infty}(0,\infty;W^{1,q}(\ofo))} \left\|e^{\eta(\cdot)} \nabla \widetilde  \vartheta\right\|_{L^{p}(0,\infty;W^{1,q}(\ofo))} \leqslant C \varepsilon^{2}.
\end{align*}
We write
\begin{align*}
e^{\eta t}\mathcal{H}_{F}|_{\partial \Omega} \cdot n = \sum_{j,k} \left[\left(\delta_{j,k} - Z_{j,k}\right) e^{\eta t}\frac{\partial \widetilde \vartheta}{\partial y_{k}}\right] (t,y) n_{j}(y), \quad y \in \partial\Omega.
\end{align*}
We know $e^{\eta t} \widetilde \vartheta \in W^{2,1}_{q,p}(Q_{F}^{\infty})$. Thus by \cite[Proposition 6.4]{DenkHieberPruss07}
\begin{align*}
\left\| e^{\eta t}\frac{\partial \widetilde \vartheta}{\partial y_{k}}\Big|_{\partial \Omega}\right\|_{F^{(1-1/q)/2}_{p,q}(0,\infty;L^{q}(\partial \Omega)}  \leqslant C \|e^{\eta(\cdot)} \widetilde \vartheta \|_{W^{2,1}_{q,p}(Q_{F}^{\infty})} \leqslant C \varepsilon, \mbox{ for } k=1,2,3.
\end{align*}
Also from \cref{lem:X-O-G} we have
\begin{align*}
\left\| \left(\delta_{j,k} - Z_{j,k}\right) |_{\partial\Omega}\right\|_{C^{1/p'}([0,\infty);W^{1-1/q,q}(\partial \Omega))} \leqslant \left\|Z^{\top} - I_{3}\right\|_{C^{1/p'}([0,\infty);W^{1,q}(\ofo))} \leqslant C\varepsilon.
\end{align*}
Therefore by \cite[Theorem 2.8.2(ii)]{Triebel}, we obtain
\begin{align*}
&\left\| e^{\eta(\cdot)} \mathcal{H}_{F}\cdot n \right\|_{F^{(1-1/q)/2}_{p,q}(0,\infty;L^{q}(\partial \Omega))}  \\
& \leqslant C \sum_{j,k} \left\| e^{\eta t}\frac{\partial \widetilde \vartheta}{\partial y_{k}}\Big|_{\partial \Omega}\right\|_{F^{(1-1/q)/2}_{p,q}(0,\infty;L^{q}(\partial \Omega)} \left\| \left(\delta_{j,k} - Z_{j,k}\right) |_{\partial\Omega}\right\|_{C^{1/p'}([0,\infty);W^{1-1/q,q}(\partial \Omega))}  \\
&\leqslant C \varepsilon^{2}.
\end{align*}
The other estimates can be obtained similarly.
\paragraph{\underline{Estimates of  $\mathcal{G}_{1}$  and $\mathcal{G}_{2}$}}
\begin{align} \label{d-g}
\left\| e^{\eta(\cdot)}\mathcal{G}_{1}\right\|_{L^{p}(0,\infty)} + \left\| e^{\eta(\cdot)}\mathcal{G}_{2} \right\|_{L^{p}(0,\infty)} \leqslant C \varepsilon^{2}.
\end{align}
The proof is easy and left to the reader.
\end{proof}

\begin{prop} \label{prop:lip-g}
Let $\varepsilon_{0}$ be the constant defined as in \eqref{e0}.
 Let us set
\begin{align*}
& \mathcal{F}_{1}^{j} =  \mathcal{F}_{1}(\widetilde\rho^{j},\widetilde u^{j},\widetilde \vartheta^{j},\widetilde \ell^{j},\widetilde \omega^{j}), \quad \mathcal{F}_{2}^{j} =  \mathcal{F}_{2} (\widetilde\rho^{j},\widetilde u^{j},\widetilde \vartheta^{j},\widetilde \ell^{j},\widetilde \omega^{j}), \quad \mathcal{F}_{3}^{j} =  \mathcal{F}_{3}(\widetilde\rho^{j},\widetilde u^{j},\widetilde \vartheta^{j},\widetilde \ell^{j},\widetilde \omega^{j})  \\
& \mathcal{H}_{F}^{j} =  \mathcal{H}_{F}(\widetilde\rho^{j},\widetilde u^{j},\widetilde \vartheta^{j},\widetilde \ell^{j},\widetilde \omega^{j}), \quad \mathcal{H}_{S}^{j} =  \mathcal{H}_{S}(\widetilde\rho^{j},\widetilde u^{j},\widetilde \vartheta^{j},\widetilde \ell^{j},\widetilde \omega^{j}),  \mathcal{H}^{j} = \; \mathbbm{1}_{\partial \Omega} \mathcal{H}_{F}^{j} + \mathbbm{1}_{\partial \oso} \mathcal{H}_{S}^{j}  \\
& \mathcal{G}_{1}^{j} =  \mathcal{G}_{1}(\widetilde\rho^{j},\widetilde u^{j},\widetilde \vartheta^{j},\widetilde \ell^{j},\widetilde \omega^{j}), \quad \mathcal{G}_{2}^{j} =  \mathcal{G}_{2}(\widetilde\rho^{j},\widetilde u^{j},\widetilde \vartheta^{j},\widetilde \ell^{j},\widetilde \omega^{j}),
\end{align*}
Then there exists a constant $C_{lip} > 0$ depending only on $p,q,\eta,\ofo$ such that  for every for every $\varepsilon \in (0,\varepsilon_{0})$ and for all $(\widetilde \rho^{1}, \widetilde u^{1},\widetilde \vartheta^{1},\widetilde \ell^{1},\widetilde \omega^{1}) \in \widetilde{\mathcal S}_{\varepsilon}$ and $(\widetilde \rho^{2}, \widetilde u^{2},\widetilde \vartheta^{2},\widetilde \ell^{2},\widetilde \omega^{2}) \in \widetilde{\mathcal S}_{\varepsilon}$, we have
\begin{multline}
\left\|e^{\eta (\cdot)} \Big(\mathcal{F}_{1} - {\mathcal{F}}_{1}^{2},  \mathcal{F}_{2}^{1} - \mathcal{F}_{2}^{2},  \mathcal{F}_{3}^{1} - \mathcal{F}_{3}^{2},  \mathcal{H}^{1} - \mathcal{H}^{2} ,   \mathcal{G}_{1}^{1} - \mathcal{G}_{2}^{2}, \mathcal{G}_{2}^{1} - \mathcal{G}_{2}^{2}\Big) \right\|_{\mathcal{B}_{\infty,p,q}}  \\
 + \left\| \left(\mathcal{F}_{1,avg}^{1}- {\mathcal{F}}_{1,avg}^{2}, \mathcal{F}_{3,avg}^{1}- {\mathcal{F}}_{3,avg}^{2}, \mathcal{H}_{avg}^{1}- {\mathcal{H}}_{avg}^{2}\right)\right\|_{L^{1}(0,\infty)^{3}} \\
 \leqslant C_{lip} \varepsilon \left\|(\widetilde \rho^{1}, \widetilde u^{1},\widetilde \vartheta^{1},\widetilde \ell^{1},\widetilde \omega^{1}) -  (\widetilde \rho^{2}, \widetilde u^{2},\widetilde \vartheta^{2},\widetilde \ell^{2},\widetilde \omega^{2})\right\|_{\mathcal S}.
\end{multline}
\end{prop}
\begin{proof}
The proof of this proposition is similar to the proof of \cref{prop:NL-g}
\end{proof}


\section{Proof of the Global Existence Theorem} \label{sec:global-existence}

This section is devoted to the proof of \cref{mainthm_glob} and \cref{cor_density}. First we prove a global existence theorem for \eqref{sys:NL-G} - \eqref{G0-g}. More precisely, we prove the following theorem, which implies  \cref{thm:g-f-2}.

\begin{thm}  \label{thm:g-f}
Let $2 < p < \infty$ and $3 < q < \infty$ satisfying the condition $\displaystyle \frac{1}{p} + \frac{1}{2q} \neq \frac12$. Assume that \eqref{nocontact} is satisfied. Let $\overline \rho > 0$ and $\overline \vartheta > 0$ be two given constants and  $\eta \in (0, \eta_{0})$, where $\eta_{0}$ is the constant introduced in  \cref{thm:stab-afs}.
 Then there exists a constant  $\widetilde\varepsilon_{0} > 0$ such that, for all $\varepsilon \in (0,\widetilde\varepsilon_{0})$ and  for any $(\rho_{0},u_{0}, \vartheta_{0},\ell_{0},\omega_{0})$ belongs to $\mathcal{I}^{cc}_{p,q}$ satisfying
 \begin{align*}
\frac{1}{|\ofo|} \int_{\ofo} \rho_{0} \ {\rm d}x = \overline\rho,
\end{align*}
and
\begin{align*}
\|(\rho_{0} - \overline\rho, u_{0}, \vartheta_{0} - \overline \vartheta, \ell_{0},\omega_{0})\|_{\mathcal{I}_{p,q}} \leqslant \frac{\varepsilon}{2C_{L}},
\end{align*}
where $C_{L}$ is the continuity constant appearing in \cref{thm:nh-g},
the system \eqref{sys:NL-G} - \eqref{G0-g} admits a unique solution
$(\widetilde \rho, \widetilde u,\widetilde \vartheta,\widetilde \ell,\widetilde \omega)$ with
\begin{gather*}
\left\|(e^{\eta t}\widetilde\rho_{m}, e^{\eta t} \widetilde u,e^{\eta t} \widetilde\vartheta_{m}, e^{\eta t} \widetilde\ell, e^{\eta t}\widetilde\omega)\right\|_{{\mathcal S}_{\infty,p,q}} +  \|\widetilde \rho_{avg}, \widetilde \vartheta_{avg}\|_{L^{\infty}(0,\infty)} \leqslant \varepsilon.
\end{gather*}
Moreover, $X \in L^{\infty}(0,\infty;W^{2,q}(\ofo))^{3} \cap W^{1,\infty}(0,\infty;W^{1,q}(\ofo))$ and 
$ X(t,\cdot) : \ofo\to \oft $ is a  $C^1$-diffeormorphim for all  $t\in [0,\infty).$
\end{thm}

\begin{proof}
Let us set
\begin{align}
\widetilde \varepsilon_{0} = \min\left\{\varepsilon_{0}, \frac{1}{2C_{L}C_{N}}, \frac{1}{2C_{L}C_{lip}} \right\},
\end{align}
where $\varepsilon_{0}$ is defined as in \eqref{e0} and $C_{L},$ $C_{N}$ and $C_{lip}$ are the constants appearing in \cref{thm:nh-g}, \cref{prop:NL-g} and \cref{prop:lip-g} respectively. Let us choose $\varepsilon \in (0, \widetilde\varepsilon_{0})$ and  $(\sigma, v, \varphi, k , \tau) \in \widetilde{\mathcal S}_{\varepsilon}$, where $\widetilde{\mathcal S}_{\varepsilon}$ is defined as in \eqref{ball-g}. We consider the following problem
\begin{align}  \label{eq:fp-g}
&\partial_{t} \widetilde\rho + \overline \rho \div\widetilde u = {\mathcal F}_{1}(\sigma, v, \varphi, k , \tau),  \mbox{ in } (0,\infty) \times \ofo, \notag \\
& \partial_{t} \widetilde u - \div\sigma_{l}(\widetilde \rho,\widetilde u,\widetilde\vartheta) = {\mathcal F}_{2}(\sigma, v, \varphi, k , \tau),  \mbox{ in } (0,\infty) \times \ofo, \notag \\
& \partial_{t} \widetilde \vartheta  - \frac{\kappa}{\overline\rho c_{v}} \Delta \widetilde\vartheta + \frac{R \overline \vartheta}{c_{v}} \div\widetilde u = {\mathcal F}_{3}(\sigma, v, \varphi, k , \tau), \mbox{ in } (0,\infty) \times \ofo, \notag \\
&\widetilde  u = 0 \mbox{ on } (0,\infty) \times \partial\Omega, \quad \widetilde u = \widetilde\ell + \widetilde\omega \times y \mbox{ on } (0,\infty) \times \partial \oso \\
& \frac{\partial \widetilde \vartheta}{\partial n} = {\mathcal H}_{F} (\sigma, v, \varphi, k , \tau)\cdot n \mbox{ on } (0,\infty) \times \partial \Omega, \quad \frac{\partial \widetilde \vartheta}{\partial n} = {\mathcal H}_{S} (\sigma, v, \varphi, k , \tau)\cdot n \mbox{ on } (0,\infty) \times \partial \oso, \notag \\
& \frac{d}{dt} \widetilde \ell = - m^{-1} \int_{\partial\oso} \sigma_{l} (\widetilde\rho,\widetilde u,\widetilde\vartheta) n \ d\gamma + {\mathcal G}_{1}(\sigma, v, \varphi, k , \tau), \quad t \in (0,\infty) \notag \\
&\frac{d}{dt} \omega = - J(0)^{-1} \int_{\partial\oso} y \times \sigma_{l} (\rho,u,\vartheta) n \ d\gamma  + {\mathcal G}_{2}(\sigma, v, \varphi, k , \tau), \quad t \in (0,\infty)  \notag \\
&   \widetilde\rho(0) = \rho_{0} - \overline\rho , \quad u(0) = u_{0},  \quad  \vartheta(0) = \vartheta_{0} - \overline \vartheta , \quad \mbox{ in } \ofo, \notag \\
&   \ell(0) = \ell_{0},  \quad \omega(0) = \omega_{0}.\notag
\end{align}
We are going to show, the mapping
\begin{align*}
\mathcal{N} : (\sigma, v, \varphi, k , \tau) \mapsto (\widetilde \rho, \widetilde u,\widetilde \vartheta,\widetilde \ell,\widetilde \omega)
\end{align*}
where  $(\widetilde \rho, \widetilde u,\widetilde \vartheta,\widetilde \ell,\widetilde \omega)$ is the solution to the system \eqref{eq:fp-g}, is a contraction in $\widetilde{\mathcal S}_{\varepsilon}$. Since $(\sigma, v, \varphi, k , \tau) \in \widetilde{\mathcal S}_{\varepsilon}$, we can  apply
\cref{thm:nh-g} and \cref{prop:NL-g} to the system  \eqref{eq:fp-g} and using \eqref{eq:ini-g} we obtain
\begin{align*}
\left\| (\widetilde \rho, \widetilde u,\widetilde \vartheta,\widetilde \ell,\widetilde \omega)\right\|_{\mathcal{S}} &\leqslant C_{L} \left\| (\rho_{0} - \overline\rho,u_{0}, \vartheta_{0} - \overline \vartheta,\ell_{0},\omega_{0})\right\|_{\mathcal{I}_{p,q}} + C_{L}C_{N} \varepsilon^{2} \leqslant \varepsilon.
\end{align*}
Thus $\mathcal{N}$ is a mapping from $\widetilde{\mathcal S}_{\varepsilon}$ to itself for all $\varepsilon \in (0, \widetilde\varepsilon_{0}).$

 Let $(\sigma^{1}, v^{1}, \varphi^{1}, k^{1} , \tau^{1})$ and $(\sigma^{2}, v^{2}, \varphi^{2}, k^{2} , \tau^{2})$ belong to $\widetilde{\mathcal S}_{\varepsilon}$. For $j=1,2$, we set \linebreak $\mathcal{N}(\sigma^{j}, v^{j}, \varphi^{j}, k^{j} , \tau^{j}) = (\widetilde \rho^{j}, \widetilde u^{j},\widetilde \vartheta^{j},\widetilde \ell^{j},\widetilde \omega^{j})$. Using \cref{thm:nh-g} and \cref{prop:lip-g}, we obtain
\begin{align}
&\left\| (\widetilde \rho^{1}, \widetilde u^{1},\widetilde \vartheta^{1},\widetilde \ell^{1},\widetilde \omega^{1})- (\widetilde \rho^{2}, \widetilde u^{2},\widetilde \vartheta^{2},\widetilde \ell^{2},\widetilde \omega^{2})\right\|_{\mathcal{S}} \notag \\
& \leqslant C_{L} C_{lip} \varepsilon \left\| (\sigma^{1}, v^{1}, \varphi^{1}, k^{1} , \tau^{1})- (\sigma^{2}, v^{2}, \varphi^{2}, k^{2} , \tau^{2})\right\|_{\mathcal{S}}
\end{align}
Using the definition of $\widetilde \varepsilon_{0}$ one can easily check that the mapping  $\mathcal{N}$ is a contraction in  $\widetilde{\mathcal S}_{\varepsilon}$.  This completes  the proof of the theorem.
\end{proof}

\begin{proof}[Proof of  \cref{mainthm_glob}]
Let $(\widetilde \rho, \widetilde u,\widetilde \vartheta,\widetilde \ell,\widetilde \omega)$ be the solution of \eqref{sys:NL-G} - \eqref{G0-g} constructed in \cref{thm:g-f}. Since   $X(t,\cdot)$ is $C^{1}-$ diffeomorphism from $\ofo$ into $\oft,$ we set $Y(t,\cdot) = X^{-1}(t,\cdot)$ and for $x \in \oft, \ t \geq 0$
\begin{align}
&\rho(t,x) = \widetilde \rho(t,Y(t,x)) + \overline\rho, \quad u(t,x) = Q(t)\widetilde u(t,Y(t,x)),  \quad \vartheta(t,x) = \widetilde \vartheta(t,Y(t,x)) + \overline \vartheta,  \notag \\
&\qquad \dot a(t) = Q(t) \widetilde \ell(t) ,  \quad  \omega(t) = Q(t)\widetilde \omega(t).
\end{align}
We can easily check that $(\rho,u,\vartheta,a,\omega)$ solves the original system \eqref{fluid-eq} - \eqref{inicond} satisfying the estimate \eqref{est:glob}. By choosing $\delta_{0}$ sufficiently small, from
\eqref{est:glob} it is easy to see that $\rho(t,x) \geqslant \ds \frac{\bar\rho}{2}$ for all $(t,x) \in (0,\infty) \times \oft.$  Finally from \eqref{def:ost} and \eqref{est:glob}, we obtain 
\begin{equation*}
\mathrm{dist}(\ost,\oso) \leqslant \|a(t)\|_{\rt} + \|Q(t) - I_{3}\|_{\mathbb{R}^{3\times 3}} |y| < \frac{\nu}{2} \quad \mbox{ for all } t \geqslant 0,
\end{equation*}
for sufficiently small $\delta_{0}.$ Therefore,  $\mathrm{dist}(\Omega_{S}(t),\partial\Omega) \geqslant \nu/2$ for all $t \in [0,\infty)$.
\end{proof}

\begin{proof}[Proof of \cref{cor_density}]
The first estimate obviously follows from \eqref{est:glob}.
To prove \eqref{den} we integrate the density equation \eqref{fluid-eq} over $\oft$ and using the boundary conditions and \eqref{eq:br} we obtain
\begin{equation*}
\frac{1}{|\oft|} \int_{\oft} \rho(t,x)  \ dx = \overline \rho \qquad\qquad(t\geqslant 0).
\end{equation*}
Since $\text{dist}(\ost, \partial\Omega) > \nu/2$ and $\ost$ has smooth boundary for every $t\in [0,\infty),$ by Poincar\'e-Wirtinger inequality we obtain
\begin{equation}\label{poincare-wirtinger}
\|\rho(t,x) - \overline\rho\|_{L^{q}(\oft)} \leqslant C \|\nabla \rho\|_{L^{q}(\oft)},
\end{equation}
where the constant $C$   can be chosen uniformly with respect to $t$ (see for instance \cite[Theorem 1]{BC07}). Thus we have
\begin{align*}
\|e^{\eta (\cdot)} (\rho - \overline\rho) \|_{W^{1,p}(0,\infty; W^{1,q}(\Omega_{F}(\cdot)))} \leqslant C \|e^{\eta (\cdot)}\nabla \rho\|_{W^{1,p}(0,\infty; L^{q}(\Omega_{F}(\cdot)))} \leqslant C \delta.
\end{align*}
and consequently \eqref{den} follows.
\end{proof}

\end{document}